\documentclass[10pt, reqno]{amsart}
\usepackage{amsmath}
\usepackage{amssymb}
\usepackage{amsthm}
\usepackage{epsfig}
\usepackage{enumerate}
\usepackage{amsbsy}
\usepackage{amsfonts}
\usepackage{color}

\newtheorem{thm}{Theorem}[section]
\newtheorem{prop}[thm]{Proposition}
\newtheorem{lm}[thm]{Lemma}

\newtheorem{cor}[thm]{Corollary}
\newtheorem{dfn}[thm]{Definition}

\theoremstyle{remark}
\newtheorem{rk}{Remark}

\numberwithin{equation}{section}

\newcommand{\sa}{{e^{-it|\nabla|^\al}}}
\newcommand{\sap}{{e^{-i(t-s)|\nabla|^\al}}}
\newcommand{\na}{|\nabla|^\al}
\newcommand{\al}{\alpha}
\newcommand{\ep}{\varepsilon}
\newcommand{\bxi}{\langle\xi\rangle}

\newcommand{\xil}{\xi_\tau}
\newcommand{\bt}{{\langle t \rangle}}

\begin{document}
\title[Modified scattering of fractional Schr\"odinger equations]
      {On the modified scattering of $3$-d Hartree type fractional Schr\"odinger equations with Coulomb potential}

\author[Y. Cho]{Yonggeun Cho}
\address{Department of Mathematics, and Institute of Pure and Applied Mathematics, Chonbuk National University, Jeonju 561-756, Republic of Korea}
\email{changocho@jbnu.ac.kr}

\author[G. Hwang]{Gyeongha Hwang}
\address{National Center for Theoretical Sciences, No. 1 Sec. 4 Roosevelt Rd., National Taiwan University, Taipei, 106, Taiwan}
\email{ghhwang@ncts.ntu.edu.tw}

\author[C. Yang]{Changhun Yang}
\address{Department of Mathematical Sciences, Seoul National University, Seoul 151-747, Republic of Korea}
\email{maticionych@snu.ac.kr}

\begin{abstract}
	In this paper we study 3-d Hartree type fractional 
	Schr\"odin-ger equations:
	$$
	i\partial_{t}u-\vert\nabla\vert^{\alpha}u
	= \lambda\left(
	|x|^{-\gamma} *\vert u\vert^{2}
	\right)u,\;\;1 < \al < 2,\;\;0 < \gamma < 3,\;\; \lambda \in \mathbb R \setminus \{0\}.
	$$
	In \cite{cho}  it is known that no scattering occurs in $L^2$ for the long range ($0 < \gamma \le 1$). In \cite{c0, chooz2, cho1} the short-range scattering ($1 < \gamma < 3$) was treated for the scattering in $H^s$. In this paper we consider the critical case ($\gamma = 1$) and
	prove a modified scattering in $L^\infty$ on the frequency to the Cauchy problem with small initial data. For this purpose we investigate the global behavior of $x e^{it\na} u$, $x^2 e^{it\na} u$ and $\bxi^5 \widehat{e^{it\na} u}$. Due to the non-smoothness of $\na$ near zero frequency the range of $\al$ is restricted to $(\frac{17}{10}, 2)$.
\end{abstract}

\maketitle

	\section{Introduction}
	The fractional Schr\"odinger equations have been derived to describe natural phenomena in the context of fractional quantum mechanics \cite{la1,la2}, system of bosons \cite{fl}, system of long-range lattice interaction \cite{kls}, water waves \cite{ss}, turbulence \cite{cmmt} and so on.
	In this paper we consider the Hartree type fractional Schr\"odinger equation with the Coulomb potential. Heuristically, Hartree nonlinearity can be interpreted as an interaction between particles or waves with potential $V$ \cite{fl}. We are concerned with the following Cauchy problem:
	\begin{equation}\label{equation}
	\begin{cases}
	i\partial_{t}u-\vert\nabla\vert^{\alpha}u
	= \left(
	V *\vert u\vert^{2}
	\right)u \;\;\mbox{in}\;\; {\mathbb R^{1+3}}, \\
	u(0,x) = u_{0}(x),	
	\end{cases}
	\end{equation}
	where $ u:(t,x)\in\mathbb{R}\times\mathbb{R}^{3}\rightarrow\mathbb{C},
	1<\alpha<2$, and $V = \lambda |x|^{-1}$ for $\lambda\in\mathbb{R}\setminus \{0\}$.	
	Here $\na = (-\Delta)^\frac\al2 = \mathcal F^{-1}|\xi|^\al \mathcal F$ is the fractional derivative of order $\al$ and $*$ denotes the space convolution.
	
	By Duhamel's formula, \eqref{equation} is written as an
	integral equation
	\begin{equation}\label{integral}
	u(t) = \sa u_0  - i\int_0^t \sap(V *|u(s)|^2)u(s)\,ds,
	\end{equation}
	where the linear propagator $\sa v_0$ to be the solution to the linear problem $i\partial_t v =  |\nabla|^\al v$ with initial datum $v_0$.
	Then it is at least formally given  by
	\begin{align}\label{int eqn}
	\sa v_0 = \mathcal F^{-1} e^{-it|\xi|^\al} \mathcal Fv_0 = (2\pi)^{-3}\int_{\mathbb{R}^3} e^{i( x\cdot \xi - t|\xi|^\al )}\widehat{v_0}(\xi)\,d\xi.
	\end{align}
	Here $\widehat{v} = \mathcal F v$ is the Fourier transform of $v$  such that $\widehat{v}(\xi) = \int_{\mathbb{R}^3} e^{- ix\cdot \xi} v(x)\,dx$ and we denote its inverse Fourier transform by  $\mathcal F^{-1} w$ defined as $\mathcal F^{-1} w(x)=(2\pi)^{-3}\int_{\mathbb{R}^3} e^{ix\cdot \xi} w(\xi)\,d\xi$. We list basic notations at the end of this section.
	
	We can also describe the solution of \eqref{equation} in the frequency space. To do so let us define
	\begin{equation}\label{def:f}
	v(t,x) := e^{it\vert\nabla\vert^{\alpha}}u(x).
	\end{equation}
	By Duhamel's formula \eqref{int eqn} is written as
	\begin{equation}\label{integraleqn}
	\widehat{v} = \widehat{u}_{0}(\xi)
	+\int_{0}^{t}I(s,\xi)\,ds,
	\end{equation}
	where
	\begin{equation}\label{integraleqn}
	I(s, \xi)
	=
	ic_0
	\int_{\mathbb{R}^{3}}
	e^{is\phi_{\alpha}(\xi,\eta)}
	\vert\eta\vert^{-2}
	\widehat{\vert u\vert^{2}}(s,\eta)
	\widehat{v}(s,\xi-\eta)\,d\eta ds.
	\end{equation}
	Here $c_0 = -2(2\pi)^{-2}\lambda$ (for this we used $\widehat{|x|^{-1}} = 4\pi |\eta|^{-2}$) and the phase function $\phi_\al$ is defined by $$\phi_{\alpha}(\xi,\eta) = \vert\xi\vert^{\alpha}-\vert\xi-\eta\vert^{\alpha}.$$
	The formula of $I$ can be rewritten as
	\begin{align}\label{Duhamel}
	I(s,\xi) :=
	ic_1\iint_{\mathbb{R}^{3}\times\mathbb{R}^{3}}
	e^{is\phi(\xi,\eta,\sigma)}
	\vert\eta\vert^{-2}
	\widehat{v}(s,\xi-\eta)
	\widehat{v}(s,\eta+\sigma)
	\overline{\widehat{v}(s,\sigma)}
	d\eta d\sigma,
	\end{align}
	where $c_1 = -2(2\pi)^{-5}\lambda$ and
	$$
	\phi(\xi,\eta,\sigma) =
	\vert\xi\vert^{\alpha}
	-\vert\xi-\eta\vert^{\alpha}
	-\vert\eta+\sigma\vert^{\alpha}
	+\vert\sigma\vert^{\alpha}.
	$$	
	These formulae play a crucial role in the proof of weighted energy estimates.
	This method which is based on dealing with space-time resonances have been systematically studied in \cite{PNJ}, where general approach
	of using Fourier analysis methods to investigate the long-time behavior of dispersive PDEs is arranged. Especially for the modified scattering with this technique, see also \cite{KP}. Note that $\sa ( xv) =  \mathbf J u$, where
	$$
	\mathbf J = \sa x e^{it\vert\nabla\vert^{\alpha}} = x + i\al t |\nabla|^{\al-2}\nabla.
	$$

	If the solution $u$ of \eqref{equation} is sufficiently smooth, then it satisfies the mass and energy conservation laws:
	\begin{align}\begin{aligned}\label{consv}
	m(u) &= \|u\|^2_{L^2} = m(u_0), \\  
	E(u) &=  \frac12\||\nabla|^\frac\al2 u\|_{L^2}^2 + \frac\lambda4\iint |x-y|^{-1}|u(x)|^2|u(y)|^2\,dxdy = E(u_0).
	\end{aligned}\end{align}
	The equation \eqref{equation} has the scaling invariance in $\dot H^\frac{1-\al}2$ and thus it is referred to mass(energy)-subcritical.
	The subcritical nature readily leads us to the global well-posedness. This can be done by a simple energy estimate in $H^N$. In addition, a weighted energy estimate enables us to show the global evolution of $xv, x^2v, \widehat v$ such that
	$$
	x v \in C(\mathbb R; H^3), x^2 v \in C(\mathbb R; H^2), \bxi^5 \widehat v \in C(\mathbb R; C_b(\mathbb R^3)),
	$$
	provided the same regularity conditions are imposed to the initial data. We will deal with the details in Section 2.

	In this paper we focus on asymptotic behavior of solution to $\eqref{equation}$ as time goes to infinity. We say that the solution $u$ scatters to a linear asymptotic state if the effect of the nonlinear term becomes negligible as time goes to infinity. But our equation may not scatter even though the initial data is arbitrarily small \cite{cho}. Instead, we can observe the phenomenon of ``modified scattering" for small solutions by identifying a proper nonlinear logarithmic correction. This nonlinear modified scattering also happens similarly for standard Hartree euqation($\alpha=2$ case) \cite{HN} and Boson star equation \cite{pu}. More precisely, our main topic can be stated as follows: For sufficiently small initial data $u_{0}$ which is defined in a weighted Sobolev space, there exist a global solution to $\eqref{equation}$ which decays in $L^\infty$ but behaves in nonlinear fashion over time.
	\begin{rk}
		Until now there have been numerous results on the scattering on the fractional Schr\"odinger equations with general Hartree type nonlinearity including power type nonlinearity. We refer the readers to \cite{chooz, chho, gswz, hs, cho1, chkl, bhl, pu, hnog, cho, c0, chooz2} and references therein.
		We brief on the known results.
		
		$(1)$ If $V = \lambda |x|^{-\gamma}$ and $0 < \gamma \le 1$ (or $1 < \gamma < 3)$, then $V$ is referred to be of long-range (or short-range, respectively) interaction. If $V$ has a long range, it was shown in \cite{cho} that many smooth solutions may not scatter even in $L^2$. The short-range scattering in $H^s$ can be shown simply by Strichartz estimates when $2 < \gamma < 3$ and $s > \frac{\gamma-\al}2$ since the dispersion of solution is fast enough. This is also the case for Hartree and semi-relativistic equation. See \cite{go, chooz, hnog, chho}.
		
		$(2)$ In case when $1 < \gamma \le 2$, the dispersion of solution to \eqref{equation} is not enough for Strichartz estimate on the whole time interval. In view of the scattering theory of Hartree and semi-relativistic equations, the scattering is expected to be shown in this range via radial symmetry assumption or weighted energy estimates. Under the radial assumption the global Strichartz estimate can cover the range $1 < \gamma \le 2$ in part. To be more precise, small data scattering in $H^\frac{\gamma-\al}2$ is possible when $\al, \gamma$ is restricted to $\frac{6}{5} \le \al < 2$ and $\al \le \gamma < 3$. For this see \cite{chho}. In \cite{cho1} even a large data scattering in energy space is treated under radial symmetry when $\gamma = 2\al$ (energy-critical) and $\frac{6}{5} < \al < 2$.
		
		$(3)$ The other way is to use a weighted energy estimate for the norm $\|\mathbf J u\|_{\dot H^s}$ as in Hartree Schr\"odinger equations \cite{hats} and semi-relativistic equations \cite{hnog}. If the initial data is in a weighted space, then the solution could be dispersive enough to scatter. Recently, in \cite{c0, chooz2} the small data scattering was shown in weighted space when $1 < \gamma \le 2$ and $\al_0:= \max\big(\frac{6-4\gamma}{2-\gamma}, 1\big) < \al <2$. The authors used a commutator estimate based on Balakrishnan's formula \cite{ba, klr, bhl} to get around the difficulty caused by the non-locality, the lack of dispersiveness of $\sa$ and the lack of smoothness of $|\nabla|^\al$. The cost for the commutator estimate is to restrict the range $\al$. It would be interesting to settle the remaining range $1 < \al \le \al_0$ for the short-range scattering problem.
	\end{rk}
	
	Our goal is to show that the global solution to $\eqref{equation}$ with the long-range ($\gamma = 1$) scatters in $L^\infty$ on the frequency space. Heuristically speaking, the time decay of $L^{2}$-norm of nonlinear term, which is computed on a linear solution, is $t^{-\gamma}$. For the details see the proof of Theorem 1.2 in \cite{cho}. In particular, the decay of nonlinear term is not integrable in time if $\gamma=1$, which referred as ``scattering-critical" case. The following is our main theorem.

	\begin{thm}[Modified scattering]\label{main thm}
		Let $\frac{17}{10} < \alpha < 2$ and $N = 1500$.
		Suppose $u_{0}$ satisfies that
		\begin{equation}\label{initialvalue}
		\lVert u_{0}\|_{H^{N}}
		+ \|x u_0\|_{H^3}
		+\lVert x^2 u_0\|_{H^{2}}
		+\lVert \bxi^{5} \widehat u_0\|_{L^{\infty}}
		\le \epsilon_{0}.
		\end{equation}	
		Then there exists $\bar{\epsilon}_{0}$ such that
		for all $\epsilon_{0}\le\bar{\epsilon}_{0}$,
		the Cauchy problem \eqref{equation} has a unique global solution $u(t,x)$ such that
		\begin{align}\label{uniform}
		\sup_{t > 0}\langle t \rangle^{\frac{3}{2}}
		\lVert u(t)\|_{L^{\infty}}
		\lesssim
		\epsilon_{0}.
		\end{align}
		Moreover, $u$ satisfies the asymptotic behavior as follows:
		Let
		\begin{align}\label{balpha}
		B_{\alpha}(t,\xi):= -\frac{\lambda}{\alpha(2\pi)^3}
		\int_{0}^{t}\int_{\mathbb{R}^3}
		\left\vert
		\frac{\xi}{|\xi|^{2-\alpha}}
		-\frac{\sigma}{|\sigma|^{2-\alpha}}
		\right\vert^{-1}
		\vert \widehat{u}(\xi)\vert^{2}
		d\sigma
		\varphi(s^{-\theta}\xi )
		\frac{1}{\langle s\rangle}ds,
		\end{align}
		where $\varphi$ is a smooth compactly supported function and $\theta = \frac{3\al-5}{40(\al+1)}$.
		Then there exist asymptotic state $v_{+}$,
		such that for all $t > 0$
		\begin{align}\label{asympt}
		\lVert \bxi^{5}
		[
		e^{-iB_{\alpha}(t,\xi)}
		\widehat{v}(t,\xi) - v_{+}(\xi)
		]
		\|_{L_{\xi}^{\infty}}
		\lesssim
		\langle t \rangle^{-\delta}
		\end{align}
		for some $0 < \delta < \min(\frac{2-\alpha}{3\al}, \frac1{100})$. Similar result holds for $t < 0$.
	\end{thm}

	Our approach is inspired by the work \cite{pu} of Pusateri developed to study semi-relativistic equations. As stated in \cite{pu} we prefer to state asymptotic state in the frequency space because the formulae \eqref{balpha} and \eqref{asympt} appear explicitly in the proof. We show the uniform norm estimate \eqref{uniform} and asymptotic behavior \eqref{asympt} for sufficiently small initial data via  refined time-decay and weighted energy estimates in frequency space, which are rephrased as Proposition \ref{prop1} and Propositions \ref{scattering}, \ref{Linfty}, respectively.
	

	Let us briefly give some intuition for the formula \eqref{balpha}. For simplicity we assume that $|\xi|\sim1$ and $|\eta|\lesssim 2^{n}$ with $n < 0$. Applying the Taylor expansion to the phase function $\phi$ of \eqref{Duhamel}, we approximate $I$ by
	\begin{align*}
	ic_1
	&\iint_{\mathbb{R}^{3}\times\mathbb{R}^{3}}
	e^{i \alpha s (\eta \cdot \mathbf z)}
	\vert\eta\vert^{-2}
	\widehat{v}(s,\xi+\eta)
	\overline{\widehat{v}(}s,\xi+\eta+\sigma)
	\widehat{v}(s,\xi+\sigma)
	d\eta d\sigma
	\\
	&\qquad =ic_1
	\widehat{v}(s,\xi)
	\int_{\mathbb{R}^{3}}
	\mathcal{F}^{-1}(\vert\eta\vert^{-2})(s\mathbf z)
	|\widehat{v}(s,\sigma)|^{2}
	d\sigma + [\mathbf{ err}],
	\end{align*}
	where $\mathbf z = \frac{\xi}{\vert\xi\vert^{2-\alpha}} - \frac{\sigma}{\vert\sigma\vert^{2-\alpha}}$ and $c_1\mathcal F^{-1}(|\eta|^{-2}) = -\lambda(2\pi)^{-3}|x|^{-1}$. This formula yields an insight for \eqref{balpha} and \eqref{asympt}. In Section 5 below $[\mathbf{err}]$ will turn out to be $O(s^{-1-})$ as $s \to \infty$. The contribution of remaining region $|\eta|\gtrsim 2^{n}$ is shown to decay faster than $s^{-1}$ by making integration by parts twice.

	Similarly to the case of short-range potential with $\gamma$ close to $1$, we could not obtain a modified scattering in the whole range $1 <\alpha < 2$ for the present. This is due to the lack of smoothness of $|\nabla|^\al$ near zero frequency. This is the main difference of fractional  from the usual Schr\"odinger or semi-relativistic equations. The drawback can be overcome by the refined time-decay estimate. But it is inevitable to control at least the $L^{2}$ norm of $x^{2}v$ for the requested time decay. Among the terms from taking $\nabla_{\xi}^{2}$ to $\widehat{v}$ and $I$ in \eqref{integraleqn}, the following is the most worst case when we consider small $\alpha$:
	$$
	c_0 \int_{0}^{t}
	s
	\int_{\mathbb{R}^{3}}
	[\nabla_\xi\otimes\nabla_\xi \phi_\al]
	e^{is\phi_\al(\xi, \eta)}
	\vert\eta\vert^{-2}
	\widehat{\vert u\vert^{2}}(\eta)
	\widehat{v}(\xi-\eta)
	d\eta
	ds.
	$$
	In this expression, twice differentiation of the phase function $\phi_\al$ rises to singularity near $0$ of order $\alpha-2$, which makes a problem in bounding the low frequency part when $\al$ becomes closer to $1$. In the derivation of the asymptotic correction term as above on the region $|\eta|\lesssim 2^{n}$, we require the integral of $|\eta|^{-2+\alpha}$ over this region to be $O(s^{-2-})$. This condition also restricts the range of $\alpha$. One of our next subjects will be to remove the gap on $\alpha$.
	
	\begin{rk}
		The indices $\al, N, \delta, \theta$ are not sharp and can be adjusted. For example, the range of $\al$ can be made slightly less than $\frac{17}{10}$ by defining $\theta = \frac{3\al-5}{1000(\al+1)}$. See conditions on the indices in Propositions \ref{scattering} and \ref{Linfty}.
	\end{rk}
	\begin{rk}
		One can obtain similar refined time-decay estimate for high dimensional case such that
		\begin{align*}
		\| e^{-it\left\vert \nabla\right\vert^{\alpha}}v\|_{L^\infty(\mathbb{R}^{d})}
		&\lesssim
		\bt^{-\frac{d}{2}}\| \bxi^{a}\widehat{v}\|_{L^{\infty}(\mathbb{R}^{d})}\\
		&\qquad\qquad+\bt^{-\frac{3}{2}-\delta}
		\big(\| x^{[\frac{d}{2}]+1}v \|_{L^{2}(\mathbb{R}^{d})}+\| v\|_{H^{N}(\mathbb{R}^{d})}\big),
		\end{align*}
		for some $a(\alpha, d) > 1$ and $\delta(\alpha,d) \ll 1$ and for sufficiently large $N(a,\delta) $. It is highly expected to extend Theorem \ref{main thm} to the high dimensional case $(d \ge 4)$ even though there will be much more complexity in frequency space analysis arising from the $[d/2]+1$-times differentiation of the phase function $\phi$.
	\end{rk}

	This paper is organized as follows: In Section 2, we deal with the global well-posedness and evolution of $xv, x^2v$ and $\bxi^5 \widehat v$. Section 3 is devoted to proving the refined time-decay estimate. In Section 4, we establish the weighted energy estimate based on the Littlewood-Paley theory. Main effort is made to overcome the singularity from differentiation. In Section 5 we move on to the last step for the proof of modified scattering.  In the last section we list lemmas for multiplier estimates and  bounds for $|\mathbf z|$.

	{\it Notation}. We conclude this introduction by giving some notations which will be used frequently throughout this paper.
	
	\noindent$\bullet$ Littlewood-Paley operators: $\beta \in C_{0, rad}^\infty$ with $\beta \widetilde{\beta} = \beta$ and $\widetilde{\beta}(\xi) = \beta(\xi/2) + \beta(\xi) + \beta(2\xi)$. $\mathcal F ( P_k f)(\xi) = \beta(\xi/2^k)\widehat{v}$ for any $k \in \mathbb Z$. Let $\widetilde{P}_k = P_{k-1} + P_k + P_{k+1}$. Then $P_k\widetilde{P}_k = P_k$.

	\noindent$\bullet$ For any $b \in \mathbb R$ we use $\langle b \rangle = (1 + b^2)^\frac12$ and also $\langle x \rangle = (1 + |x|^2)^\frac12$ for any $x \in \mathbb R^d$.
	
	\noindent$\bullet$ Fractional derivatives: $|\nabla|^s = (-\Delta)^\frac{s}2 = \mathcal F^{-1}|\xi|^s\mathcal F$, $(1- \Delta)^\frac{s}2 = \mathcal F^{-1}\bxi^s \mathcal F$ for $s > 0$.
	
	\noindent $\bullet$ Let $\mathbf A = (A_i)$ and $\mathbf B = (B_j)$ be any vectors in $\mathbb R^d$. Then $\mathbf A \otimes \mathbf B$ denotes the usual tensor product such that
	$(A\otimes B)_{ij} = A_i B_j$. The same notation is used for the derivatives, i.e. $\nabla \otimes \nabla = (\partial_i\partial_j)_{i,j = 1,\cdots, d}$. We also use $\nabla \otimes x$, $x\otimes \nabla$.
	
	\noindent $\bullet$ For any positive integer $\ell$, and for any vector or derivative $\mathbf A$, $\mathbf A^\ell$ denotes the $\ell$-times product $\mathbf A \otimes \cdots \otimes \mathbf A$.
	
	\noindent $\bullet$ Let $\mathbf T = (T_{j_1, \cdots, j_k}), \mathbf S = (S_{i_1, \cdots, i_l})$ be $k$-times and $l$-times product of tensors and derivatives. Then we define their dot product by
	$$\mathbf T ; \mathbf S := \sum_{\substack{1\le j_1, \cdots, j_k\le d\\
			1\le i_1, \cdots, i_l\le d}}T_{j_1,\cdots,j_k}S_{i_1,\cdots, i_l}.$$

	\noindent $\bullet$ For tensor-valued function $\mathbf{F} = (F_{i_1, \cdots i_l})$ we use the norm $$\|\mathbf{F}\|_{X} := \sum_{1\le i_1,\cdots, i_l \le d} \|F_{i_1, \cdots, i_l}\|_X.$$

	\noindent$\bullet$ As usual, different positive
	constants depending only on $d, \al$ are denoted by the same letter $C$, if not specified. $A \lesssim B$ and $A \gtrsim B$ mean that $A \le CB$ and
	$A \ge C^{-1}B$, respectively for some $C>0$. $A \sim B$ means that $A \lesssim B$ and $A \gtrsim B$.

	\section{Global Well-posedness}
	In this section we establish a global theory.
	\begin{thm}\label{GW}
		$(1)$ Let $1 < \al < 2$ and $N \ge \frac\al2$. If $u_0 \in H^N$, then there exists a unique solution $u \in C(\mathbb R; H^N)$ satisfying mass and energy conservations.\\
		$(2)$ Let $1 < \al < 2$ and $N \ge 5$. Assume that
		$$
		u_{0} \in H^N, x u_0 \in H^3, x^2 u_0 \in H^2.
		$$
		Then there exists a unique solution $u$ to \eqref{equation} such that
		\begin{align*}
		u \in C(\mathbb R; H^N), \partial_t u \in C(\mathbb R; H^{N-\al}), x v \in C(\mathbb R; H^3), x^2 v \in C(\mathbb R; H^2).
		\end{align*}
		Moreover, if  we further assume that $\bxi^5\widehat u_0 \in L^\infty$, then $\bxi^5 \widehat v \in C(\mathbb R; C_b(\mathbb R^3))$.
	\end{thm}

	\begin{proof}

		For the proof of the global well-posedness in $H^N$ and conservations, we refer the readers to Theorem 3.3 of \cite{chho}. The control norm of well-posedness is $\|u\|_{\dot H^\frac12}$. That is, if $\|u\|_{\dot H^\frac12}$ is finite at some time, then the solution evolves beyond the time in $H^N$. Regardless of the sign of $\lambda$, one can show that the control norm is uniformly bounded in time due to the subcritical nature of \eqref{equation}. In fact, if $\lambda < 0$, then by interpolation and conservation laws we have
		\begin{align*}
		\|u\|_{\dot H^\frac12} \le \|u\|_{L^2}^\frac{\al-1}\al\||\nabla|^\frac\al2 u\|_{L^2}^\frac1\al \le 2^\frac1{2\al} m(u_0)^\frac{\al-1}{2\al}\big(|E(u_0)| - V(u)\big)^\frac1{2\al},
		\end{align*}
		where $V(u) = -\frac{\lambda}4\iint |x-y|^{-1}|u(x)|^2|u(y)|^2\,dxdy$. Hardy-Sobolev inequality gives $-V(u) \le Cm(u_0)\|u\|_{\dot H^\frac12}^2$ and hence
		$$
		\|u\|_{\dot H^\frac12} \le 2^\frac1{2\al}m(u_0)^\frac{\al-1}{2\al}|E(u_0)|^\frac1{2\al} + (2C)^\frac1{2\al} m(u_0)^\frac12\|u\|_{\dot H^\frac12}^\frac1\al.
		$$
		Now Young's inequality yields $\|u\|_{\dot H^\frac12} \le C(m(u_0), E(u_0))$.
		
		By direct calculation we get $xv = \al t |\nabla|^{\al-2}\nabla u + e^{it\nabla|^\al}xu$ and
		$$
		x^2v = \al t((\al-2)|\nabla|^{\al-4}\nabla^2 + |\nabla|^{\al-2}I)u + \al t|\nabla|^{\al-2}\nabla \otimes x u + e^{it|\nabla|^\al} x^2u.
		$$
		Since $\||\nabla|^{\al-2}u\|_{L^2} \lesssim \||x|^{2-\al}u\|_{L^2}$ by Hardy-Sobolev inequality, to show that $x v \in C(\mathbb R; H^3)$, $x^2 v \in C(\mathbb R; H^2)$ we have only to take into account $x u(t) \in C(\mathbb R; H^3)$, $x^2u(t) \in C(\mathbb R; H^2)$.
		But this can be done by the standard approximation with $\psi_\ep(x) = e^{-\ep|x|^2}$. We treat this part in Lemma \ref{moment} below.
		
		Lastly, we show that $\bxi^5\widehat{v} \in C(\mathbb R; C_b(\mathbb R^3))$.
		Since $x^2v \in L^2$, $\widehat{v} \in C_b(\mathbb R^3)$. We show the time continuity on $[0, \infty)$. The continuity on the negative time interval can be shown by symmetry. From \eqref{integraleqn} it follows that
		\begin{align*}
		&\|\bxi^5 \widehat{v}(t) \|_{L_\xi^{\infty}}\\
		&\lesssim \|\bxi^5 \widehat{u}_0 \|_{L_\xi^{\infty}} +
		\int_0^{t}
		\big\| \langle\xi\rangle^{5}
		\int_{\mathbb{R}^{3}}
		e^{is(|\xi|^{\alpha}-|\xi-\eta|^{\alpha})}
		|\eta|^{-2}
		\widehat{|u|^{2}}(\eta)\widehat{v}(\xi-\eta)
		d\eta
		\big\|_{L_{\xi}^{\infty}}\,ds \\
		&\lesssim  \|\bxi^5 \widehat{u}_0 \|_{L_\xi^{\infty}} +
		\int_0^{t}
		\big\|
		\int_{\mathbb{R}^{3}}
		|\eta|^{-2}
		|\eta|^{5}
		\widehat{|u|^{2}}(\eta)
		\widehat{v}(\xi-\eta)
		d\eta
		\big\|_{L_{\xi}^{\infty}}ds  	\\
		&\qquad\qquad\qquad\;\, +\int_0^{t}
		\big\|
		\int_{\mathbb{R}^{3}}
		|\eta|^{-2}
		\widehat{|u|^{2}}(\eta)
		\langle\xi-\eta\rangle^{5}
		\widehat{v}(\xi-\eta)
		d\eta
		\big\|_{L_{\xi}^{\infty}}ds  	\\
		&\lesssim  \|\bxi^5 \widehat{u}_0 \|_{L_\xi^{\infty}} +
		\int_0^t\|\widehat{v}\|_{L_\xi^2}\|\widehat{\nabla^3 |u|^2}\|_{L_\xi^2}\,ds +
		\int_0^{t}\|\langle\xi\rangle^{5}
		\widehat{v}\|_{L_\xi^\infty}
		\big\| |\eta|^{-2}
		\widehat{|u|^{2}}(\eta)
		\big\|_{L_{\eta}^{1}}ds.
		\end{align*}
		Now for the second integral we use the estimate
		$$\int |\eta|^{-2}|\widehat{|u|^2}|\,d\eta = \int_{|\eta|\le 1} + \int_{|\eta|> 1} \lesssim \|u\|_{L^2}^2 + \|u\|_{L^4}^2$$
		and get
		\begin{align}\label{gronwal}
		\|\bxi^5 \widehat{v}(t) \|_{L_\xi^{\infty}} \lesssim  \|\bxi^5 \widehat{u}_0 \|_{L_\xi^{\infty}} + \int_0^t\|u\|_{H^3}^3\,ds +
		\int_{0}^t
		\|\langle\xi\rangle^{5}
		\widehat{v}\|_{L_\xi^\infty}\|u\|_{H^1}^2\, ds.
		\end{align}
		Therefore by Gronwal's inequality we obtain that for each $t > 0$
		$$
		\|\bxi^5 \widehat{v}(t) \|_{L_\xi^{\infty}} \lesssim \left(\|\bxi^5 \widehat{u}_0 \|_{L_\xi^{\infty}} + \int_0^t\|u\|_{H^3}^3\,ds\right)e^{C\int_0^t \|u\|_{H^1}^2\,ds} < \infty.
		$$
		Then time continuity readily follows from considering the inequality
		\begin{align*}
		\|\bxi^5(\widehat{v}(t) - \widehat{v}(t'))\|_{L^\infty} \lesssim \int_{t'}^t\left(\|u\|_{H^3}^3 + \|\langle\xi\rangle^{5}
		\widehat{v}\|_{L_\xi^\infty}\|u\|_{H^1}^2\right)\,ds
		\end{align*}
		for any $0 \le t' < t$.
		This completes the proof of Theorem \ref{GW}.
		
	\end{proof}
	
	\newcommand{\mm}{\mathcal{M}}
	\newcommand{\wmm}{\widetilde{\mathbf{m}}}
	\newcommand{\pe}{\psi_{\varepsilon}}
	\newcommand{\bv}{{\mathbf u}}
	\newcommand{\tmm}{{\widetilde{\mathcal M}}}
	\newcommand{\tv}{{\widetilde u}}
	
	\begin{lm}\label{moment}
		Let $u $ be the solution to \eqref{equation} belong to $C(\mathbb R; H^5)$ with $\partial_t u \in C(\mathbb R; H^{5-\al})$ and initial data $u_0$ such that $xu_0 \in H^3, x^2u_0 \in H^2$. Then $x u \in C(\mathbb R; H^3)$, $x^2 u \in C(\mathbb R; H^2)$.
	\end{lm}
	\begin{proof}[Proof of Lemma \ref{moment}]
		Let us set $\pe(x) = e^{-\varepsilon|x|^2}$. Let $\mathbf u = \nabla^\ell u$ for $0 \le \ell \le 3$, and
		$$
		\mm_{\ell, \varepsilon}(t) = \int \bv(t) : \overline{|x|^{2} \pe^2
			\bv(t)}\,dx.
		$$
		From the regularity of the solution $u$ it follows that
		\begin{align}\begin{aligned}\label{diff-m}
		\frac{d}{dt}\mm_{\ell, \varepsilon}(t) &= 2{\rm Im} \int \bv : \overline{[|\nabla |^\alpha, |x|^2 \pe^2] \bv}\,dx \\
		&\qquad\qquad+ 2\lambda{\rm Im} \int |x| \pe \bv : \overline{|x| \pe \nabla^\ell((|x|^{-1}*|u|^2)u)}\,dx\\
		& =: 2(\,A_\ell + B_\ell\,).
		\end{aligned}\end{align}
		We rewrite $A_\ell$ as
		\begin{align*}
		A_\ell &= {\rm Im }\int |x| \pe \bv : [|\nabla|^\alpha (1-\Delta)^{-1}, |x| \pe] \overline{(1-\Delta) \bv}\,dx\\
		&\qquad + {\rm Im}\int |\nabla|^\alpha (1-\Delta)^{-1}(|x| \pe \bv) : \overline{[1-\Delta, |x| \pe] \bv}\,dx =: A_{\ell, 1} + A_{\ell, 2}.
		\end{align*}
		Here $[T, S]$ denotes the commutator $TS - ST$.
		By the kernel representation of $|\nabla|^\alpha (1-\Delta)^{-1}$, we  have
		\begin{align*}
		|[|&|\nabla|^\al (1-\Delta)^{-1}, |x| \pe|](1-\Delta)\bv (x)|\\
		&\le \int K(x-y)||x| \pe(x)- |y| \pe(y)||(1-\Delta)\bv u(y)|\,dy\\
		&\lesssim \int K(x-y)|x-y| |(1-\Delta)\bv(y)|\,dy.
		\end{align*}
		Since $(1+|x|)^N K$ is integrable for all $N \ge 1$ (see \cite{st}, Cauchy-Schwarz inequality gives
		$$
		A_{\ell, 1} \lesssim \sqrt{\mm_{\ell, \varepsilon}}\,\|u\|_{H^{\ell+2}}.
		$$
		As for $A_{\ell, 2}$
		we have
		\begin{align*}
		A_{\ell, 2} &=  -{\rm Im}\int |\nabla|^\alpha(1-\Delta)^{-1}(|x| \pe \bv) :
		\overline{\big(\Delta( |x| \pe) \bv + 2\nabla(|x|\pe ) \cdot \nabla \bv \big)}\,dx\\
		&\lesssim \sqrt{\mm_{\ell, \ep}}\,\|u\|_{H^{\ell+1}}.
		\end{align*}

		Now we proceed to estimate $B$. If $\bv = u$, then $B = 0$. For the case $\bv = \nabla^\ell u$ with $\ell > 0$ let us observe that
		\begin{align*}
		B_\ell &\lesssim  \sqrt{\mm_{\ell, \varepsilon}}\,\sum_{1 \le \ell' \le \ell}\||x| \nabla^{\ell'}(|x|^{-1}*|u|^2)\|_{L^\infty}\|\nabla^{\ell-\ell'} u\|_{L^2}\\
		&\lesssim  \sqrt{\mm_{\ell, \varepsilon}}\,\sum_{1 \le \ell' \le \ell} \||x| (|x|^{-1}* \nabla^{\ell'} (|u|^2)\|_{L^\infty}\|\nabla^{\ell-\ell'} u\|_{L^2}.
		\end{align*}
		By Young's and Hardy-Sobolev inequalities we have
		\begin{align*}
		\left||x| (|x|^{-1}*\nabla^{\ell'} (|u|^2))\right| &\le \int \nabla^{\ell'}(|u|^2) \,dy + \int |x-y|^{-1}|y|\nabla^{\ell'}(|u|^2)\,dy\\
		&\lesssim \|xu\|_{H^{\ell'-1}}^2 + \|u\|_{H^{\ell'+1}}^2.
		\end{align*}

		By integrating \eqref{diff-m} over $[0, t]$ we have
		\begin{align*}
		\sqrt{\mm_{0, \ep}} &\lesssim \|xu_0\|_{L^2} +\int_0^t \|u\|_{H^2}\,ds,\\
		\sqrt{\mm_{\ell, \ep}} &\lesssim \|xu_0\|_{H^\ell} + \int_0^t (\|u\|_{H^{\ell+2}} + \|u\|_{H^{\ell+1}}^3 + \|xu\|_{H^{\ell-1}}^2\|u\|_{H^{\ell-1}} )\,ds,
		\end{align*}
		Fatou's lemma and induction lead us to $xu(t) \in L^\infty(K; H^3)$ for any compact interval $K \subset \mathbb R$ provided $xu_0 \in H^3$. By using this fact and the equation \eqref{equation} we can conclude that $xu \in C(\mathbb R; H^3)$.

		Let us move onto the proof of $x^2 u \in C(\mathbb R; H^2)$. Let us set $\tv = (1-\Delta) u$ and
		$$
		\tmm_\varepsilon(t) = \int \tv(t) \overline{|x|^{4} \pe^2 \tv(t)}\,dx.
		$$
		Differentiating w.r.t $t$ we have as above that
		\begin{align*}
		\frac{d}{dt}\tmm_\varepsilon(t) &= 2{\rm Im} \int \tv  \overline{[|\nabla |^\alpha, |x|^4 \pe^2] \tv}\,dx\\
		&\qquad\qquad\qquad+ 2\lambda{\rm Im} \int |x|^2 \pe \tv  \overline{|x|^2 \pe (1-\Delta)((|x|^{-1}*|u|^2)u)}\,dx\\
		& =: 2(\,\widetilde A + \widetilde B\,).
		\end{align*}
		$\widetilde A$ is written as
		\begin{align*}
		\widetilde A &= {\rm Im }\int |x|^2 \pe \tv  \overline{[|\nabla|^\alpha (1-\Delta)^{-1}, |x|^2 \pe] (1-\Delta)\tv}\Big\rangle\\
		&\qquad + {\rm Im}\int |\nabla|^\alpha (1-\Delta)^{-2}(|x|^2 \pe \tv) \overline{[1-\Delta, |x|^2 \pe]\tv}\,dx =: \widetilde A_1 + \widetilde A_2.
		\end{align*}
		Then we  have
		\begin{align*}
		|[|&|\nabla|^\al (1-\Delta)^{-1}, |x|^2 \pe|](1-\Delta)\tv (x)|\\
		&\le \int K(x-y)||x|^2 \pe(x)- |y|^2 \pe(y)||(1-\Delta)\tv(y)|\,dy\\
		&\lesssim \int K(x-y)|x-y|(|x-y|+ 2|y|) |(1-\Delta)\tv(y)|\,dy
		\end{align*}
		and hence
		$$
		\widetilde A_1 \lesssim \sqrt{\tmm_\varepsilon}\,(\|u\|_{H^{4}} +
		\|x u\|_{H^2}).
		$$
		As for $A_2$
		we have
		\begin{align*}
		A_2 &=  -{\rm Im}\int |\nabla|^\alpha(1-\Delta)^{-1}(|x|^2 \pe \tv)\overline{
			\big(\Delta( |x|^2 \pe) v + 2\nabla(|x|^2\pe ) \cdot \nabla \tv \big)}\,dx\\
		&\lesssim \sqrt{\tmm_\ep}\,(\|u\|_{H^{2}} +
		\|x  u\|_{H^3}).
		\end{align*}
		$\widetilde B$ can be treated similarly to $B_\ell$ as follows:
		\begin{align*}
		\widetilde B &\lesssim  \sqrt{\tmm_\varepsilon}\,\big(\||x| (1-\Delta)(|x|^{-1}*|u|^2)\|_{L^\infty}\||x| u\|_{L^2}\\
		&\qquad\qquad\qquad + \||x| \nabla(|x|^{-1}*|u|^2)\|_{L^\infty}\||x|\nabla u\|_{L^2}\big)\\
		&\lesssim \sqrt{\tmm_\varepsilon}\,\|xu\|_{H^1}^2\|u\|_{H^2}.
		\end{align*}
		Combining these estimates and the previous argument, we deduce that $x^2u \in C([0, \infty); H^2)$ since $u \in C([0, \infty); H^5)$ and $xu \in C([0, \infty); H^3)$. At the negative time, we can carry out the same argument. This completes the proof of Lemma \ref{moment}.
	\end{proof}

	\section{Time decay}

	\begin{prop}\label{prop1}
		Let $ 1<\alpha<2$, $t\in\mathbb{R}$ and $N_0 > 4$. Then we have
		$$
		\| e^{-it\left\vert \nabla\right\vert^{\alpha}}v\|_{L^\infty}
		\lesssim \bt^{-\frac{3}{2}}\| \bxi^{4-2\alpha}\widehat{v}\|_{L^{\infty}}
		+\bt^{-\frac{3}{2}-(\frac14-\frac1{N_0})}
		\big(\| x^{2} v \|_{L^{2}}+\| v\|_{H^{N}}\big),
		$$
		if $N \ge \frac72 N_0(2-\al) + 2\al-\frac{5}{2}$.
	\end{prop}
	
	\begin{proof}[Proof of Proposition \ref{prop1}] We assume that $t \ge 0$.
		
		Since $\| e^{-it\left\vert \nabla\right\vert^{\alpha}}v\|_{L^\infty}
		\lesssim 	
		\| v\|_{H^{N}}
		$ for $\ N>\frac{3}{2} $,
		we may assume $t > 1$.
		We write by dyadic decomposition in the Fourier side
		$$
		\| e^{-it\vert \nabla\vert^{\alpha}} v \|_{L^\infty}
		\lesssim \sup_{x}\sum_{k \in \mathbb Z}
		I_{k}(t,x),
		$$
		where
		\begin{align}\label{Ikdyadic}
		I_{k}(t,x)	= \frac1{(2\pi)^3}\big\vert	\int  e^{i\phi_{t,x}(\xi)}	\beta_{k}(\xi)\widehat{v}(\xi)	d\xi \big\vert,\quad	\phi_{t,x}(\xi)	= -t\left\vert\xi\right\vert^{\alpha}+x\cdot\xi.
		\end{align}
		
		For the low frequency part of summation we estimate
		\begin{align*}
		\sum_{2^{k}\leq \bt^{-\frac{1}{2}}}
		I_{k}(t,x)
		\lesssim \sum_{2^{k}\leq \bt^{-\frac{1}{2}}}
		\| \beta_{k}\|_{L^{1}}\|\widehat{v}\|_{L^{\infty}}
		\lesssim \bt^{-\frac{3}{2}}
		\|\widehat{v}\|_{L^{\infty}}.
		\end{align*}
		
		And for the high frequency part we have that
		\begin{align*}
		\sum_{2^{k}\ge \bt^{\frac{1}{2N_0(2-\alpha)}}}
		I_{k}(t,x)
		&\lesssim
		\sum_{2^{k}\ge\bt^{\frac{1}{2N_0(2-\alpha)}}}
		2^{\frac{3k}{2}}
		\|\beta_{k}\widehat{v}\|_{L^{2}} \\
		&\lesssim
		\sum_{2^{k}\ge \bt^{\frac{1}{2N_0(2-\alpha)}}}
		2^{\frac{3k}{2}}2^{-Nk}
		\| v\|_{H^{N}}   \\
		&\lesssim
		\bt^{-\frac1{2N_0(2-\alpha)}(N-\frac{3}{2})}
		\| v\|_{H^{N}}.
		\end{align*}
		If $N \ge \frac72N_0(2-\al) + 2\al-\frac{5}{2}$, then 	
		\begin{align}\label{t-sob}
		\sum_{2^{k}\ge \bt^{\frac1{2N_0(2-\alpha)}}}
		I_{k}(t,x)
		\lesssim \bt^{-\frac{3}{2}-(\frac14-\frac1{N_0})} \| v\|_{H^{N}}.
		\end{align}
		
		Now, let us bound the remaining part:
		\begin{equation}\label{remainderterm}
		\sum_{\bt^{-\frac{1}{2}}\le2^{k}\le \bt^{\frac{1}{2N_0(2-\alpha)}}}
		\left\vert
		\int
		e^{i(-t\left\vert\xi\right\vert^{\alpha}+x\cdot\xi)}
		\beta_{k}(\xi)\widehat{v}(\xi)
		d\xi
		\right\vert.
		\end{equation}
		First, we consider the non-stationary case.	We write the phase function as
		\begin{equation}\label{xizero}
		\nabla\phi_{t,x}(\xi)
		= -\alpha t
		(\frac{\xi}{\vert\xi\vert^{2-\alpha}}
		-\frac{\xi_{0}}{\vert\xi_{0}\vert^{2-\alpha}}),
		\ \textrm{where} \
		\xi_{0}
		= \left(\frac{\vert x\vert}{\alpha t}\right)^{\frac{1}{\alpha-1}}
		\frac{x}{\vert x\vert}.
		\end{equation}
		For $\xi_{0}$ in $(\ref{xizero})$, there exist $k_{0}\in\mathbb{Z}$	such that  $\vert \xi_{0} \vert \sim 2^{k_{0}}.$
		Then by Lemma $\ref{lowbound}$, we see that
		if $\vert k-k_{0}\vert\ge4 $
		$$
		\vert\nabla\phi_{t,x}(\xi)\vert
		\gtrsim\
		\vert t\vert2^{(\alpha-1)k}
		\ \textrm{for} \
		\vert\xi\vert\sim2^{k}.
		$$
		By taking integration by parts twice in the expression \eqref{remainderterm}, we get
		\[
		\sum_{k \in [Non-Stat]}
		I_{k}(t,x)
		\lesssim \mathcal A + \mathcal B + \mathcal C,
		\]
		where $[Non-Stat]$ denotes the set $\{k : \bt^{-\frac{1}{2}}\le2^{k}\le \bt^{\frac1{2N_0(2-\alpha)}}, \left\vert k-k_{0}\right\vert > 4\}$,
		\begin{align*}
		\mathcal A &= \sum_{k \in [Non-Stat]} t^{-2}2^{-2\alpha k}
		\| \beta_{k}\widehat{v} \|_{L^{1}}, \\
		\mathcal B &= \sum_{k \in [Non-Stat]}t^{-2}2^{(-2\alpha+1)k}
		\| \nabla(\beta_{k}\widehat{v}) \|_{L^{1}}, \\
		\mathcal C &= \sum_{k \in [Non-Stat]}t^{-2}2^{(-2\alpha+2)k}
		\|\nabla^{2}(\beta_{k}\widehat{v})\|_{L^{1}}.
		\end{align*}
		We denote briefly $v_k=\widetilde{P}_k v$, then it holds $\beta_k\widehat{v}=\beta_k\widehat{v_k}$. H\"older's and Bernstein's inequalities give us that
		\begin{align*}
		\mathcal A &\lesssim
		\sum_{\bt^{-\frac{1}{2}}\le2^{k}}
		t^{-2}2^{-2\alpha k}
		\| \beta_{k}\|_{L^{1}}
		\| \widehat{v_{k}}\|_{L^{\infty}} \lesssim
		\sum_{\bt^{-\frac{1}{2}}\le2^{k}} t^{-2}
		2^{-2\alpha k}2^{3k}
		\| \widehat{v_{k}}\|_{L^{\infty}}  \\
		&\lesssim
		\sum_{\bt^{-\frac{1}{2}}\le2^{k}}
		t^{-2}2^{-k}
		\|
		\bxi^{4-2\alpha}\widehat{v}
		\|_{L^{\infty}}  \lesssim
		t^{-\frac{3}{2}}
		\| \bxi^{4-2\alpha}\widehat{v}
		\|_{L^{\infty}}.
		\end{align*}
		By Sobolev embedding and Plancherel's theorem we have
		\begin{equation}\label{ineq:L6x2f}
		\|\nabla\widehat{v}\|_{L^{6}}
		\lesssim
		{\|\nabla ^2 \widehat{v}\|_{L^{2}}}
		\lesssim \| x^{2}v\|_{L^{2}}.
		\end{equation}
		Using this inequality, we see that
		\begin{align*}
		\mathcal B &\lesssim
		\sum_{k \in [Non-Stat]}
		t^{-2}2^{(-2\alpha+1)k}
		\Big( 2^{-k}\| \beta_{k}\widehat{v}\|_{L^{1}}
		+\| \beta_{k}\nabla\widehat{v}\|_{L^{1}} \Big) \\
		&\lesssim
		\sum_{k \in [Non-Stat]}
		t^{-2}2^{(-2\alpha+1)k}
		\Big( 2^{2k}\| \widehat{v_{k}}\|_{L^{\infty}}
		+2^{\frac{5k}{2}}\|\nabla\widehat{v}\|_{L^{6}} \Big) \\
		&\lesssim
		\sum_{k \in [Non-Stat]}
		t^{-2}2^{(-2\alpha+1)k}
		\Big( 2^{2k}\| \widehat{v_{k}}\|_{L^{\infty}}
		+2^{\frac{5k}{2}}\| x^{2}v \|_{L^{2}} \Big) \\
		&\lesssim
		\sum_{\bt^{-\frac{1}{2}}\le2^{k}}
		t^{-2}2^{-k}2^{(-2\alpha+4)k}\| \widehat{v_{k}}\|_{L^{\infty}} +\sum_{k \in [Non-Stat]}
		t^{-2}2^{(-2\alpha+\frac{7}{2})k}
		\| x^{2}v \|_{L^{2}}.
		\end{align*}
		The first sum is bounded by $t^{-\frac32}\| \bxi^{4-2\alpha}\widehat{v}\|_{L^{\infty}} $. The second sum can be estimated case by case w.r.t. $\al$ as follows:
		\begin{align*}
		\sum_{k \in [Non-Stat]} &t^{-2}2^{(-2\alpha+\frac{7}{2})k}\| x^{2}v \|_{L^{2}}\\
		&\lesssim \left\{\begin{array}{ll}
		t^{-2+(-2\alpha+\frac{7}{2})	\frac1{2N_0(2-\alpha)}}, &\;\textrm{if}\; 1 < \alpha < \frac{7}{4}, \\
		t^{-2}\ln(1+t),&\; \textrm{if}\; \alpha = \frac{7}{4}, \\
		t^{-2+\frac{1}{2}(2\alpha-\frac{7}{2})},&\;\textrm{if}\; \frac{7}{4}< \alpha < 2
		\end{array}\right\} \| x^{2}v \|_{L^{2}}\\
		&\lesssim t^{-\frac74}\| x^{2}v \|_{L^{2}}.
		\end{align*}
		Lastly, we see that    	
		\begin{align*}
		\mathcal C  &\lesssim
		t^{-2}
		\sum_{k \in [Non-Stat]}
		\Big(
		2^{(-2\alpha+3)k}\|\widehat{v_{k}}\|_{L^{\infty}}
		+2^{(-2\alpha+2)k}2^{\frac{3k}{2}}\| x^{2}v\|_{L^{2}}
		\Big).
		\end{align*}	
		This can be treated similarly to $\mathcal B$. In conclusion, we have the bound for non-stationary case: \\
		\begin{align*}
		\sum_{k \in [Non-Stat]}I_k(t,x) \lesssim t^{-\frac32}\|\bxi^{4-2\al}\widehat{v}\|_{L^\infty} + t^{-\frac74}\|x^2 v\|_{L^2}.
		\end{align*}
		
		Now, we move onto stationary case: $$ k \in [Stat] := \{k : \bt^{-\frac{1}{2}}\le2^{k}\le \bt^{\frac1{2N_0(2-\alpha)}}, \left\vert k-k_{0}\right\vert \le 4\}.$$	Let $n_{0}$ denote the smallest integer such that $2^{n_{0}}\sim t^{-\frac{1}{2}} $.
		We further decompose the integral $I_{k}$ in $(\ref{Ikdyadic})$ dyadically around $\xi_{0}$ as
		\begin{equation}\label{Ikndyadic}
		\sum_{k \in [Stat]} I_{k}(t,x)
		\lesssim  \sum_{k \in [Stat]} \sum_{n\ge n_{0}}^{k_{0}+4} I_{k,n}(t,x),
		\end{equation}
		where
		$$
		I_{k,n}(t,x) =\left\vert \int_{\mathbb{R}^{3}} e^{i\phi_{t,x}(\xi)}\beta_{k}(\xi)\beta_n^{(n_0)}(\xi-\xi_{0}) \widehat{v}(\xi)d\xi \right\vert
		$$
		and
		\begin{align}\label{beta}
		\beta^{(n_{0})}_{n}(\xi-\xi_{0}) = \left\{
		\begin{array}{ll}
		\beta_{n}(\xi-\xi_{0}),& \textrm{if} \;\; n>n_{0},\\
		\varphi(\frac{\xi-\xi_{0}}{2^{n_{0}}}), &\textrm{if}\;\;  n=n_{0}
		\end{array}\right.
		\end{align}
		for a fixed $\varphi \in C_0^\infty$ which cut off near the origin.
		If $n = n_{0}$, it is easy to check	
		$$
		I_{k,n}(t,x) \lesssim \| \beta_{n_{0}}^{(n_{0})}\|_{L^{1}}
		\|\beta_{k}\widehat{v}\|_{L^{\infty}}
		\lesssim t^{-\frac{3}{2}}\|\widehat{v}\|_{L^{\infty}}.
		$$
		If $n > n_{0}$, Lemma \ref{lowbound} yields
		$$
		\vert\nabla\phi_{t,x}(\xi)\vert
		\gtrsim\
		\vert t\vert 2^{n}2^{(\alpha-2)k}
		\;\; \textrm{for} \;\;
		\vert\xi-\xi_{0}\vert\sim2^{n}.
		$$
		By taking integration by parts twice, we estimate
		\begin{align*}
		I_{k,n}
		\lesssim
		t^{-2} \sum_{0 \le \ell \le 2}
		2^{-(4- l)n}2^{2(2-\alpha)k}
		\| \nabla^{\ell}(\beta_{k}\beta_{n}^{(n_{0})}\widehat{v})
		\|_{L^{1}}
		\lesssim
		A_{k,n}+B_{k,n}+C_{k,n},
		\end{align*}
		where	
		\begin{align*}
		\begin{split}
		A_{k,n}&=
		t^{-2}2^{-4n}2^{2(2-\alpha)k} \|\beta_{k}\beta_{n}^{(n_{0})}\widehat{v}\|_{L^{1}}, \\
		B_{k,n}&=
		t^{-2}2^{-3n}2^{2(2-\alpha)k}
		\| \beta_{k}\beta_{n}^{(n_{0})}\nabla\widehat{v}\|_{L^{1}},\\
		C_{k,n}&=
		t^{-2}2^{-2n}2^{2(2-\alpha)k}
		\| \beta_{k}\beta_{n}^{(n_{0})}\nabla^{2}\widehat{v} \|_{L^{1}}.
		\end{split}
		\end{align*}
		We bound $\sum A_{k, n}$ as
		\begin{align*}
		\sum_{k \in [Stat]}
		\sum_{n> n_{0}}^{k+4}
		A_{k,n}
		&\lesssim
		t^{-2}
		\sum_{\vert k-k_{0}\vert\leq4}
		\sum_{n> n_{0}}^{k+4}
		2^{-4n}2^{2(2-\alpha)k} \|\beta_n^{(n_0)}\|_{L^{1}}
		\|\widehat{v_{k}}\|_{L^{\infty}}  \\
		&\lesssim
		t^{-2}t^{\frac{1}{2}} \| \bxi^{4-2\alpha}\widehat{v} \|_{L^{\infty}}.
		\end{align*}
		For $\sum B_{k, n}$, let us invoke that $2^{k}\le\bt^{\frac{1}{2N_0(2-\alpha)}}$. Then by using \eqref{ineq:L6x2f} we estimate
		\begin{align*}
		\sum_{k \in [Stat]}
		\sum_{n> n_{0}}^{k+4}
		B_{k,n}
		&\lesssim
		t^{-2}
		\sum_{\vert k-k_{0}\vert\leq4}
		2^{2(2-\alpha)k}
		\sum_{n> n_{0}}^{k_{0}+8}
		2^{-\frac{n}{2}}
		\| x^{2}v\|_{L^{2}} \\
		&\lesssim
		t^{-2}t^{\frac{1}{4}}t^{\frac1{N_0}}
		\| x^{2}v\|_{L^{2}}.
		\end{align*}
		The sum $\sum C_{k, n}$ is bounded as
		\begin{align*}
		\sum_{k \in [Stat]}
		\sum_{n> n_{0}}^{k+4}
		C_{k,n}
		&\lesssim
		t^{-2}
		\sum_{\vert k-k_{0}\vert\leq4}
		\sum_{n> n_{0}}^{k+4}
		2^{-2n}2^{2(2-\alpha)k} 2^{\frac{3n}{2}}
		\| x^{2}v\|_{L^{2}} \\
		&\lesssim
		t^{-2}t^{\frac{1}{4}}t^{\frac1{N_0}}
		\| x^{2}v\|_{L^{2}}.
		\end{align*}
		In conclusion, we obtain
		$$
		\sum_{k \in [Stat]}
		I_{k}(t,x)
		\lesssim
		t^{-\frac{3}{2}}
		\| \bxi^{4-2\alpha}\widehat{v}\|_{L^{\infty}}
		+t^{-\frac{3}{2}}t^{-(\frac14-\frac1{N_0})}
		\| x^{2}v\|_{L^{2}}.
		$$
		Comparing the bound obtained by non-stationary case with that by stationary case and \eqref{t-sob}, we get the desired result.

	\end{proof}

	\begin{cor}\label{ineq:W2}
		Let $N_0 > 4$ and $N\ge \frac72N_0(2-\alpha)+2\alpha-\frac{5}{2}$. Then we have
		$$
		\| e^{-it\vert\nabla\vert^{\alpha}}v\|_{W^{2,\infty}}
		\lesssim
		\bt^{-\frac{3}{2}}
		\| \bxi^{6-2\alpha}\widehat{v}\|_{L^{\infty}}
		+ \bt^{-\frac{3}{2}-(\frac14 - \frac1{N_0})}
		\big(\| v\|_{H^{N+2}}
		+ \| x^{2}v\|_{H^{2}} \big).
		$$
	\end{cor}
	
	\begin{proof}
		By definition we estimate
		\begin{align*}
		&	\| e^{-it\vert\nabla\vert^{\alpha}}v
		\|_{W^{2,\infty}} = \sum_{0 \le \ell \le2}
		\| e^{-it\vert\nabla\vert^{\alpha}}\nabla^\ell v\|_{L^\infty} \\
		&\lesssim
		\sum_{0 \le \ell \le2}
		\bt^{-\frac{3}{2}} {\| \bxi^{4-2\alpha}\widehat{\nabla^\ell v}
			\|_{L^{\infty}}}
		+ \bt^{-\frac{3}{2}-(\frac14-\frac1{N_0})}
		(\| x^{2}(\nabla^\ell v)\|_{L^{2}}+  \|\nabla^\ell v\|_{H^{N}}) \\
		&\lesssim
		\bt^{-\frac{3}{2}}
		\| \bxi^{6-2\alpha}\widehat{v}\|_{L^{\infty}}
		+ \bt^{-\frac{3}{2}-(\frac14-\frac1{N_0})}
		\big(\| x^{2}v\|_{H^{2}}+\|v\|_{H^{N+2}}  \big).
		\end{align*} 	
	\end{proof}

	\newcommand{\bs}{{\langle s \rangle}}
	
	\section{Weighted Energy Estimate}
	
	We prove Theorem \ref{main thm} in Section 4 through Section 5. All the estimates in both sections are implemented for the positive time.
	
	To investigate the asymptotic behavior of the solution $u$ at time infinity we introduce a control norm.
	\begin{dfn}\label{def:norm} 	
		Let $N, \delta_0, T > 0$. Then we define the norm $\|\cdot\|_{\Sigma_T}$ for the functions $u \in C([0,T]; H^N)$ with $x v \in C([0, T]; H^3), x^2 v \in C([0, T]; H^2)$ by
		\begin{align*}
		\| u \|_{\Sigma_{T}} :=\sup_{[0,T]}
		\Big[
		&\bt^{-\delta_0} \| u(t)\|_{H^{N}}
		+ \bt^{-\delta_0}\| xv(t)\|_{H^{3}}
		+ \bt^{-2\delta_0}\| x^{2}v(t)\|_{H^{2}} \\
		&\qquad\qquad + \| \bxi^{5}\widehat{v}(t)\|_{L_{\xi}^{\infty}}
		\Big].
		\end{align*}
	\end{dfn}
	As observed in the proof of Lemma \ref{moment}, the role of $\| xv(t)\|_{H^{3}}$ is crucial to control of $\| x^{2}v(t)\|_{H^{2}}$. So, it is taking part in the definition of $\|\cdot\|_{\Sigma_T}$.
	
	\begin{prop}[Weighted energy estimate]\label{scattering}
		Let $\frac53 < \al < 2$, $N \ge 10$. And let $\delta_0 > 0$ be such tat
		$$
		\max\left(0, \frac{17}{12} - \frac{5\al}6\right) \le \delta_0 \le \frac{1}{36}.
		$$
		Suppose that $u$ is the global solution to \eqref{equation}  such that $\| u \|_{\Sigma_{T}}\le \epsilon_{1}$  for some $\epsilon_1, T > 0$ and $u_0$ satisfies \eqref{initialvalue} with $\epsilon_0 \le \epsilon_1$. Then we have
		\begin{equation}\label{ineq:HN}
		\sup_{t\in[0,T]}\bt^{-\delta_0}
		\| u\|_{H^{N}}
		\le
		\epsilon_{0}+C\epsilon_{1}^{3},
		\end{equation}
		
		\begin{equation}\label{eq:eps1}
		\sup_{t\in[0,T]}\bt^{-\delta_0}
		\| xv\|_{H^{3}}
		\le
		\epsilon_{0}+C\epsilon_{1}^{3},
		\end{equation}
		
		\begin{equation}\label{eq:eps2}
		\sup_{t\in[0,T]}\bt^{-2\delta_0}
		\| x^{2}v\|_{H^{2}}
		\le
		\epsilon_{0}+C\epsilon_{1}^{3}.
		\end{equation}
	\end{prop}

	In this section, we prove Proposition \ref{scattering}. We emphasize again that $\frac{5}{3} < \alpha < 2$, $N \ge 10$ and $\max(0, \frac{17}{12} - \frac{5\al}{6} ) \le \delta_0 \le \frac{1}{36}$. If we choose $ N_0 = 6 $ in Corollary \ref{ineq:W2}, then since $N \ge 10$ we have
	\begin{align}\begin{aligned}\label{w2bound}
	\bt^\frac32&\| e^{-it\vert\nabla\vert^{\alpha}}v\|_{W^{2,\infty}}\lesssim
	\| \bxi^{5}\widehat{v}\|_{L^{\infty}}
	+\bt^{-2\delta_0}
	\big(\| x^{2}v\|_{H^{2}}+\|v\|_{H^{N}}  \big).
	\end{aligned}\end{align}
	Therefore if $\|u\|_{\Sigma_T} \le \epsilon_1$, then by the definition of $\Sigma_{T}$ and \eqref{w2bound} we have
	\begin{align}\label{w2-infty-bound}
	\| u(t)\|_{W^{2,\infty}}
	\lesssim \bt^{-\frac{3}{2}}\epsilon_1.
	\end{align}
	We will see  in Section \ref{j3} why the conditions $\al > \frac53$ and $\delta_0 \ge \frac{17}{12} - \frac{5\al}{6}$ are necessary.
	
	\begin{subsection}{Proof of \eqref{ineq:HN}}
		From \eqref{integral} and Leibniz rule it follows that for any $t > 0$
		\begin{align*}
		&\|u( t)\|_{H^N} \\
		&\le \|u_0\|_{H^N} + C\int_0^{ t} \left\| (|x|^{-1}*|u|^2)u\right\|_{H^N}\,ds\\
		&\le \|u_0\|_{H^N} + C\int_0^{ t} (\||x|^{-1}*|u|^2\|_{L^\infty}\|u\|_{H^N} +  \||x|^{-1}*|u|^2\|_{H_6^N}\|u\|_{L^3})\,ds.
		\end{align*}
		Using  the estimate $\||x|^{-1}*|u|^2\|_{L^\infty} \lesssim \|u\|_{L^2}\|u\|_{L^6}$ and Leibniz rule once more,
		\begin{align*}
		\|u\|_{H^N}
		&\le \|u_0\|_{H^N} + C\int_0^{ t} (\|u\|_{L^2}\|u\|_{L^6} + \|u\|_{L^3}^2)\|u\|_{H^N}\,ds\\
		&\le \|u_0\|_{H^N} + C\int_0^{ t} \|u\|_{L^2}\|u\|_{L^6}\|u\|_{H^N} \,ds\\
		&\le \|u_0\|_{H^N} + C\int_0^{ t} \|u\|_{L^2}^\frac43\|u\|_{L^\infty}^\frac23\|u\|_{H^N} \,ds\\
		&\le \epsilon_0 + C\epsilon_1^\frac53 \|u_0\|_{L^2}^\frac43\int_0^{ t}\bs^{-1+\delta_0}\,ds\\
		&\le \epsilon_0 + C\bt^{\delta_0}\epsilon_1^3.
		\end{align*}

	\end{subsection}

	\begin{subsection}{Proof of \eqref{eq:eps1} }
		In order to prove \eqref{eq:eps1}, we need to establish the following.
		\begin{lm}\label{lemma24}
			Let $u$ satisfy the condition of Proposition\ref{scattering}. Then we have
			\begin{equation}\label{ineq:infty}
			\| P_{k}(\vert u\vert^{2})(s)\|_{L^{\infty}}
			\lesssim
			\min(\langle 2^k\rangle^{-2}\bs^{-3}, 2^{3k})\epsilon_{1}^{2},
			\end{equation}
			
			\begin{equation}\label{ineq:L2}
			\| P_{k}(\vert u\vert^{2})(s)\|_{L^{2}}
			\lesssim
			\min( 2^{\frac{3k}{2}}, 2^{\frac{3k}{2}}\langle 2^k\rangle^{-5}) \epsilon_{1}^{2},
			\end{equation}
			
			\begin{equation}\label{ineq:xf}
			\|
			P_{k}(xv)(s)
			\|_{L^{2}}
			\lesssim
			\min( 2^{(2-\alpha)k}\bs^{(3-\alpha)\delta_0}, \langle2^{k}\rangle^{-3}\bs^{\delta_0}).
			\end{equation}
			
		\end{lm}
		\begin{proof}[Proof of Lemma \ref{lemma24}]
			
			We first show \eqref{ineq:infty}. Using the fact  
			$\| \mathcal{F}^{-1}(\beta_{k}\bxi^{-2}) \|_{L^{1}} $ 
			$\lesssim \langle 2^k\rangle^{-2}$ and also \eqref{w2-infty-bound}, we obtain
			\begin{align*}
			\| P_{k}(\vert u\vert^{2})\|_{L^{\infty}}
			&=\sup_{x}
			\big\vert
			\int e^{ix\cdot\xi}\beta_{k}(\xi)
			\widehat{\vert u\vert^{2}}(\xi)\,
			d\xi
			\big\vert \\
			&=\sup_{x}
			\vert
			\int e^{ix\cdot\xi}
			\beta_{k}(\xi)\bxi^{-2}
			\widehat{(1-\Delta)(\vert u\vert^{2})}(\xi)\,
			d\xi \vert \\
			&\lesssim
			\| \mathcal{F}^{-1}(\beta_{k}\bxi^{-2})
			*(1-\Delta)(\vert u\vert^{2})\|_{L^{\infty}}  \\
			&\lesssim
			\| \mathcal{F}^{-1}(\beta_{k}\bxi^{-2})
			\|_{L^{1}}
			\|(1- \Delta) (|u|^2) \|_{L^{\infty}} \\
			&\lesssim
			\langle 2^k\rangle^{-2} \| u\|_{W^{2,\infty}}^{2}
			\lesssim \langle 2^k\rangle^{-2}\bs^{-3}\epsilon_1^2.
			\end{align*}
			On the other hand, we see that
			\begin{align*}
			\| P_{k}(\vert u\vert^{2})\|_{L^{\infty}} \lesssim \|\beta_k \, \widehat{|u|^2}\|_{L^1}
			\lesssim
			\|\beta_{k}\|_{L^{1}}
			\|\widehat{|u|^2}\|_{L^{\infty}}
			\lesssim 2^{3k}
			\|u\|_{L^2}^{2} \lesssim 2^{3k}\epsilon_1^2.
			\end{align*}	
			
			We consider \eqref{ineq:L2}. One can easily check
			\begin{align*}
			\| P_{k}(\vert u\vert^{2})\|_{L^{2}}
			\lesssim \|\beta_{k}\|_{L^{2}}\|\widehat{|u|^2}\|_{L^{\infty}} \lesssim
			2^{\frac{3k}{2}}\epsilon_{1}^{2}.
			\end{align*}
			Plancherel's theorem gives us
			\begin{align*}
			\| P_{k}(\vert u\vert^{2})\|_{L^{2}}^{2}
			&\lesssim
			\int_{\mathbb{R}^{3}}\beta_{k}(\xi)\vert \widehat{u}*\widehat{\bar{u}}\vert^{2}(\xi)d\xi  \\
			&\lesssim
			\int_{\mathbb{R}^{3}}
			\beta_{k}(\xi)\|\bxi^{5}\widehat{u}\|_{L^{\infty}}^{4} \big\vert\int_{\mathbb{R}^{3}}
			(1+\vert y\vert)^{-5}(1+\vert y-\xi\vert)^{-5}dy
			\big\vert^{2}
			d\xi \\
			&\lesssim
			\epsilon_{1}^{4}
			\int_{\mathbb{R}^{3}}
			\beta_{k}(\xi)
			\vert \int_{\mathbb{R}^{3}}
			(1+\vert y\vert)^{-5}(1+\vert y-\xi\vert)^{-5}dy \vert^{2}d\xi \\
			&\lesssim
			\epsilon_{1}^{4}
			2^{3k}\langle2^{k}\rangle^{-10}.
			\end{align*}
			
			%

			For  \eqref{ineq:xf} we use Bernstein's inequality to get
			\begin{align*}
			\| P_{k}(xv)\|_{L^{2}}
			\lesssim
			2^{k(2-\alpha)}
			\| xv\|_{L^{\frac{6}{7-2\alpha}}}.
			\end{align*}	
			Using H\"older's inequality, for any $R > 0$ we have
			\begin{align*}
			\| xv\|_{L^{\frac{6}{7-2\alpha}}}^{\frac{6}{7-2\alpha}}
			& \lesssim \sum_{j = 1}^3\left(\int_{\vert x\vert\le R}
			\vert x_jf(x)\vert^{\frac{6}{7-2\alpha}}dx
			+\int_{\vert x\vert\ge R}
			\vert x_jf(x)\vert^{\frac{6}{7-2\alpha}}dx\right)  \\
			&\lesssim
			R^{\frac{6(2-\alpha)}{7-2\alpha}}\| xv \|_{L^{2}}^{\frac{6}{7-2\alpha}}
			+R^{\frac{6(1-\alpha)}{7-2\alpha}}
			\| x^{2}v\|_{L^{2}}^{\frac{6}{7-2\alpha}}.
			\end{align*}
			The optimization by taking $R = \| x^{2}v\|\| xv \|^{-1}$ leads us to the estimate
			\begin{align*}
			\| xv\|_{L^{\frac{6}{7-2\alpha}}}  \lesssim
			\| x^{2}v\|_{L^{2}}^{2-\alpha}\cdot
			\| xv \|_{L^{2}}^{\alpha-1}   \lesssim
			\bs^{(3-\alpha)\delta_0}\epsilon_{1}.
			\end{align*}
			Also, we estimate in another manner
			\begin{align*}
			\| P_{k}(xv)\|_{L^{2}}
			\lesssim
			\langle 2^{k}\rangle^{-3}
			\| xv\|_{H^{3}}
			\lesssim
			\langle2^{k}\rangle^{-3}
			\bs^{\delta_0}\epsilon_{1}.
			\end{align*}
			This completes the proof of Lemma \ref{lemma24}.
			
		\end{proof}

		Note that Plancherel's theorem yields
		\begin{equation}
		\| xv \|_{H^{3}}
		\sim \|
		\left\langle\xi\right\rangle^{3}
		\widehat{xv}
		\|_{L^{2}} \\
		=\|
		\left\langle\xi\right\rangle^{3}
		\nabla_{\xi}\widehat{v}
		\|_{L^{2}}.
		\end{equation}
		Differentiating the both sides of \eqref{integraleqn} w.r.t. $\xi$, we have
		$$
		(\nabla\widehat{v})(t,\xi)
		= (\nabla \widehat{u_0})(\xi) + \mathbf I_{1}(t,\xi) + \mathbf I_{2}(t,\xi), $$
		$$    \mathbf I_{1}(t,\xi)
		= ic_0\int_{0}^{t}\int_{\mathbb{R}^{3}}
		e^{is\phi_{\alpha}(\xi,\eta)}
		\vert\eta\vert^{-2}
		\widehat{\vert u\vert^{2}}(\eta)
		(\nabla\widehat{v})(\xi-\eta)\, d\eta ds, $$
		$$   \mathbf I_{2}(t,\xi)
		= ic_0\int_{0}^{t} s \int_{\mathbb{R}^{3}} \mathbf m(\xi,\eta)
		e^{is\phi_{\alpha}(\xi,\eta)}
		\vert\eta\vert^{-2}
		\widehat{\vert u\vert^{2}}(\eta)\widehat{v}(\xi-\eta)\, d\eta ds,
		$$
		where
		\begin{equation}\label{m}
		\mathbf m(\xi,\eta) := \nabla_{\xi}[\phi_{\alpha}(\xi,\eta) ]
		=
		\alpha \left(\vert\xi\vert^{\alpha-2}\xi
		- \vert\xi-\eta\vert^{\alpha-2}(\xi - \eta)
		\right).
		\end{equation}
		Then the estimate \eqref{eq:eps1} follows from
		$$
		\|\bxi^3 \mathbf I_1(t, \xi)\|_{L^2} + \|\bxi^3 \mathbf I_2(t, \xi)\|_{L^2} \lesssim \bt^{\delta_0}\epsilon_1^3.
		$$
		We consider this in the next two subsections.
		
		\begin{subsubsection}{Estimate of ${\mathbf I}_{1}(s,\xi)$}
			We rewrite $\mathbf I_1$ as
			\begin{align*}
			{\mathbf I}_{1}(t,\xi)
			&=ic_0\int_{0}^{t}\int_{\mathbb{R}^{3}}
			e^{is\vert\xi\vert^{\alpha}}
			\vert\eta\vert^{-2}
			\widehat{\vert u\vert^{2}}(\eta)
			e^{-is\vert\xi-\eta\vert^{\alpha}}
			(\nabla\widehat{v})(\xi-\eta)
			d\eta ds  \\
			& = \lambda \int_{0}^{t}
			\mathcal{F}\big\{e^{is\vert\nabla\vert^{\alpha}}
			(\vert\cdot\vert^{-1}*\vert u\vert^{2})
			(e^{-is\vert\nabla\vert^{\alpha}}xv)\big\} ds.
			\end{align*}
			From the Sobolev estimates and $\||x|^{-1}*|u|^2\|_{\dot H^3} \lesssim \|\nabla |u|^2\|_{L^2}$, it follows that
			\begin{align}\begin{aligned}\label{xvH2}
			\|
			\left\langle\xi\right\rangle^{3}
			&{\mathbf I}_{1}(t,\xi)
			\|_{L^{2}}\\
			&\lesssim
			\int_{0}^{t}
			\|
			e^{is\vert\nabla\vert^{\alpha}}
			(\vert\cdot\vert^{-1}*\vert u\vert^{2})
			e^{is\vert\nabla\vert^{\alpha}}(xv)
			\|_{H^{3}}\,ds \\
			&\lesssim
			\int_0^t \|\vert\cdot\vert^{-1}*\vert u\vert^{2}\|_{L^\infty}\|x v\|_{H^3}\,ds + \int_0^t \|\vert\cdot\vert^{-1}*\vert u\vert^{2}\|_{\dot H^3}\|xv\|_{L^\infty}\,ds\\
			&\lesssim
			\int_{0}^{t}
			\| u\|_{L^{2}}
			\| u\|_{L^{6}}
			\| xv\|_{H^{3}}
			ds + \int_0^t \|\nabla |u|^2\|_{L^2}\|x v\|_{H^3}\, ds.
			\end{aligned}\end{align}
			Since $\| u\|_{\Sigma_{T}} \le \epsilon_{1}$, \eqref{w2-infty-bound} implies $\|u(s)\|_{L^\infty} \lesssim \bs^{-\frac32}\epsilon_1$, and by mass conservation $\|u(s)\|_{L^2} = \|u_0\|_{L^2} \le \epsilon_1$. Hence we have
			\begin{align}\begin{aligned}\label{ineq:L6}
			\| u\|_{L^{6}} &\le \|u\|_{L^2}^\frac13\|u\|_{L^\infty}^\frac23 \lesssim \bs^{-1}\epsilon_{1}\\
			\| \nabla |u|^2\|_{L^2} &\le \|u\|_{L^\infty}\|\nabla u\|_{L^2} \lesssim \bs^{-\frac32 + \delta_0}\epsilon_1^2.
			\end{aligned}\end{align}
			These give us
			\begin{align*}
			\|
			\left\langle\xi\right\rangle^{3}
			{\mathbf I}_{1}(t,\xi)
			\|_{L^{2}}
			\lesssim
			\epsilon_{1}^3\int_{0}^{t}
			\bs^{-1}\bs^{\delta_0}\,
			ds
			\lesssim
			\bt^{\delta_0}\epsilon_{1}^{3}.
			\end{align*} 	
		\end{subsubsection}

		\begin{subsubsection}{Estimate of ${\mathbf I}_{2}(t,\xi)$}\label{sec:i2}
			By the dyadic decomposition, we have
			\begin{align*}
			&\big\|
			\left\langle\xi\right\rangle^{3}
			{\mathbf I}_{2}(t,\xi)
			\big\|_{L^{2}}  \\
			\qquad
			&\lesssim
			\int_{0}^{t} s
			\sum_{k_{1},k_{2},k \in \mathbb{Z}}
			\langle 2^{k}\rangle^{3}
			\bigg\|
			\int_{\mathbb{R}^{3}}
			e^{is\phi_{\alpha}(\xi,\eta)}
			\mathbf m(\xi,\eta)\\
			&\qquad\quad
			\times\beta_{k}(\xi)
			\beta_{k_{1}}(\eta-\xi)
			\beta_{k_{2}}(\eta)
			\vert\eta\vert^{-2}
			\widehat{\widetilde P_{k_{2}}(\vert u\vert^{2})}(\eta)
			\widehat{\widetilde P_{k_1}(v)}(\xi-\eta)
			d\eta
			\bigg\|_{L^{2}_{\xi}}\,
			ds \\
			&\lesssim
			\int_{0}^{t}
			s
			\sum_{k_{1},k_{2},k\in\mathbb{Z}}
			{I}_{2}^{k, k_1, k_2}(s)\,
			ds,
			\end{align*}
			where
			$$
			{I}_{2}^{k, k_1, k_2}(s)=
			\langle 2^{k}\rangle^{3}
			\left\|
			\int
			\mathbf m^{k, k_1, k_2}(\xi,\eta)
			\widehat{\widetilde P_{k_{2}}(\vert u\vert^{2})}(\eta) e^{-is|\xi-\eta|^\al}
			\widehat{\widetilde P_{k_1}(v)}(\xi-\eta)
			d\eta
			\right\|_{L^{2}_{\xi}}
			$$
			and
			\begin{align}\label{mkkk}
			\mathbf m^{k,k_{1},k_{2}}(\xi,\eta)
			=  \mathbf m(\xi,\eta)
			\vert\eta\vert^{-2}
			\beta_{k}(\xi)\beta_{k_{1}}(\eta-\xi)\beta_{k_{2}}(\eta).
			\end{align}
			Thus it suffices to prove that
			\begin{equation}\label{i2kkk}
			\sum_{k_{1},k_{2},k\in\mathbb{Z}}
			{ I}_{2}^{k,k_{1},k_{2}}(s)
			\lesssim
			\bs^{-2+\delta_0}\epsilon_1^3.
			\end{equation}

			In view of the Fourier support condition we can divide the sum over $k, k_1, k_2$ into three possible cases:
			$$
			\sum_{k_{1},k_{2},k\in\mathbb{Z}}
			{I}_{2}^{k,k_{1},k_{2}}(s)
			\lesssim
			\big(
			\sum_{k\sim k_{1}\ge k_{2}}
			+\sum_{k\ll k_{1}\sim k_{2}}
			+\sum_{k\sim k_{2}\gg k_{1}}
			\big)
			{ I}_{2}^{k,k_{1},k_{2}}(s).
			$$
			
			In the first two cases, we use the following inequality: If $k \lesssim k_1$, then for any positive integers $\ell_1, \ell_2$
			\begin{align}\begin{aligned}\label{ineq:pt}
			&\left\vert
			\nabla_{\xi}^{\ell_1}\nabla_{\eta}^{\ell_2}
			\mathbf m^{k,k_{1},k_{2}}(\xi,\eta)
			\right\vert\\
			&  \lesssim
			\max(2^{k},2^{k_{1}},2^{k_{2}})^{\alpha-2}
			2^{-k_{2}}
			2^{-k_{1} \ell_1}
			2^{-k_{2} \ell_2}
			\widetilde \beta_k(\xi)\widetilde \beta_{k_1}(\eta-\xi)\widetilde \beta_{k_2}(\eta).
			\end{aligned}\end{align}
			This can be readily shown by direct calculation. We omit the details.
			This leads us to the bound for Lemma \ref{LemmaB}.
			\begin{equation}\label{lemmaC}
			\left\|
			\iint_{\mathbb{R}^{3}\times\mathbb{R}^{3}}
			\mathbf m^{k,k_{1},k_{2}}(\xi,\eta)
			e^{ix\cdot\xi}e^{iy\cdot\eta}\,
			d\xi d\eta
			\right\|_{L_{x,y}^{1}}
			\lesssim
			\max(2^{k},2^{k_{1}},2^{k_{2}})^{\alpha-2}
			2^{-k_{2}}.
			\end{equation}
			
			In the last case $(k_{1}\ll k\sim k_{2})$,
			even though $\mathbf m^{k,k_{1},k_{2}}$ fails to satisfy \eqref{ineq:pt},
			the estimate \eqref{lemmaC}) turns out to hold.
			Indeed, by making change of variables, one can verify that
			\begin{align*}
			&  \left\|
			\iint_{\mathbb{R}^{3}\times\mathbb{R}^{3}}
			\mathbf m^{k,k_{1},k_{2}}(\xi,\eta)
			e^{ix\cdot\xi}e^{iy\cdot\eta}\,
			d\xi d\eta
			\right\|_{L_{x,y}^{1}}\\
			&\qquad\qquad\qquad\qquad =
			\left\|
			\iint_{\mathbb{R}^{3}\times\mathbb{R}^{3}}
			\widetilde{\mathbf m}^{k,k_{1},k_{2}}(\xi,\eta)
			e^{ix\cdot\xi}e^{iy\cdot\eta}\,
			d\xi d\eta
			\right\|_{L_{x,y}^{1}},
			\end{align*}
			where
			\[
			\widetilde{ \mathbf m}^{k,k_{1},k_{2}}(\xi,\eta)
			=\alpha
			\left(
			\vert\xi\vert^{\alpha-2}\xi
			-\vert\eta\vert^{\alpha-2}\eta
			\right)
			\vert\xi-\eta\vert^{-2}
			\beta_{k}(\xi)\beta_{k_{1}}(\eta)\beta_{k_{2}}(\xi-\eta).
			\]
			Then $\widetilde{\mathbf m}^{k,k_{1},k_{2}}(\xi,\eta)$ satisfies that
			\begin{align*}
			&\left\vert
			\nabla_{\xi}^{l_1}\nabla_{\eta}^{l_2}
			\widetilde{\mathbf m}^{k,k_{1},k_{2}}(\xi,\eta)
			\right\vert\\
			&\lesssim
			\max(2^{k},2^{k_{1}},2^{k_{2}})^{\alpha-2}
			2^{-k_{2}}
			2^{-k l_1}
			2^{-k_1 l_2}
			\widetilde \beta_k(\xi)\widetilde \beta_{k_1}(\eta)\widetilde \beta_{k_2}(\xi-\eta)
			\end{align*}
			and hence the claim follows.
			
			({\it Case:} $k \sim k_{1}\ge k_{2}$)
			Applying the Lemma $\ref{LemmaB}$ whose $A(\mathbf m^{k, k_1, k_2})$ is the RHS of \eqref{lemmaC} and then using \eqref{ineq:infty}, we estimate
			\begin{align*}
			& \sum_{k_{2}\lesssim k\sim k_{1} }
			{I}_{2}^{k, k_1, k_2}(s)\\
			&\lesssim
			\sum_{k_{2}\lesssim k\sim k_{1} }
			\langle2^{k}\rangle^{3}
			\max(2^{k},2^{k_{1}})^{\alpha-2}2^{-k_{2}}
			\| \widetilde P_{k_{2}}(\vert u\vert^{2})\|_{L^{\infty}}
			\| \widetilde P_{k_{1}}(v)\|_{L^{2}} \\
			&\lesssim
			\sum_{k_{2}\lesssim k\sim k_{1} }
			\langle2^{k}\rangle^{3}
			\max(2^{k},2^{k_{1}})^{\alpha-2}
			2^{-k_{2}} \\
			&\qquad\times
			\min\left(
			\langle 2^{k_{2}}\rangle^{-2}\bs^{-3},
			2^{3k_{2}}
			\right)
			\langle 2^{k_{1}}\rangle^{-5}
			2^{\frac{3k_{1}}{2}}\epsilon_{1}^{3} \\
			&\lesssim
			\big(
			\sum_{2^{k_{2}}\le\bs^{-1}}
			2^{2k_{2}}
			+\sum_{2^{k_{2}}\ge\bs^{-1}}
			\bs^{-3}2^{-k_{2}}
			\big)
			\sum_{k}
			\langle 2^{k}\rangle^{-2}2^{k(\alpha-\frac{1}{2})}\epsilon_1^3 \\
			&\lesssim
			\bs^{-2+\delta_0}\epsilon_1^3.
			\end{align*}
			
			({\it Case}: $k\ll k_{1}\sim k_{2}$) The sum over $k$ with $2^{k}\le\bs^{-2}$ can be easily dealt with using the pointwise bound \eqref{ineq:pt} as follows.
			\begin{align*}
			&\sum_{ \substack{ k\ll k_{1}\sim k_{2} \\
					2^{k}\le\bs^{-2} }}
			{ I}_{2}^{k,k_{1},k_{2}}(s)\\
			&\lesssim
			\sum_{ \substack{ k\ll k_{1}\sim k_{2} \\
					2^{k}\le\bs^{-2} }}
			\langle2^{k}\rangle^{3}
			\| \mathbf m^{k,k_{1},k_{2}}(\xi,\eta)\|_{L^{\infty}}
			\| \beta_{k}\|_{L^{2}}
			\|\widehat{\widetilde P_{k_{2}}(\vert u\vert^{2})}\|_{L^{2}}
			\|\widehat{\widetilde P_{k_{1}}(v)}\|_{L^{2}} \\
			&\lesssim
			\sum_{ \substack{ k\ll k_{1}\sim k_{2} \\
					2^{k}\le\bs^{-2} }}
			\langle2^{k}\rangle^{3}
			2^{k_{1}(\alpha-2)}2^{-k_{2}}2^{\frac{3k}{2}}
			\langle 2^{k_{2}}\rangle^{-5}2^{\frac{3k_{2}}{2}}
			\langle2^{k_{1}}\rangle^{-5}2^{\frac{3k_{1}}{2}}
			\epsilon_{1}^{3}  \\
			&\lesssim
			\sum_{ \substack{k\ll k_{2} \\ 2^{k}\le\bs^{-2}}}
			\langle 2^{k}\rangle^{3}
			2^{\frac{3k}{2}} 2^{k_{2}\alpha}
			\langle2^{k_{2}}\rangle^{-10}\epsilon_{1}^{3}  \\
			&\lesssim
			\sum_{ 2^{k}\le\bs^{-2}}
			2^{\frac{3k}{2}} \epsilon_{1}^{3}
			\sum_{k_{2}}
			2^{k_{2}\alpha}\langle2^{k_{2}}\rangle^{-7} \\
			&\lesssim\bs^{-3} \epsilon_{1}^{3}.
			\end{align*}
			
			The remaining case can be estimated by applying Lemma \ref{LemmaB} with the bound $A(\mathbf m^{k, k_1, k_2} )$ from \eqref{lemmaC}, as follows.
			\begin{align*}
			&\sum_{ \substack{ k\ll k_{1}\sim k_{2} \\
					2^{k}\ge\bs^{-2} }}
			{ I}_{2}^{k, k_1, k_2}(s)
			\lesssim
			\sum_{ \substack{ k\ll k_{1}\sim k_{2} \\
					2^{k}\ge\bs^{-2} }}
			\langle2^{k}\rangle^{3}2^{k_{1}(\alpha-2)}
			2^{-k_{2}}
			\| \widetilde P_{k_{2}}(\vert u\vert^{2})\|_{L^{\infty}}
			\| \widetilde P_{k_{1}}(v)\|_{L^{2}} \\
			&\lesssim
			\sum_{ \substack{ k\ll  k_{2} \\
					2^{k}\ge\bs^{-2} }}
			\langle 2^{k}\rangle^{3}
			2^{k_{2}(\alpha-\frac{3}{2})}
			\min( \langle 2^{k_{2}}\rangle^{-2}\bs^{-3}, \ 2^{3k_{2}} )
			\langle 2^{k_{2}}\rangle^{-5}\epsilon_{1}^{3} \\
			&\lesssim
			\sum_{2^{k_{2}}\le\bs^{-1}}
			2^{k_{2}(\alpha+\frac{3}{2})}\epsilon_{1}^{3}
			\sum_{\bs^{-2}\le 2^{k}\le\bs^{-1}}1  \\
			&\qquad +\sum_{2^{k_{2}}\ge\bs^{-1}}
			\bs^{-3}
			2^{k_{2}(\alpha-\frac{3}{2})}
			\langle 2^{k_{2}}\rangle^{-7}
			\sum_{\substack{k\le k_{2} \\
					2^{k}\ge\bs^{-2}     }}
			\langle 2^{k}\rangle^{3}
			\epsilon_{1}^{3}     \\
			&\lesssim
			\bs^{-\alpha-\frac{3}{2}+\delta_0}
			\epsilon_{1}^{3}
			+\bs^{-\frac{5}{2}}
			\epsilon_{1}^{3}
			\sum_{2^{k_{2}}\ge\bs^{-1}}
			2^{k_{2}(\alpha-1)}
			\langle 2^{k_{2}}\rangle^{-3}
			\sum_{\substack{k\le k_{2} \\
					2^{k}\ge\bs^{-2}     }}
			\langle 2^{k}\rangle^{-1} \\
			&\lesssim
			\bs^{-2+\delta_0}\epsilon_1^3.
			\end{align*}

			({\it Case: }$k_{1}\ll k\sim k_{2}$) In this case we use Lemma \ref{LemmaB} with $A({\mathbf m}^{k, k_1, k_2})$ of \eqref{lemmaC} and get that
			\begin{align*}
			&  \sum_{k_{1}\le k \sim k_{1}}
			{I}_{2}^{k, k_1, k_2}(s)\\
			&\lesssim
			\sum_{k_{1}\ll k\sim k_{2} }
			\langle2^{k}\rangle^{3}2^{k(\alpha-1)}2^{-2k_{2}}
			\| \widetilde P_{k_{2}}(\vert u\vert^{2})\|_{L^{\infty}}
			\| \widetilde P_{k_{1}}(v)\|_{L^{2}} \\
			&\lesssim
			\sum_{k_{1}\ll k}
			\langle2^{k}\rangle^{3}2^{k(\alpha-3)}
			\min( \langle2^{k}\rangle^{-2}\bs^{-3}, \ 2^{3k} )
			\langle2^{k_{1}}\rangle^{-5}2^{\frac{3k_{1}}{2}}
			\epsilon_{1}^{3} \\
			&\lesssim
			\epsilon_1^3\sum_{2^{k}\le\bs^{-1}}
			2^{\alpha k}
			\sum_{k_{1}\le k}
			2^{\frac{3}{2}k_{1}}
			+ \epsilon_1^3\sum_{2^{k}\ge\bs^{-1}}
			\bs^{-3}2^{k(\alpha-2)}
			\sum_{k_{1}\le k}
			\langle2^{k_{1}}\rangle^{-5}2^{\frac{k_{1}}{2}} \\
			&\lesssim
			\bs^{-\alpha-1}\epsilon_1^3.
			\end{align*}
			
		\end{subsubsection}
		
	\end{subsection}

	\begin{subsection}{Proof of \eqref{eq:eps2} }
		By Plancherel's theorem we have
		$$\| x^2v \|_{H^{2}}
		\sim \| \left\langle\xi\right\rangle^{2}\widehat{x^{2}v}\|_{L^{2}}
		=\| \left\langle\xi\right\rangle^{2}\nabla^2\widehat{v}\|_{L^{2}}. $$
		Then the second derivative of $\widehat{v}$ can be written as
		$$
		\nabla^2\widehat{v}(t,\xi)
		={\mathbf J_{1}}(t,\xi)
		+{\mathbf J_{2}}(t,\xi)
		+{\mathbf J_{3}}(t,\xi)
		+{\mathbf J_{4}}(t,\xi).
		$$
		Here $\mathbf J_i$ are defined by
		\begin{align*}
		{\mathbf J_{1}}(t,\xi)
		&=
		ic_0  \int_{0}^{t}\int_{\mathbb{R}^{3}}
		e^{is\phi_\al(\xi, \eta)}
		\vert\eta\vert^{-2}
		\widehat{\vert u\vert^{2}}(\eta)
		(\nabla^2\widehat{v})(\xi-\eta)
		d\eta ds,  \\
		{\mathbf J_{2}}(t,\xi)
		&=
		2c_0\int_{0}^{t}
		s
		\int_{\mathbb{R}^{3}}
		\mathbf m(\xi, \eta) \otimes
		[e^{is\phi_\al(\xi, \eta)}
		\vert\eta\vert^{-2}
		\widehat{\vert u\vert^{2}}(\eta)
		(\nabla\widehat{v})(\xi-\eta)]
		d\eta
		ds,  \\
		{\mathbf J_{3}}(t,\xi)
		&=
		c_0 \int_{0}^{t}
		s
		\int_{\mathbb{R}^{3}}
		[\nabla_\xi \mathbf m(\xi, \eta)]
		e^{is\phi_\al(\xi, \eta)}
		\vert\eta\vert^{-2}
		\widehat{\vert u\vert^{2}}(\eta)
		\widehat{v}(\xi-\eta)
		d\eta
		ds, \\
		{\mathbf J_{4}}(t,\xi)
		&=
		-ic_0\int_{0}^{t}
		s^{2}
		\int_{\mathbb{R}^{3}}
		\mathbf m(\xi, \eta)
		\otimes  \mathbf m(\xi, \eta)
		e^{is\phi_\al(\xi, \eta)}
		\vert\eta\vert^{-2}
		\widehat{\vert u\vert^{2}}(\eta)
		\widehat{v}(\xi-\eta)
		d\eta ds,
		\end{align*}
		where $\mathbf m$ is defined as in \eqref{m}.
		
		\begin{subsubsection}{Estimate of ${\mathbf J_{1}}$}
			By \eqref{xvH2} and \eqref{ineq:L6} one can readily get
			\begin{align*}
			\|
			\langle\xi\rangle^{2}
			{\mathbf J_{1}}(t,\xi)
			\|_{L^{2}}    &\lesssim
			\int_{0}^{t}
			\|
			e^{is\vert\nabla\vert^{\alpha}}
			(\vert\cdot\vert^{-1}*\vert u\vert^{2})
			e^{is\vert\nabla\vert^{\alpha}}(x^{2}v)
			\|_{H^{2}}\,ds \\
			&\lesssim
			\bt^{2\delta_0}\epsilon_{1}^{3}.
			\end{align*}
		\end{subsubsection}
		
		\begin{subsubsection}{Estimate of ${\mathbf J_{2}}$}
			
			By dyadic decomposition in $\xi,\eta$ and $\xi-\eta$, we write
			\begin{align*}
			&\langle\xi\rangle^{2}
			{\mathbf J_{2}}(t,\xi)\\
			&= 2c_0 \sum_{k,k_{1},k_{2}\in\mathbb{Z}}
			\langle\xi\rangle^{2}
			\int_{0}^{t}
			s
			\int_{\mathbb{R}^{3}}
			\mathbf m_2^{k,k_{1},k_{2}}(\xi,\eta) \\
			&\qquad\qquad\qquad\qquad\qquad\qquad \otimes [e^{is\phi_\al(\xi, \eta)}
			\widehat{\widetilde P_{k_{2}}(\vert u\vert^{2})}(\eta)
			\widehat{\widetilde P_{k_{1}}(xv)}(\xi-\eta)]
			d\eta
			ds,
			\end{align*}
			where $\mathbf m_2^{k, k_1, k_2} = \mathbf m^{k, k_1, k_2}$ as in \eqref{mkkk}.
			So, it suffices to show that
			\[
			\sum_{k,k_{1},k_{2}\in\mathbb{Z}}
			\langle2^{k}\rangle^{2}
			{J_{2}}^{k,k_{1},k_{2}}(s)
			\lesssim \bs^{-2+2\delta_0}\epsilon_{1}^{3},
			\]
			where
			\[
			{J_{2}}^{k,k_{1},k_{2}}(s)
			=\bigg\|
			\int_{\mathbb{R}^{3}}
			\mathbf m_2^{k,k_{1},k_{2}}(\xi,\eta)\otimes
			[\widehat{\widetilde P_{k_{2}}\vert u\vert^{2}}(\eta)
			e^{-is|\xi-\eta|^\al}\widehat{\widetilde P_{k_{1}}(xv)}(\xi-\eta)]
			d\eta
			\bigg\|_{L_{\xi}^{2}}.
			\]
			We then divide the sum into three possible cases as above.
			
			({\it Case:} $k \sim k_{1}\ge k_{2}$)
			We proceed with the same method as in Section \ref{sec:i2}.
			Using Lemma \ref{LemmaB} with $A(\mathbf m_2^{k,k_1,k_2})$ of \eqref{lemmaC} and the Lemma \ref{lemma24}, we have
			\begin{align*}
			&\sum_{k\sim k_{1}\ge k_{2}}
			\langle2^{k}\rangle^{2}
			{ J_{2}}^{k,k_{1},k_{2}}(s) 
			\lesssim
			\sum_{k\sim k_{1}\ge k_{2}}
			\langle2^{k}\rangle^{2}
			2^{k(\alpha-2)}2^{-k_{2}}
			\| \widetilde P_{k_{2}}\vert u\vert^{2}\|_{L^{\infty}}
			\| \widetilde P_{k_{1}}(x v)\|_{L^{2}}   \\
			&\lesssim
			\sum_{k\sim k_{1}\ge k_{2}}
			\langle2^{k}\rangle^{2}
			2^{k(\alpha-2)}2^{-k_{2}}
			\min(\langle 2^{k_{2}}\rangle^{-2}\bs^{-3}, \
			2^{3k_{2}}
			)\\
			&\qquad\qquad\qquad\qquad\qquad\qquad\times
			\min(2^{k_{1}(2-\alpha)}\bs^{(3-\alpha)\delta_0},
			\langle2^{k_{1}}\rangle^{-3}\bs^{\delta_0}
			)
			\epsilon_{1}^{3}    \\
			&\lesssim
			\epsilon_1^3\sum_{2^{k_{2}}\le\bs^{-1}}
			2^{2k_{2}}
			\sum_{k\ge k_{2}}
			\langle2^{k}\rangle^{2}
			2^{k(\alpha-2)}
			\min(2^{k(2-\alpha)}\bs^{(3-\alpha)\delta_0},
			\langle2^{k}\rangle^{-3}\bs^{\delta_0}
			) \\
			&\quad
			+\bs^{-3}\epsilon_1^3
			\sum_{2^{k_{2}}\ge\bs^{-1}}
			\langle2^{k_{2}}\rangle^{-2}
			2^{-k_{2}}
			\sum_{k\ge k_{2}}
			\langle2^{k}\rangle^{2}
			2^{k(\alpha-2)}\\
			&\qquad\qquad\qquad\qquad\qquad\times
			\min(2^{k(2-\alpha)}\bs^{(3-\alpha)\delta_0},
			\langle 2^{k}\rangle^{-3}\bs^{\delta_0}
			)  \\
			&\lesssim
			\bs^{-2+2\delta_0}\epsilon_1^3.
			\end{align*}

			({\it Case:} $k \ll k_{1}\sim k_{2}$)
			At first, we deal with the case $2^{k}\le\bs^{-2}$. By \eqref{ineq:pt}) we have
			\begin{align*}
			&\sum_{\substack{k\ll k_{1}\sim k_{2}
					\\ 2^{k}\le\bs^{-2}}}
			\langle 2^{k}\rangle^{2}
			{J_{2}}^{k,k_{1},k_{2}}(s) \\
			&\lesssim
			\sum_{\substack{k\ll k_{1}\sim k_{2}
					\\ 2^{k}\le\bs^{-2}}}
			\langle2^{k}\rangle^{2}
			2^{k_{1}(\alpha-3)}
			\| \beta_{k} \|_{L^{2}}
			\| {\widetilde P_{k_{2}}(\vert u\vert^{2})}\|_{L^{2}}
			\| {\widetilde P_{k_{1}}(xv)}\|_{L^{2}}   \\    	
			&\lesssim
			\sum_{2^{k}\le\bs^{-2}}
			2^{\frac{3k}{2}}
			\sum_{k_{1}\gg k}
			2^{k_{1}(\alpha-\frac{3}{2})}
			\langle 2^{k_{1}}\rangle^{-2}\\
			&\qquad\qquad\qquad\qquad\times
			\min(2^{k_{1}(2-\alpha)}\bs^{(3-\alpha)\delta_0},
			\langle2^{k_{1}}\rangle^{-3}\bs^{\delta_0}
			) \epsilon_1^3\\
			&\lesssim
			\bs^{-3+\delta_0}\epsilon_1^3
			\big(
			\sum_{2^{k_{1}}\le\bs^{-\delta_0}}
			2^{\frac{k_{1}}{2}}\bs^{(2-\alpha)\delta_0}
			+\sum_{2^{k_{1}}\ge\bs^{-\delta_0}}
			\langle 2^{k_{1}}\rangle^{-5}2^{k_{1}(\alpha-\frac{3}{2})}
			\big) \\
			&\lesssim
			\bs^{-2+2\delta_0}\epsilon_1^3.
			\end{align*}
			
			For the case $2^k \ge \bs^{-2}$, we use Lemma \ref{LemmaB} as follows.
			\begin{align*}
			&\sum_{\substack{k\ll k_{1}\sim k_{2}
					\\ 2^{k}\ge\bs^{-2}}}
			\langle 2^{k}\rangle^{2}
			{ J_{2}}^{k, k_1, k_2}(s) 
			\lesssim
			\sum_{\substack{k\ll k_{1}\sim k_{2}
					\\ 2^{k}\ge\bs^{-2}}}
			\langle2^{k}\rangle^{2}
			2^{k_{1}(\alpha-3)}
			\|{\widetilde P_{k_{2}}(\vert u\vert^{2})}\|_{L^{\infty}}
			\|{\widetilde P_{k_{1}}(xv)}\|_{L^{2}}   \\    	
			&\lesssim
			\sum_{2^{k_{1}}\le\bs^{-1}}
			2^{2k_{1}}\bs^{\delta_0(3-\alpha)}\epsilon_1^3
			\sum_{\bs^{-2}\le 2^{k}\le\bs^{-1}}1 \\
			&\qquad
			+\sum_{2^{k_{1}}\ge\bs^{-1}}
			\bs^{-3}
			2^{k_{1}(\alpha-3)}
			\langle 2^{k_{1}}\rangle^{-2}\epsilon_1^3\\
			&\qquad\qquad\qquad\times \min(2^{k_{1}(2-\alpha)}\bs^{(3-\alpha)\delta_0},
			\langle2^{k_{1}}\rangle^{-3}\bs^{\delta_0}
			)
			\sum_{k\le k_{1}}
			\langle 2^{k}\rangle^{2}     \\
			&\lesssim
			\bs^{-2+2\delta_0}\epsilon_1^3.
			\end{align*}
			
			({\it Case:} $k_{1} \ll k\sim k_{2}$)
			Using Lemma \ref{LemmaB} as above, we get
			\begin{align*}
			&\sum_{k_{1}\ll k\sim k_{2}}
			\langle 2^{k}\rangle^{2}
			{ J_{2}}^{k, k_1, k_2}(s) 
			\lesssim
			\sum_{k_{1}\ll k\sim k_{2}}
			\langle2^{k}\rangle^{2}
			2^{k_{2}(\alpha-3)}
			\|{\widetilde P_{k_{2}}(\vert u\vert^{2})}\|_{L^{\infty}}
			\|{\widetilde P_{k_{1}}(xv)}\|_{L^{2}}   \\
			&\lesssim \epsilon_1^3
			\sum_{2^{k_{2}}\le\bs^{-1}}
			2^{\alpha k_{2}}
			\sum_{k_{1}\le k_{2}}
			2^{k_{1}(2-\alpha)}\bs^{-\delta_0(3-\alpha)} \\
			&\qquad
			+\bs^{-3+\delta_0}\epsilon_1^3
			\sum_{2^{k_{2}}\ge\bs^{-1}}
			2^{k_{2}(\alpha-3)}
			\sum_{k_{1}\le k_{2}}
			\min(2^{k_{1}(2-\alpha)}\bs^{(2-\alpha)\delta_0},
			\langle2^{k_{1}}\rangle^{-3}
			)      \\
			&\lesssim
			\bs^{-2+2\delta_0}\epsilon_1^3.
			\end{align*}
		\end{subsubsection}

		\begin{subsubsection}{Estimate of ${\mathbf J_{3}}$}\label{j3}
			Through the dyadic decomposition in $\xi,\eta$ and $\xi-\eta$,
			it suffices to show that
			\[
			\sum_{k,k_{1},k_{2}\in\mathbb{Z}}
			\langle2^{k}\rangle^{2}
			{ J_{3}}^{k,k_{1},k_{2}}(s)
			\lesssim
			\bs^{-2+2\delta_0}\epsilon_{1}^{3},
			\]
			where
			\[
			{ J_{3}}^{k,k_{1},k_{2}}(s)
			=\bigg\|
			\int_{\mathbb{R}^{3}}
			\mathbf m_3^{k,k_{1},k_{2}}(\xi,\eta)
			\widehat{\widetilde P_{k_{2}}(\vert u\vert^{2})}(\eta)
			e^{-is|\xi-\eta|^\al}\widehat{\widetilde P_{k_{1}}(v)}(\xi-\eta)
			d\eta
			\bigg\|_{L_{\xi}^{2}}
			\]
			and
			\[
			\mathbf m_3^{k,k_{1},k_{2}}(\xi,\eta)
			= [\nabla_\xi \mathbf m(\xi,\eta)]
			\vert\eta\vert^{-2}
			\beta_{k}(\xi)
			\beta_{k_{1}}(\xi-\eta)
			\beta_{k_{2}}(\eta).
			\]
			Direct calculation gives: If $k \lesssim k_1$, then for any positive integers $\ell_1, \ell_2$
			\begin{equation}\label{ineq:pt2}
			\begin{split}
			\left\vert
			\nabla_{\xi}^{\ell_1}\nabla_{\eta}^{\ell_2}
			\mathbf m_3^{k,k_{1},k_{2}}(\xi,\eta)
			\right\vert
			&\lesssim
			\min(2^{k},2^{k_{1}})^{\alpha-2}
			\max(2^{k},2^{k_{1}})^{-1}
			2^{-k_{2}}
			2^{-\ell_1 k}
			2^{-\ell_2 k_{2}}\\
			&\qquad\times
			\widetilde \beta_{k}(\xi)\widetilde\beta_{k_{1}}(\xi-\eta)\widetilde\beta_{k_{2}}(\eta).
			\end{split}
			\end{equation}
			And also from this we have
			\begin{align}\begin{aligned}{\label{lemmaC2}}
			&\left\|
			\iint_{\mathbb{R}^{3}\times\mathbb{R}^{3}}
			e^{ix\cdot\xi}e^{iy\cdot\eta}
			\mathbf m_{3}^{k,k_{1},k_{2}}(\xi,\eta)
			d\xi d\eta
			\right\|_{L_{x,y}^{1}}\\
			&\qquad\qquad\qquad\lesssim
			\min (2^{k},2^{k_{1}})^{\alpha-2}
			\max (2^{k},2^{k_{1}})^{-1}
			2^{-k_{2}}.
			\end{aligned}\end{align}
			
			Now we divide the sum into three parts:
			\[
			\sum_{k,k_{1},k_{2}\in\mathbb{Z}}
			\langle2^{k}\rangle^{2}
			{ J_{3}}^{k,k_{1},k_{2}}(s)
			\lesssim
			\left(
			\sum_{k\sim k_{1}\ge k_{2}}
			+\sum_{k\ll k_{1}\sim k_{2}}
			+\sum_{k_{1}\ll k\sim k_{2}}
			\right)
			\langle2^{k}\rangle^{3}
			{J_{3}}^{k,k_{1},k_{2}}(s).
			\]

			({\it Case:} $k\sim k_{1}\ge k_{2}$) If $\alpha\ge\frac{3}{2},$
			we can obtain the desired bound by applying Lemma \ref{LemmaB} with $A(\mathbf m_3^{k, k_1, k_2})$ of \eqref{lemmaC2}.
			On the other hand, if $\alpha<\frac{3}{2},$ the sum over $2^{k}\le\bs^{-2}$ can be estimated by using just the pointwise bound \eqref{ineq:pt2} as before, and the remaining case can be treated by applying the Lemma \ref{LemmaB}.
			
			({\it Case:} $k\ll k_{1}\sim k_{2}$) In this case the condition $\delta_0 \ge \frac{17}{12} - \frac{5\al}{6}$ is necessary.
			Taking the sum over $2^{k}\le\bs^{-\frac{5}{3}}$, we estimate
			\begin{align*}
			\sum_{\substack{k\ll k_{1}\sim k_{2}\\
					2^{k}\le\bs^{-\frac{5}{3}}}}
			\langle2^{k}\rangle^{2}&
			{ J_{3}}^{k,k_{1},k_{2}}(s)\\
			&\lesssim
			\sum_{\substack{k\ll k_{1}\sim k_{2}\\
					2^{k}\le\bs^{-\frac{5}{3}}}}	
			\|\mathbf m_3^{k,k_{1},k_{2}} \|_{L^{\infty}}
			\| \beta_{k} \|_{L^{2}}
			\| {\widetilde P_{k_{2}}(\vert u\vert^{2})}\|_{L^{2}}
			\| {\widetilde P_{k_{1}}(v)}\|_{L^{2}}   \\    	
			&\lesssim
			\sum_{2^{k}\le\bs^{-\frac{5}{3}}}
			2^{k(\alpha-\frac{1}{2})}
			\sum_{k_{1}\gg k}
			2^{k_{1}}
			\langle 2^{k_{1}}\rangle^{-10}
			\epsilon_{1}^{3}. \\
			&\lesssim
			\bs^{-\frac{5}{3}(\alpha-\frac{1}{2})}
			\epsilon_{1}^{3}.
			\end{align*}
			And the sum over $2^{k}\ge\bs^{-\frac{5}{3}}$ follows from Lemma \ref{LemmaB} with $A(\mathbf m_3^{k, k_1, l_2} )$ of \eqref{lemmaC2}:
			\begin{align*}
			\sum_{\substack{k\ll k_{1}\sim k_{2}\\
					2^{k}\ge\bs^{-\frac{5}{3}}}}
			\langle2^{k}\rangle^{2} &
			{ J_{3}}^{k,k_{1},k_{2}}(s)\\
			&\lesssim
			\sum_{\substack{k\ll k_{1}\sim k_{2}\\
					2^{k}\ge\bs^{-\frac{5}{3}} }}
			\langle 2^{k}\rangle^2
			2^{k(\alpha-2)}
			2^{-2k_{1}}
			\| {\widetilde P_{k_{2}}(\vert u\vert^{2})}\|_{L^{\infty}}
			\| {\widetilde P_{k_{1}}(v)}\|_{L^{2}}   \\
			&\lesssim
			\sum_{2^{k_{1}}\le\bs^{-1}}
			2^{\frac{5k_{1}}{2}}
			\sum_{\bs^{-\frac{5}{3}}\le 2^{k}\le\bs^{-1}}
			2^{k(\alpha-2)}
			\epsilon_{1}^{3} \\
			&\qquad+
			\bs^{-3}
			\sum_{2^{k_{1}}\ge\bs^{-1}}
			2^{-\frac{k_{1}}{2}}
			\langle 2^{k_{1}}\rangle^{-5}
			\sum_{2^{k}\ge\bs^{-\frac{5}{3}}}
			2^{k(\alpha-2)}
			\epsilon_{1}^{3} \\
			&\lesssim
			\bs^{-\frac{5}{2}+\frac{5}{3}(2-\alpha)}
			\epsilon_{1}^{3}.       	
			\end{align*}
			If $\delta_0 \ge \frac{17}{12} - \frac{5\al}{6}$, then the last two terms of the above estimates have the desired bound.

			({\it Case:} $k_{1}\ll k\sim k_{2}$)
			By using Lemma \ref{LemmaB} again, we have
			\begin{align*}
			\sum_{k_{1}\ll k\sim k_{2}}
			\langle 2^{k}\rangle^{2}
			{J_{3}}^{k,k_{1},k_{2}}(s)
			\lesssim
			\bs^{-\frac{1}{2}-\alpha+\delta_0}
			\epsilon_{1}^{3}.
			\end{align*}
			Therefore, if $\al \ge \frac32$, then we get the desired bound.
			
		\end{subsubsection}

		\begin{subsubsection}{Estimate of ${\mathbf J_{4}}$} By dyadic decomposition $\langle\xi\rangle^{2}
			{\mathbf J_{4}}(s,\xi)$ can be written as
			\begin{align*}
			\langle\xi\rangle^{2}
			{\mathbf J_{4}}(s,\xi)
			= -i c_0\sum_{k,k_{1},k_{2}\in\mathbb{Z}}
			\langle\xi\rangle^{3}
			&\int_{0}^{t}
			s^{2}
			\int_{\mathbb{R}^{3}}
			\mathbf m_4^{k,k_{1},k_{2}}(\xi,\eta)\\
			&\times
			e^{is\phi_\al(\xi, \eta)} \widehat{\widetilde P_{k_{2}}(\vert u\vert^{2})}(\eta)
			\widehat{\widetilde P_{k_{1}}(v)}(\xi-\eta)
			d\eta
			ds,
			\end{align*}
			where $$\mathbf m_4^{k, k_1, k_2} = \mathbf m(\xi,\eta) \otimes  \mathbf m^{k,k_{1},k_{2}}(\xi,\eta).$$
			It is also enough to show the following estimate
			\[
			\sum_{k,k_{1},k_{2}\in\mathbb{Z}}
			\langle2^{k}\rangle^{3}
			{ J_{4}}^{k,k_{1},k_{2}}(s)
			\lesssim
			\bs^{-3+\delta_0}\epsilon_{1}^{3},
			\]
			where
			\[
			{ J_{4}}^{k,k_{1},k_{2}}(s)
			=\bigg\|
			\int_{\mathbb{R}^{3}}
			\mathbf m_4^{k, k_1, k_2}
			\widehat{\widetilde P_{k_{2}}(\vert u\vert^{2})}(\eta)
			e^{-is|\xi-\eta|^\al}\widehat{\widetilde P_{k_{1}}(v)}(\xi-\eta)
			d\eta
			\bigg\|_{L_{\xi}^{2}}.
			\]
			From $(\ref{ineq:pt})$ and Leibniz rule, it follows that if $k\lesssim k_{1}$, then for any positive integers $\ell_1, \ell_2$
			\begin{align}\begin{aligned}\label{ineq:pt4}
			\left\vert
			\nabla_{\xi}^{\ell_1}\nabla_{\eta}^{\ell_2}
			{\mathbf m}_4^{k,k_{1},k_{2}}(\xi,\eta)
			\right\vert
			&\lesssim
			\max(2^{k},2^{k_{1}},2^{k_{2}})^{2\alpha-4}
			2^{-\ell_1 k}2^{-\ell_2 k_{2}} \\
			&\qquad\qquad \times
			\widetilde\beta_{k}(\xi)\widetilde\beta_{k_{1}}(\xi-\eta)\widetilde\beta_{k_{2}}(\eta).
			\end{aligned}\end{align}
			Then one can easily observe from this that
			\begin{equation}{\label{ineq:lemmac3}}
			\left\|
			\iint_{\mathbb{R}^{3}\times\mathbb{R}^{3}}
			e^{ix\cdot\xi}e^{iy\cdot\eta}
			{\mathbf m}_4^{k,k_{1},k_{2}}(\xi,\eta)d\xi d\eta
			\right\|_{L_{x,y}^{1}}
			\lesssim
			\max(2^{k},2^{k_{1}},2^{k_{2}})^{2\alpha-4}.
			\end{equation}
			
			Now dividing the sum into three parts, we have
			\[
			\sum_{k,k_{1},k_{2}\in\mathbb{Z}}
			\langle2^{k}\rangle^{2}
			{J_4}^{k,k_{1},k_{2}}(s)
			\lesssim
			\big(
			\sum_{k\sim k_{1}\le k_{2}}
			+\sum_{k\ll k_{1}\sim k_{2}}
			+\sum_{k_{1}\ll k\sim k_{2}}
			\big)
			\langle2^{k}\rangle^{2}
			{J_4}^{k,k_{1},k_{2}}(s).
			\]
			
			({\it Case:} $k\sim k_{1}\ge k_{2}$) We can easily obtain the bound applying the Lemma \ref{LemmaB}
			with $A({\mathbf m}_4^{k, k_1, k_2})$ of \eqref{ineq:lemmac3}.
			\begin{align*}
			\sum_{k\sim k_{1}\le k_{2}}
			\langle2^{k}\rangle^{2}
			{J_4}^{k,k_{1},k_{2}}(s)
			&\lesssim
			\sum_{2^{k_{2}}\le\bs^{-1}}
			2^{3k_2}
			\sum_{k\ge k_{2}}
			\langle 2^{k}\rangle^{-2}2^{k(2\alpha-\frac{5}{2})}\epsilon_1^3\\
			&\quad\quad
			+\bs^{-3}
			\sum_{2^{k_2}\ge\bs^{-1}}
			\langle 2^{k_{2}}\rangle^{-2}
			\sum_{k\ge k_{2}}
			\langle 2^{k}\rangle^{-2}
			2^{k(2\alpha-\frac{5}{2})}\epsilon_1^3\\
			&\lesssim
			\bs^{-3+\delta_0}\epsilon_1^3.
			\end{align*}
			
			({\it Case:} $k\ll k_{1}\sim k_{2}$) Taking the sum over $ 2^{k}\le\bs^{-2} $, by the pointwise estimate we have
			\begin{align*}
			\sum_{\substack{k\ll k_{1}\sim k_{2}\\
					2^{k}\le\bs^{-1}}}
			\langle 2^{k}\rangle^{2}
			{J_4}^{k,k_{1},k_{2}}(s)
			&\lesssim
			\sum_{2^{k}\le\bs^{-2}}
			2^{\frac{3k}{2}}
			\sum_{k_{1}\ge k}
			2^{k_{1}(2\alpha-1)}
			\langle 2^{k_{1}}\rangle^{-5}\epsilon_1^3\\
			&\lesssim
			\bs^{-3}\epsilon_1^3.
			\end{align*}
			The remaining case can also be estimated by Lemma \ref{LemmaB} as
			\begin{align*}
			&    \sum_{\substack{k\ll k_{1}\sim k_{2}\\
					2^{k}\ge\bs^{-2}}}
			\langle2^{k}\rangle^{2}
			{J_4}^{k,k_{1},k_{2}}(s)\\
			&\lesssim
			\epsilon_1^3\sum_{2^{k_{1}}\le\bs^{-1}}
			2^{(2\alpha-1)k_{1}}
			\sum_{\substack{k\le k_{1}\\
					2^{k}\ge\bs^{-2} }}
			2^{\frac{3k}2} \\
			&\qquad +\bs^{-3}\epsilon_1^3
			\sum_{2^{k_{1}}\ge\bs^{-1}}
			\langle 2^{k_{1}}\rangle^{-4}
			2^{k_{1}(2\alpha-\frac{5}{2})}
			\sum_{\substack{k\le k_{1} \\
					2^{k}\ge\bs^{-2}}}
			\langle 2^{k_{1}}\rangle^{-1} \\
			&\lesssim
			\bs^{-3+\delta_0}\epsilon_1^3.
			\end{align*}
			
			({\it Case:} $k_{1}\ll k\sim k_{2}$) By Lemma \ref{LemmaB} with \eqref{ineq:lemmac3} we finally have
			\begin{align*}
			\sum_{k_{1}\ll k\sim k_{2}}
			\langle2^{k}\rangle^{2}
			{J_4}^{k,k_{1},k_{2}}(s)
			&\lesssim
			\epsilon_1^3\sum_{2^{k}\le\bs^{-1}}
			2^{(2\alpha-1)k}
			\sum_{k_{1}\le k}
			2^{\frac{3k_{1}}{2}} \\
			&\qquad +\bs^{-3}\epsilon_1^3
			\sum_{2^{k}\ge\bs^{-1}}
			2^{k(2\alpha-4)}
			\sum_{k_{1}\le k}
			2^{\frac{3k_{1}}{2}}
			\langle 2^{k_{1}}\rangle^{-5}   \\
			&\lesssim
			\bs^{-2\alpha-\frac{1}{2}+\delta_0}\epsilon_1^3\\
			&\lesssim \bs^{-3 + \delta_0}\epsilon_1^3, \;\;\mbox{if}\;\; \al \ge \frac54.
			\end{align*}

		\end{subsubsection}
		
	\end{subsection}

	$\newline$
	
	\section{Modified scattering}
	
	We consider $L^\infty$-control for $\bxi^5 \widehat{v}$ and show the modified scattering Theorem \ref{main thm}.
	At first we have the following.
	\begin{prop}\label{Linfty}
		Let $\frac53 < \al < 2$, $N > \frac{75\al+35}{3\al-5}$ and
		$$
		0 < \delta_0 < \min\left(\frac{3\al-5}{100(\al+1)}(N-5) - \frac35, \frac{2-\al}{3\al}, \frac1{100}\right).
		$$
		Suppose that $u$ is the global solution to \eqref{equation}  such that $\| u \|_{\Sigma_{T}}\le \epsilon_{1}$  for some $\epsilon_1, T > 0$ and $u_0$ satisfies \eqref{initialvalue} with $\epsilon_0 \le \epsilon_1$. Then we have
		\begin{align}\label{Linfty-bound}
		\sup_{t \in [0, T]}\|\bxi^5 \widehat{v}\|_{L^\infty} \le \epsilon_0 + {C\epsilon_1^{3}}.
		\end{align}
	\end{prop}
	\begin{proof}
		If $T\le2$, \eqref{Linfty-bound} follows from \eqref{gronwal}.
		The case where $T\ge2$ follows straightforwardly from Proposition \ref{diff-prop} below by considering $t_1 = 1$ and $t_2 = T$.	
	\end{proof}

	Given a solution $u$ of $(\ref{equation})$ satisfying the priori bounds $\| u\|_{\Sigma_{T}}\le \epsilon_{1},$
	we define for any $t\in[0,T]$ and $\xi\in\mathbb{R}^{3}$
	$$
	B_{\alpha}(t,\xi)
	= \frac{c_0}{4\alpha\pi}\int_{0}^{t}\int_{\mathbb{R}^{3}}
	\left\vert
	\frac{\xi}{\vert\xi\vert^{2-\alpha}}
	-\frac{\sigma}{\vert\sigma\vert^{2-\alpha}}
	\right\vert^{-1}
	\left\vert
	\widehat{v}(s,\sigma)
	\right\vert^{2}
	d\sigma
	\varphi_{s}(\xi)
	\frac{1}{\bs}
	ds,
	$$
	where
	$$
	\varphi_{s}(\xi)
	=\varphi\left(s^{-\theta}\xi\right),\;\; \theta = \frac{(3\al-5)}{40(\alpha+1)}.
	$$
	for a smooth compactly supported function $\varphi$. 
	We also define the modified profile
	$$
	w(t, \xi):=e^{-iB_\al(t,\xi)}\widehat{v}(t,\xi).
	$$
	From Lemma \ref{lowbound} we have the following a priori estimate.
	\begin{align}\label{ba}
	|B_\al| \lesssim \int_0^t\int (|\sigma|^{-(\al-1)} + |\xi-\sigma|^{-1}|\sigma|^{2-\al})|\widehat u(s, \sigma)|^2\,d\sigma \varphi_s(\xi) \frac{ds}{1+s}.
	\end{align}
	For the inner integral we have
	\begin{align}\begin{aligned}\label{ba2}
	\int (|\sigma|^{-(\al-1)} &+ |\sigma|^{2-\al}\,|\xi-\sigma|^{-1})|\widehat u(s, \sigma)|^2\,d\sigma\\
	&\lesssim \||x|^\frac{\al-1}2 v\|_{L^2}^2 +
	\|u\|_{\dot{H}^{\frac{2-\alpha}{2}}}^{2}
	+\||\cdot|^{\frac{2-\alpha}{2}}\widehat{u}\|
	_{L^{\infty}}^{2} \\
	&\lesssim \|xv\|_{H^{3}}^2+\|u\|_{H^{N}}^2
	+\|\langle\cdot\rangle^{5}\widehat{u}\|_{L^{\infty}}^{2}<\infty
	\end{aligned}\end{align}
	Therefore $B_\al$ is well-defined and real-valued.

	The main goal of this section is to show the following.
	
	\begin{prop}\label{diff-prop}
		Let $\frac53 < \al < 2$. Assume that $u\in C([0,T]; H^{N})$ satisfies a priori bound $ \| u \|_{\Sigma_{T}}\le\epsilon_{1} $ for $N, \delta_0 > 0$ as in Proposition \ref{Linfty}.
		Then there holds
		\begin{equation}\label{difference}
		\sup_{t_{1}\le t_{2}\in[1,T]}
		\langle t_{1}\rangle^{\delta_0}
		\|\bxi^{5}
		\big(w(t_{2},\xi) - w(t_{1},\xi)\big)
		\|_{L_{\xi}^{\infty}}
		\lesssim \epsilon_{1}^{3}.
		\end{equation}
	\end{prop}
	
	\begin{rk}\label{global bound}
		The conditions of Theorem \ref{main thm} for $\al, N, \delta_0$ satisfy simultaneously those in Proposition \ref{scattering} and Proposition \ref{Linfty}. If we choose $\epsilon_0, \epsilon_1$ such that $\epsilon_1 = 3\epsilon_0$ and $C\epsilon_1^3 \le \epsilon_0$, then by combining Proposition \ref{Linfty} with Proposition \ref{scattering} and by continuity argument we deduce that $\|u\|_{\Sigma_T} \le 3\epsilon_0$ for all $T > 0$. In particular, we can choose $\bar{\epsilon}_0$ as in Theorem \ref{main thm} as $$\bar{\epsilon}_0 = \frac1{3\sqrt{3C}}.$$
	\end{rk}

	Let us set $\|u\|_{\Sigma_\infty} := \sup_{T> 0}\|u\|_{\Sigma_T}$. If $\epsilon_0 < \bar{\epsilon_0}$ as mentioned in Remark \ref{global bound}, then clearly
	$$
	\|u\|_{\Sigma_\infty} \le \epsilon_1.
	$$
	From this we deduce the following.
	\begin{cor}\label{asympt-prop}
		Let $\frac53 < \al < 2$. Assume that $u\in C([0, \infty); H^{N})$ satisfies a priori bound $ \| u \|_{\Sigma_{\infty}} \le \epsilon_{1}$ for $N, \delta_0 > 0$ as in Proposition \ref{Linfty}.
		Then there exist asymptotic state $v_{+}$,
		such that for all $t>0$
		\begin{align}\label{mscattering1}
		\lVert \bxi^{5}
		[
		w(t,\xi) - v_{+}(\xi)
		]
		\|_{L_{\xi}^{\infty}}
		\lesssim
		\ep_1^3
		\langle t \rangle^{-\delta_0}
		\end{align}
	\end{cor}
	\begin{proof}
		Letting $t_2\rightarrow\infty$ in \eqref{difference} implies the modified scattering  \eqref{mscattering1} once we define
		$$
		v_{+} := \lim_{t_2\rightarrow \infty}w(t_2,\xi).
		$$
		Here the limit has  been taken in $\bxi^{-5}L_\xi^{\infty}$.	
	\end{proof}

	\begin{proof}[Proof of Proposition \ref{diff-prop}]
		For $(\ref{difference})$, it suffices to show that
		if $t_{1}\le t_{2}\in[2^{m}-1,2^{m+1}]\cap [0,T]$
		for some positive integer $m$, then
		\begin{equation}\label{goal5.2}
		\|\bxi^{5}
		\big(w(t_{2},\xi) - w(t_{1},\xi)\big)
		\|_{L_{\xi}^{\infty}}
		\lesssim 2^{-\delta_0 m}\epsilon_{1}^{3}.
		\end{equation}	
		Indeed, Applying \eqref{goal5.2} we get for any $1 \le t_1 \le t_2 \in [0,T]$, \begin{align}
		& \|\bxi^{5}
		\big(w(t_2,\xi) - w(t,\xi)\big)
		\|_{L_{\xi}^{\infty}}\nonumber \le\\
		&\|\bxi^{5}
		\big(w(t,\xi) - w(2^{m},\xi)\big)
		\|_{L_{\xi}^{\infty}}\nonumber 
		+\sum_{j=0}^{m'-1}
		\|\bxi^{5}
		\big(w(2^{m+j},\xi) - w(2^{m+j+1},\xi)\big)
		\|_{L_{\xi}^{\infty}} \\
		&\qquad + \|\bxi^{5}
		\big(w(2^{m+m'},\xi) - w(t_2,\xi)\big)
		\|_{L_{\xi}^{\infty}} \nonumber\\
		&\lesssim
		\sum_{j=0}^{m'-1}
		2^{-\delta_0 (m+j)}\epsilon_{1}^{3}
		\lesssim
		t^{-\delta_0}\epsilon_{1}^{3}, \label{difference2}
		\end{align}
		where $t\in [2^{m}-1,2^{m+1}]$ and $t_2\in [2^{m+m'}-1,2^{m+m'+1}]$.
		\color{black} Since $s \ge 1$, we assume that $s \sim 2^m$.
		Making change of variables, the nonlinear term $I$ as in $(\ref{Duhamel})$ can be written as
		$$
		I(s,\xi) =
		i c_1 \iint_{\mathbb{R}^{3}\times\mathbb{R}^{3}}
		e^{is\phi(\xi,\eta,\sigma)}
		\vert\eta\vert^{-2}
		\widehat{v}(s,\xi+\sigma)
		\overline{\widehat{v}(}s,\xi+\eta+\sigma)
		\widehat{v}(s,\xi+\eta)
		d\eta d\sigma,
		$$
		where
		\begin{align}\label{phase}
		\phi(\xi,\eta,\sigma) =
		\vert\xi\vert^{\alpha}
		-\vert\xi+\eta\vert^{\alpha}
		-\vert\xi+\sigma\vert^{\alpha}
		+\vert\xi+\eta+\sigma\vert^{\alpha}.
		\end{align}
		Let $n_0 = n_{0}(\alpha,m)\in\mathbb{Z}$ be the largest integer satisfying
		\begin{equation}\label{cond:j0}
		n_0 < \left(\frac{-2-2\delta_0}{\al+1} - \frac{10\theta}{3}\right)m.
		\end{equation}
		Then, it is clear that $2^{-n_{0}} \sim 2^{(\frac{2+2\delta_0}{\al+1}+\frac{10\theta}{3} )m}$.
		
		Let us now invoke the cut-off function $\beta_{n_1}^{(n_0)}$ in as \eqref{beta} replaced with $n = n_1$.
		Then the time derivative of $\widehat{v}$ can be decomposed as
		\begin{equation}\label{l1decompose}
		\partial_{s}\widehat{v}(s,\xi)
		= \partial_{s}\widehat{v}_{n_0}(s,\xi)
		+\sum_{n_1> n_0, n_1 \in \mathbb{Z}}\partial_{s}\widehat{v}_{n_1}(s,\xi),
		\end{equation}
		where
		\begin{align*}
		\partial_{s}\widehat{v}_{n_0}(s,\xi)
		&=ic_1\iint_{\mathbb{R}^{3}\times \mathbb{R}^{3}}
		e^{is\phi(\xi,\eta\,\sigma)}
		\vert\eta\vert^{-2}\beta_{n_0}^{(n_0)}(\eta)\\
		&\qquad\qquad\qquad\times\widehat{v}(s,\xi+\eta)\overline{\widehat{v}(}s, \xi+\eta+\sigma)\widehat{v}(s,\xi+\sigma)d\eta d\sigma, \\
		\partial_{s}\widehat{v}_{n_1}(s,\xi)  &=ic_1\iint_{\mathbb{R}^{3}\times \mathbb{R}^{3}}
		e^{is\phi(\xi,\eta\,\sigma)}
		\vert\eta\vert^{-2}
		\beta_{n_1}^{(n_0)}(\eta)\\
		&\qquad\qquad\qquad\times\widehat{v}(s,\xi+\eta)\overline{\widehat{v}(}s, \xi+\eta+\sigma)
		\widehat{v}(s,\xi+\sigma)
		d\eta d\sigma.
		\end{align*}
		By \eqref{l1decompose} we have
		\begin{align*}
		&\|\bxi^{5}(w(t_{1},\xi) - w(t_{2},\xi))
		\|_{L_{\xi}^{\infty}} = \big\|\bxi^{5}
		\int_{t_{1}}^{t_{2}}\partial_{s}w(s,\xi)ds
		\big\|_{L_{\xi}^{\infty}} \\
		&=
		\big\|\bxi^{5}
		\int_{t_{1}}^{t_{2}}
		e^{-iB_\al(s,\xi)}
		( \partial_{s}\widehat{v}(s,\xi)
		-i\partial_{s}B_\al(s,\xi)
		\widehat{v}(s,\xi)
		)
		ds
		\big\|_{L_{\xi}^{\infty}} \\
		&\lesssim
		\big\|\bxi^{5}
		\int_{t_{1}}^{t_{2}}
		e^{-iB_\al(s,\xi)}
		\Big(\partial_{s}\widehat{v}_{n_0}(s,\xi)-i\partial_{s}[B_\al(s,\xi)]\widehat{v}(s,\xi)\\
		&\qquad\qquad\qquad\qquad\qquad\qquad\qquad\qquad\quad\;\;    +\sum_{n_1 > n_0, n_1\in\mathbb{Z}}\partial_{s}\widehat{v}_{n_1}(s,\xi)
		\Big) ds
		\big\|_{L_{\xi}^{\infty}} \\
		&\lesssim
		\|\bxi^{5}
		\int_{t_{1}}^{t_{2}}
		e^{-iB_\al(s,\xi)}
		\Big(\partial_{s}\widehat{v}_{n_0}(s,\xi) - i\partial_{s}[B_\al(s,\xi)]\widehat{v}(s,\xi)\Big)\,ds
		\|_{L_{\xi}^{\infty}}\\
		&\qquad\qquad\qquad\qquad\quad\,    +\|\bxi^{5}
		\int_{t_{1}}^{t_{2}}
		\sum_{n_1>n_0}e^{-iB_\al(s,\xi)}
		\partial_{s}\widehat{v}_{n_1}(s,\xi)ds
		\|_{L_{\xi}^{\infty}}\\
		&\lesssim
		\big\|\bxi^{5}
		\int_{t_{1}}^{t_{2}}e^{-iB_\al(s,\xi)}
		\big(\partial_{s}\widehat{v}_{n_0}(s,\xi)\big(1-\varphi_{s}(\xi)\big)\big)
		ds
		\big\|_{L_{\xi}^{\infty}} \\
		&\quad
		+\big\|\bxi^{5}
		\int_{t_{1}}^{t_{2}}e^{-iB_\al(s,\xi)}
		\big(\partial_{s}\widehat{v}_{n_0}(s,\xi)\varphi_{s}(\xi) - i\partial_{s}[B_\al(s,\xi)]\widehat{v}(s,\xi)\big)ds
		\big\|_{L_{\xi}^{\infty}} \\
		&\qquad
		+\big\|\bxi^{5}
		\int_{t_{1}}^{t_{2}}\sum_{n_1>n_0}
		e^{-iB_\al(s,\xi)}
		\partial_{s}\widehat{v}_{n_1}(s,\xi)ds
		\big\|_{L_{\xi}^{\infty}}.
		\end{align*}
		So, it suffices to show that the following three estimates hold.
		\begin{equation}\label{ineq:1}
		\vert \partial_{s}\widehat{v}_{n_0}(s,\xi)(1-\varphi_{s}(\xi))\vert \lesssim
		2^{(-1-\delta_0)m}
		\langle\xi\rangle^{-5}\epsilon_{1}^{3},
		\end{equation}
		\begin{equation}\label{ineq:2}
		\vert
		\partial_{s}\widehat{v}_{n_0}(s,\xi)\varphi_{s}(\xi)
		- i\partial_{s}B_\alpha(s,\xi)\widehat{v}(s,\xi) \vert
		\lesssim
		2^{(-1-\delta_0)m}
		\langle\xi\rangle^{-5}\epsilon_{1}^{3},
		\end{equation}
		\begin{equation}\label{ineq:3}
		\sum_{n_1>n_0}
		\vert \partial_{s}\widehat{v}_{n_1}(s,\xi) \vert
		\lesssim
		2^{(-1-\delta_0)m}\langle\xi\rangle^{-5}\epsilon_{1}^{3}.
		\end{equation}
		
	\end{proof}

	\subsection{Proof of \eqref{ineq:1}(High Frequency part) }
	\color{black}
	Decomposing of \eqref{ineq:1} as in $\eqref{l1decompose}$, we get
	$$
	\vert \partial_{s}\widehat{v}_{n_0}(s,\xi)(1-\varphi_{s}(\xi))\vert \lesssim
	\sum_{n_1<n_0+10}
	\vert \partial_{s}\widehat{v}_{n_1}(s,\xi) \vert.
	$$
	In the case of high frequency $\vert\xi\vert\gg\vert\eta\vert$,
	i.e. $\vert\xi\vert\sim\vert\xi+\eta\vert$, we estimate
	\begin{equation}\label{ineq:high}
	\begin{split}
	\vert \partial_{s}\widehat{v}_{n_1}(s,\xi) \vert
	&\lesssim
	2^{-2n_1}\langle\xi\rangle^{-N}
	\| \beta_{n_1}\|_{L^{2}}
	\| (1+\vert\xi+\cdot\vert)^{N}\widehat{v}(\xi+\cdot)
	\widehat{v}*\widehat{v}
	\|_{L^{2}} \\
	&\lesssim
	2^{-\frac{n_1}{2}}\langle\xi\rangle^{-N}
	\| u \|_{H^{N}}
	\| \widehat{v} \|_{L^{2}}^{2}\\
	&\lesssim
	2^{-\frac{n_1}{2}}\langle\xi\rangle^{-N}
	\bs^{\delta_0}\epsilon_{1}^{3}.
	\end{split}
	\end{equation}
	And we also have
	\begin{equation}\label{ineq:high2}
	\begin{split}
	\vert\partial_{s}\widehat{v}_{n_1}(s,\xi) \vert
	&\lesssim
	2^{-2n_1}\langle\xi\rangle^{-5}
	\|\beta\|_{L^{1}}
	\| (1+\vert\xi+\cdot\vert)^{5}\widehat{v}(\xi+\cdot)
	\widehat{v}*\widehat{v}
	\|_{L^{\infty}} \\
	&\lesssim
	2^{n_1}\langle\xi\rangle^{-5}
	\|(1+\vert\cdot\vert)^{5}
	\widehat{v}\|_{L^{\infty}}
	\|\widehat{v}\|_{L^{2}}^{2}\\
	&\lesssim
	2^{n_1}\langle\xi\rangle^{-5}
	\epsilon_{1}^{3}.
	\end{split}
	\end{equation}
	Then by dividing the sum w.r.t. $n_1$ into two parts and using
	$\langle\xi\rangle \gtrsim \bs^\theta \sim 2^{\theta m}$ we get
	\begin{align*}
	&\sum_{n_1<n_0+10}
	\vert \partial_{s}\widehat{v}_{n_1}(s,\xi) \vert\\
	&\lesssim
	\sum_{n_1\le(-1-\delta_0)m}
	2^{n_1}\langle\xi\rangle^{-5}
	\epsilon_{1}^{3}
	+\sum_{(-1-\delta_0)m\le n_1\le n_0+10}
	2^{-\frac{n_1}{2}}\langle\xi\rangle^{-N}
	\bs^{\delta_0}\epsilon_{1}^{3} \\
	&\lesssim
	\langle\xi\rangle^{-5}
	\epsilon_{1}^{3}
	\Big(
	2^{(-1-\delta_0)m}
	+ 2^{-\theta(N-5)m + \delta_0 m}
	\sum_{(-1-\delta_0)m\le n_1\le n_0+10}
	2^{-\frac{n_1}{2}}
	\Big)\\
	&\lesssim 2^{(-1-\delta_0)m}
	\langle\xi\rangle^{-5}
	\epsilon_{1}^{3},
	\end{align*}
	since $\delta_0 \le \frac{3\al-5}{100(\al+1)}(N-5) - \frac35$.  	
	This proves \eqref{ineq:1}.

	\subsection{Proof of $(\ref{ineq:2})$}
	
	\subsubsection{Phase approximation}
	
	$\newline$
	Define
	\begin{align*}
	\widehat{v}_{{0,1}}(s,\xi)
	= ic_1\iint_{\mathbb{R}^{3}\times\mathbb{R}^{3}}
	&e^{is\widetilde{\phi}(\xi,\eta,\sigma)}
	\vert\eta\vert^{-2}
	\beta_{n_0}^{(n_0)}(\eta)\\
	&\qquad \times \widehat{v}(s,\xi+\eta)\overline{\widehat{v}(}s, \xi+\eta+\sigma)\widehat{v}(s,\xi+\sigma)
	d\eta d\sigma,
	\end{align*}
	where
	\begin{align}\label{tildephase}
	\widetilde{\phi}(\xi,\eta,\sigma)
	=\alpha\left(
	\frac{\xi\cdot\eta}{\vert\xi\vert^{2-\alpha}}
	-\frac{(\xi+\sigma)\cdot\eta}
	{\vert\xi+\sigma\vert^{2-\alpha}}
	\right).
	\end{align}
	By Lemma \ref{diff-phase} we have
	\begin{align*}
	\vert \phi(\xi,\eta,\sigma)-\tilde\phi(\xi,\eta,\sigma)\vert
	\lesssim \vert\eta\vert^{\alpha}.
	\end{align*}
	Now, we have
	\begin{align*}
	&\vert \partial_{s}\widehat{v}_{n_0}(s,\xi)-\widehat{v}_{0,1}(s,\xi) \vert\\
	&\lesssim
	s \iint_{\mathbb{R}^{3}\times\mathbb{R}^{3}}
	\vert\eta\vert^{\alpha}
	\vert\eta\vert^{-2}
	\beta_{n_0}^{(n_0)}(\eta)
	\left|\widehat{v}(s,\xi+\eta)\overline{\widehat{v}(}s, \xi+\eta+\sigma)\widehat{v}(s,\xi+\sigma)
	\right|
	d\eta d\sigma\\
	&\lesssim
	s\langle\xi\rangle^{-5} \iint_{\mathbb{R}^{3}\times\mathbb{R}^{3}}
	\vert\eta\vert^{\alpha-2}
	\beta_{n_0}^{(n_0)}(\eta)
	\langle
	\max(\vert\xi+\eta+\sigma\vert,
	\vert\xi+\eta\vert,
	\vert\xi+\sigma\vert)
	\rangle^{5} \\
	&\qquad\quad\times
	\vert\widehat{v}(s,\xi+\eta)\overline{\widehat{v}(}s, \xi+\eta+\sigma)\widehat{v}(s,\xi+\sigma)
	\vert
	d\eta d\sigma \\
	&\lesssim
	s\langle\xi\rangle^{-5}
	\| \bxi^{5}\widehat{v}
	\|_{L^{\infty}}
	\|\vert\cdot\vert^{\alpha-2}
	\beta_{n_0}^{(n_0)}\|_{L^{1}}
	\|\widehat{v}\|_{L^{2}}
	\|\widehat{v}\|_{L^{2}} \\
	&\lesssim
	s \langle \xi \rangle^{-5}
	2^{n_0(\alpha+1)}
	\epsilon_{1}^{3} \\
	&\lesssim 2^m2^{-(2+2\delta_0 + \frac1{12}(3\al-5))m}\langle \xi \rangle^{-5} \epsilon_{1}^{3} \\
	&\lesssim
	2^{(-1-\delta_0)m} \langle \xi \rangle^{-5}
	\epsilon_{1}^{3} \;\;\mbox{for any}\;\;\delta_0 > 0.
	\end{align*}
	
	\subsubsection{Profile approximation}
	We further approximate
	$\widehat{v}_{0,1}(s,\xi)$ by
	$$
	\widehat{v}_{0,2}(s,\xi)
	= ic_1
	\iint_{ \mathbb{R}^{3}\times\mathbb{R}^{3}}
	e^{is\widetilde{\phi}(\xi,\eta,\sigma)}
	\vert\eta\vert^{-2}
	\beta_{n_0}^{(n_0)}(\eta)
	\widehat{v}(s,\xi)\overline{\widehat{v}(}s, \xi+\sigma)\widehat{v}(s,\xi+\sigma)
	d\eta d\sigma.
	$$
	In order to do this, we define with the notation in \eqref{beta}
	$$
	v_{\le J}(x):= \beta_{J}^{(J)}(x)v(x)
	\ \textrm{and} \
	v_{> J}(x) := v(x)-v_{\le J}(x)\;\;\mbox{for}\;\;J \ge 0.
	$$
	
	We see that for $ \vert\eta\vert \lesssim 2^{n_0},$
	\begin{align*}
	\vert \widehat{v}(\rho+\eta)-\widehat{v}(\rho)\vert
	&\lesssim
	\vert \widehat{v_{> J}}(\rho+\eta)-\widehat{v_{> J}}(\rho)\vert
	+\vert \widehat{v_{\le J}}(\rho+\eta)-\widehat{v_{\le J}}(\rho) \vert \\
	&\lesssim
	2\| \widehat{v_{> J}}\|_{L^{\infty}}
	+\|\nabla\widehat{v_{\le J}}\|_{L^{\infty}}
	\cdot2^{n_0} \\
	&\lesssim
	\| (|x|^{-2})_{> J}\|_{L^{2}}
	\| x^{2}v_{> J}\|_{L^{2}}
	+\| (|x|^{-1})_{\le J}\|_{L^{2}}
	\| x^{2}v_{\le J}\|_{L^{2}}2^{n_0} \\
	&\lesssim
	(2^{-\frac{J}{2}}+2^{\frac{J}{2}}2^{n_0})
	\bs^{2\delta_0}\epsilon_{1}.
	\end{align*}
	Choosing $J = -n_0$ we obtain
	\[
	\vert \widehat{v}(\rho+\eta)-\widehat{v}(\rho)\vert
	\lesssim
	2^{\frac{n_0}{2}}\bs^{2\delta_0}\epsilon_{1}.
	\]
	For the low frequency part, i.e. $\vert\xi\vert\lesssim s^\theta$
	we see that
	\begin{align*}
	& \vert \widehat{v}_{0,1}(s,\xi)-\widehat{v}_{0,2}(s,\xi) \vert\\
	&\lesssim
	2^{\frac{n_0}{2}}\bs^{2\delta_0}\epsilon_{1}
	\Big(
	\int_{\mathbb{R}^{3}}
	\int_{\mathbb{R}^{3}}
	\vert \widehat{v}(s,\sigma)\vert
	\vert \widehat{v}(s,\eta+\sigma) \vert
	d\sigma
	\vert\eta\vert^{-2}\beta_{n_0}^{(n_0)}(\eta)
	d\eta \\
	&\qquad\qquad\qquad\qquad
	+\int_{\mathbb{R}^{3}}
	\int_{\mathbb{R}^{3}}
	\vert \widehat{v}(s,\xi+\sigma)\vert
	\vert \widehat{v}(s,\xi) \vert
	d\sigma
	\vert\eta\vert^{-2}
	\beta_{n_{0}}^{(n_0)}(\eta) d\eta
	\Big) \\
	&\lesssim
	2^{\frac{n_0}{2}}\bs^{2\delta_0}\epsilon_{1}
	\| \vert\cdot\vert^{-2}
	\beta_{n_0}^{(n_0)}\|_{L^{1}}
	\left(
	\|\widehat{v}*\widehat{v}\|_{L^{\infty}}
	+\langle 2^{k}\rangle^{-5}
	\|\langle\cdot\rangle^{5}
	\widehat{v}\|_{L^{\infty}}
	\|\widehat{v}\|_{L^{1}}
	\right) \\
	&\lesssim
	2^{\frac{3n_0}{2}}2^{2\delta_0m}\epsilon_{1}^{3} \\
	&\lesssim
	2^{(-\frac{3+3\delta_0}{\alpha+1}
		- 5\theta + 2\delta_0)m} 2^{5\theta m}
	\langle \xi\rangle^{-5}\epsilon_{1}^{3}   \\
	&\lesssim
	2^{(-1-\delta_0)m}
	\langle \xi\rangle^{-5}\epsilon_{1}^{3},
	\end{align*}
	provided $\delta_0 \le \frac{2-\al}{3\al}$.
	
	\subsubsection{Final approximation}
	We need to show that
	\begin{align}\begin{aligned}\label{ineq:final}
	&\left\vert
	\widehat{v}_{0,2}(s,\xi)
	-i\frac{c_0}{4\pi}\int_{\mathbb{R}^{3}}
	| \mathbf z|^{-1}
	\vert\widehat{v}(s,\sigma)\vert^{2}d\sigma
	\frac{1}{\bs}
	\widehat{v}(s,\xi)
	\right\vert\\
	&\qquad\qquad\qquad\qquad\qquad\lesssim 2^{(-1-\delta_0)m}
	\langle \xi\rangle^{-5}\epsilon_{1}^{3},
	\end{aligned}\end{align}
	where
	$$
	\mathbf z = \frac{\xi}{\vert\xi\vert^{2-\alpha}} - \frac{\sigma}{\vert\sigma\vert^{2-\alpha}}.
	$$
	Observe that since the Fourier transform of $\vert\eta\vert^{-2}$
	is $\frac{2\pi^2}{\vert x\vert},$
	\[
	\Big\vert
	\int e^{i\eta\cdot x}\vert\eta\vert^{-2}
	\beta_{n_0}^{(n_0)}(\eta)d\eta
	-\frac{2\pi^2}{\vert x\vert }
	\Big\vert
	\lesssim
	\vert x\vert^{-2}2^{-n_0}.
	\]
	After change of variable and applying above inequality, \eqref{ineq:final} can be reduced to
	\begin{align}\label{ineq:fianl2}
	2^{-n_0}
	\int
	\vert s\mathbf z\vert^{-2}
	\vert\widehat{v}(s,\sigma)\vert^{2}
	d\sigma
	\lesssim 2^{(-1-\delta_0)m}\ep_1^2.
	\end{align}
	Since
	$ \vert \mathbf z\vert
	\gtrsim
	\min(\vert\sigma\vert^{\alpha-1},
	\frac{\vert\xi-\sigma\vert}{\vert\sigma\vert^{2-\alpha}})
	$ from Lemma \ref{lowbound}, we see that
	$$
	\int_{\mathbb{R}^{3}}
	\vert \mathbf z\vert^{-2}
	\vert\widehat{v}(s,\sigma)\vert^{2}
	d\sigma
	\lesssim
	\| \langle\cdot\rangle^{5}
	\widehat{v}\|_{L^{\infty}}^{2}.
	$$
	Plugging this into \eqref{ineq:fianl2},
	we obtain \eqref{ineq:final} provided $\al > \frac53$.

	\subsection{Proof of $(\ref{ineq:3})$}
	We aim to prove that
	\begin{equation}\label{goal}
	\sum_{n_1>n_0}
	\vert \partial_{s}\widehat{v}_{n_1}(s,\xi) \vert
	\lesssim
	\epsilon_{1}^{3}2^{(-1-\delta_0)m}
	\langle \xi \rangle^{-5}.
	\end{equation}
	Decomposing all the profiles dyadically, we have
	\begin{equation}\label{dyadic}
	\partial_{s}\widehat{v}_{n_1}(s,\xi)
	=\sum_{k_{1},k_{2},k_{3}\in\mathbb{Z}}
	I_{n_1}^{k_{1},k_{2},k_{3}}(s,\xi),
	\end{equation}
	where
	\begin{align*}
	&I_{n_1}^{k_{1},k_{2},k_{3}}(s,\xi)\\
	&=ic_1\iint_{\mathbb{R}^{3}\times \mathbb{R}^{3}}
	e^{is\phi(\xi,\eta,\sigma)}
	\vert\eta\vert^{-2}
	\beta_{n_1}^{(n_0)}(\eta)\beta_{k_1}(\xi+\eta)\beta_{k_2}(\sigma+\xi+\eta)\beta_{k_3}(\xi+\sigma)\\
	&\qquad\qquad\qquad\qquad\times     \widehat{v_{k_{1}}}(s,\xi+\eta)
	\overline{\widehat{v_{k_{2}}}(}s,\sigma+\xi+\eta)
	\widehat{v_{k_{3}}}(s,\xi+\sigma)
	d\eta d\sigma.
	\end{align*}
	By $v_{k_j}$ we denote $\widetilde P_{k_j}(v)$ for simplicity.
	Now Young's convolution inequality yields
	$$
	\vert I_{n_1}^{k_{1},k_{2},k_{3}}(s,\xi) \vert
	\lesssim
	2^{-\frac{n_1}{2}}
	\| \widehat{v_{k_{1}}} \|_{L^{2}}
	\| \widehat{v_{k_{2}}} \|_{L^{2}}
	\| \widehat{v_{k_{3}}} \|_{L^{2}}.
	$$
	Since
	$ \| v_{k}\|_{L^{2}}
	\lesssim
	\min\big( \langle 2^{k}\rangle^{-5}2^{\frac{3k}{2}}\epsilon_{1}, \langle 2^{k}\rangle^{-N}s^{\delta_0}\epsilon_{1}\big),
	$
	we observe that
	\begin{align}\label{sum:k}
	\sum_{k\in \mathbb{Z}}\| v_{k}(s)\|_{L^{2}}
	\lesssim \min(1,\langle 2^{k}\rangle^{-5}{s}^{\delta_0})\epsilon_{1}.
	\end{align}
	First, we consider the sum in $(\ref{dyadic})$ over $k_{1},k_{2},k_{3}$ with
	$\min (k_{1},k_{2},k_{3}) \le -\frac{2}{3}(\frac{\alpha+2}{\alpha+1}(1+2\delta_0)+\frac{5}{3}\theta)m.$
	We estimate the sum over $ k_{1}\le k_{2}\le k_{3} $ using \eqref{sum:k}, then the others can be treated similarly.
	Since $\max (k_{1},k_{2},k_{3})\sim k$, we estimate
	\begin{align*}
	& \sum_{\substack{n_1>n_0 \\
			k_1\le -\frac{2}{3}(\frac{\alpha+2}{\alpha+1}(1+2\delta_0)+\frac{5}{3}\theta)m }}
	\vert I_{n_1}^{k_{1},k_{2},k_{3}}(s,\xi)\vert \\
	&\lesssim
	\sum_{\substack{n_1>n_0 \\
			k_1\le -\frac{2}{3}(\frac{\alpha+2}{\alpha+1}(1+2\delta_0)+\frac{5}{3}\theta)m }}
	2^{-2l_1}\| \beta_{n_1}^{(n_0)}\|_{L^{2}}
	\|\widehat{v_{k_{1}}} \|_{L^{2}}
	\sum_{k_{2}\ge k_{1}}
	\| \widehat{v_{k_{2}}} \|_{L^{2}}
	\sum_{k_{3}\ge k_{2}}
	\| \widehat{v_{k_{3}}} \|_{L^{2}} \\
	&\lesssim
	2^{-\frac{n_0}{2}}
	\sum_{k_{1}\le -\frac{2}{3}(\frac{\alpha+2}{\alpha+1}(1+2\delta_0)+\frac{5}{3}\theta)m}
	\|\widehat{v_{k_{1}}} \|_{L^{2}}
	\sum_{k_{2}\ge k_{1}}
	\| \widehat{v_{k_{2}}} \|_{L^{2}}
	\langle 2^{k}\rangle^{-5}s^{\delta_0}\epsilon_{1} \\
	&\lesssim
	2^{(\frac{1+\delta_0}{\alpha+1}+\frac{5}{3}\theta + \delta_0)m}
	\langle 2^{k}\rangle^{-5}\epsilon_{1}^{3}
	\sum_{k_{1}\le-\frac{2}{3}(\frac{\alpha+2}{\alpha+1}(1+2\delta_0)+\frac{5}{3}\theta)m}
	2^{\frac{3k_{1}}{2}} \\
	&\lesssim
	\epsilon_{1}^{3}\langle 2^{k}\rangle^{-5}\
	2^{(-1-\delta_0)m}.
	\end{align*}
	Now we estimate the sum over $\max(k_{1},k_{2},k_{3})\ge\frac{2m}{N-5}$.
	We also treat only the case where $\max( k_{1},k_{2},k_{3} ) = k_{1}\sim k$.
	\begin{align*}
	\sum_{\substack{ n_1>n_0 \\
			k_{1}\ge \frac{2m}{N-5}}}
	\vert I_{n_1}^{k_{1},k_{2},k_{3}}(s,\xi)\vert
	&\lesssim
	\sum_{\substack{ n_1>n_0 \\
			k_{1}>\frac{2m}{N-5}}}
	2^{-\frac{n_1}{2}}
	\| \widehat{v_{k_{1}}} \|_{L^{2}}
	\sum_{k_{2},k_{3}\le k_{1}}
	\| \widehat{v_{k_{2}}} \|_{L^{2}}
	\| \widehat{v_{k_{3}}} \|_{L^{2}}     \\
	&\lesssim
	2^{-\frac{n_0}{2}}
	\sum_{k_{1}>\frac{2m}{N-5}}
	2^{(-N+5) k_{1}}
	s^{\delta_0}	
	\langle \xi\rangle^{-5}\epsilon_{1}^{3} \\
	&\lesssim
	\epsilon_{1}^{3}\langle\xi\rangle^{-5}
	2^{(-1-\delta_0)m}.
	\end{align*}
	
	Let us consider the remaining case:
	\begin{equation}\label{rangeki}
	n_1>n_0, \
	-\frac{2}{3}
	\big(\frac{\alpha+2}{\alpha+1}(1+2\delta_0)+\frac5{3}\theta\big)m
	\le
	k_{1},k_{2},k_{3}
	\le
	\frac{2m}{N-5}.
	\end{equation}
	Since
	$  n_1
	\le
	\max (k_{2},k_{3})+10
	\le \frac{2m}{N-5}+10,
	$
	the remaining indexes in the sum are $O(m^{4}).$
	Thus, to prove $(\ref{goal})$ it suffices to show that
	$$
	\vert I_{n_1}^{k_{1},k_{2},k_{3}}(s,\xi)\vert
	\lesssim
	\epsilon_{1}^{3}
	2^{(-1-2\delta_0)m}
	\langle \xi\rangle^{-5}
	$$
	for each  $n_1, k_{1},k_{2},k_{3}$ satisfying $\eqref{rangeki}$.
	
	Let us further decompose
	\begin{equation}\label{dyadic2}
	I_{n_1}^{k_{1},k_{2},k_{3}}
	=\sum_{n_2\in\mathbb{Z}}
	I_{n_1,n_2}^{k_{1},k_{2},k_{3}}(s,\xi),
	\end{equation}
	where
	\begin{align*}
	&I_{n_1,n_2}^{k_{1},k_{2},k_{3}}(s,\xi)\\
	&= ic_1\iint_{\mathbb{R}^{3}\times\mathbb{R}^{3}}
	e^{is\phi(\xi,\eta,\sigma)}
	|\eta|^{-2}
	\beta_{n_1}^{(n_0)}(\eta)
	\beta_{n_2}(\sigma) \beta_{k_1}(\xi+\eta)\beta_{k_2}(\xi+\eta+ \sigma)\\
	&\qquad\qquad\qquad\times
	\widehat{v_{k_{1}}}(s,\xi+\eta)
	\overline{\widehat{v_{k_{2}}}}(s,\xi+\eta+\sigma)
	\widehat{v_{k_3}}(s,\xi+\sigma)
	d\eta d\sigma.
	\end{align*}
	The above terms are zero if
	$n_2\ge \frac{2m}{N-5}+10.$
	Moreover, we can estimate
	\begin{align*}
	| I_{n_1,n_2}^{k_{1},k_{2},k_{3}}(s,\xi) |
	&\lesssim
	2^{-2n_1}\langle\xi\rangle^{-5}
	\| \bxi^{5}\widehat{v}
	\|_{L^{\infty}}^{3}
	\| \beta_{n_1}^{(n_0)}\|_{L^{1}}
	\| \beta_{n_2} \|_{L^{1}} \\
	&\lesssim
	\epsilon_{1}^{3}
	\langle\xi\rangle^{-5}
	2^{n_1}2^{3n_2}.
	\end{align*}
	This shows that
	\begin{equation}
	\sum_{\{n_2 : 3n_2+n_1\le (-1-2\delta_0)m\}}
	\vert I_{n_1,n_2}^{k_{1},k_{2},k_{3}}(s,\xi)\vert
	\lesssim
	\epsilon_{1}^{3}
	\langle\xi\rangle^{-5}
	2^{(-1-2\delta_0)m}.
	\end{equation}
	We are then left again with a summation over $n_2$ with only $O(m)$ terms.
	Therefore, $(\ref{goal})$ will be a consequence of the following.
	
	\begin{prop}
		Let $I_{n_1,n_2}^{k_{1},k_{2},k_{3}}$ be defined by
		\eqref{dyadic2}, and assume that	
		$\| u\|_{\Sigma_{T}}$$\le\epsilon_{1}.$
		Then one has
		$$
		\vert I_{n_1,n_2}^{k_{1},k_{2},k_{3}}(s,\xi)\vert
		\lesssim
		\epsilon_{1}^{3}
		\langle\xi\rangle^{-5}
		2^{(-1-3\delta_0)m},
		$$
		whenever
		\begin{align}\begin{aligned}{\label{restriction}}
		&-\frac{2}{3}
		\big(\frac{\alpha+2}{\alpha+1}(1+2\delta_0)
		+\frac5{3}\theta\big)m
		\le k_{1},k_{2},k_{3}
		\le \frac{2m}{N-5}, \\
		&\qquad\qquad n_1 > n_0 \
		\textrm{and} \
		n_1 + 3n_2\ge(-1-2\delta_0)m.
		\end{aligned}
		\end{align}	
	\end{prop}
	
	\subsubsection{Case: $\max(k_{1},k_{2})\le n_1$}
	Recall that
	\begin{align*}
	I_{n_1,n_2}^{k_{1},k_{2},k_{3}}(s,\xi)
	&=ic_1\iint_{\mathbb{R}^{3}\times\mathbb{R}^{3}}
	m_{n_1, n_2}^{k_1, k_2, k_3}(\eta,\sigma)e^{is\phi(\xi, \eta, \sigma)} \\
	&\qquad\qquad\qquad\times    \widehat{v_{k_{1}}}(s,\xi+\eta)
	\overline{\widehat{v_{k_{2}}}(}s,\xi+\eta+\sigma)
	\widehat{v_{k_3}}(s,\xi+\sigma)
	d\eta d\sigma,
	\end{align*}
	where
	$$
	m_{n_1,n_2}^{k_1, k_2,k_3}(\eta,\sigma)
	=  \vert\eta\vert^{-2}
	\beta_{n_1}(\eta)\beta_{n_2}(\sigma)\beta_{k_1}(\xi+\eta)\beta_{k_2}(\xi+\eta+\sigma).
	$$
	It can be readily checked that $m_{n_1,n_2}^{k_1, k_2,k_3}$ verifies the assumption of Lemma \ref{LemmaB} with $A(m_{n_1,n_2}^{k_1, k_2,k_3}) = 2^{-2n_1}$. Since $|\xi| \le |\xi+\eta+\sigma|+|\xi+\eta|+|\xi+\sigma|$ and hence $ \langle\xi\rangle^{5} \lesssim 2^{\frac{10m}{N-5}}$
	by $\eqref{restriction}$,
	we can estimate
	\begin{align*}
	\vert I_{n_1,n_2}^{k_{1},k_{2},k_{3}}(s,\xi)\vert
	&\lesssim
	2^{-2n_1}
	\| {v_{k_{1}}}\|_{L^{2}}
	\| {v_{k_{2}}}\|_{L^{2}}
	\| \widetilde P_{k_3}(u) \|_{L^{\infty}} \\
	&\lesssim
	s^{-\frac{3}{2}}
	2^{-2\max(k_1, k_2)}
	2^{\frac{3k_{1}}{2}}
	\langle 2^{k_{1}}\rangle^{-5}
	2^{\frac{3k_{2}}{2}}
	\langle 2^{k_{2}}\rangle^{-5}
	\epsilon_{1}^{3} \\
	&\lesssim
	2^{-\frac{3}{2}m}2^\frac{12m}{N-5}\bxi^{-5}
	\epsilon_{1}^{3} \\
	&\lesssim
	2^{(-1-3\delta_{0})m}
	\langle\xi\rangle^{-4}
	\epsilon_{1}^{3},
	\end{align*}
	provided $\delta_0 \le \frac16 - \frac{4}{N-5}$.
	
	\subsubsection{ Case: $\max(k_{1},k_{2})\ge n_1, \vert k_{1}-k_{2}\vert\ge 10$}
	\begin{lm}\label{derivative}
		Suppose	$f$ satisfies the condition of Proposition \ref{scattering}.
		Then we have
		\begin{align*}
		\| \nabla \widehat{v_{k}}(s)\|_{L^{2}}
		&\lesssim
		\min(2^{\frac{k}{2}},\langle 2^{k}\rangle^{-3})
		{s}^{2\delta_0}\epsilon_{1}, \\
		\|  \nabla^{2}\widehat{v_{k}}(s)\|_{L^{2}}	
		&\lesssim
		2^{-k}\langle 2^{k}\rangle^{-2}{s}^{2\delta_0}\epsilon_{1}.
		\end{align*}
	\end{lm}
	\begin{proof}[Proof of Lemma \ref{derivative}]
		By Sobolev inequality we have
		\begin{align*}
		\|\widetilde\beta_{k}\widehat{xv}(s)\|_{L^{2}}
		\lesssim
		\|\widetilde\beta_{k}\|_{L^{3}}\|\widehat{xv}(s)\|_{L^{6}}
		\lesssim
		2^{k} \| \widehat{x^{2}v}(s)\|_{L^{2}}
		\lesssim
		2^{k}{s}^{2\delta_0}\epsilon_{1}.
		\end{align*}
		Also it is easy to see that
		\begin{align*}
		\|\widetilde\beta_{k}\widehat{xv}(s)\|_{L^{2}}
		\lesssim
		\langle 2^{k}\rangle^{-3} \|xv(s)\|_{H^{3}}
		\lesssim
		\langle 2^{k}\rangle^{-3}\bs^{\delta_0}\epsilon_{1}.
		\end{align*}
		Then, we have
		\begin{align*}
		\| \nabla\widehat{v_{k}}(s)\|_{L^{2}}
		&\lesssim
		2^{\frac{k}{2}}
		\|\widehat{v_{k}}\|_{L^{\infty}}	
		+\|\widetilde\beta_{k}\widehat{xv}(s)\|_{L^{2}}\\
		&\lesssim
		2^{\frac{k}{2}}\langle 2^{k}\rangle^{-5}
		\| \bxi^{5}\widehat{v}\|_{L^{\infty}}
		+ \min (2^{k},\langle 2^{k}\rangle^{-3}){s}^{2\delta_0}
		\epsilon_{1}\\
		&\lesssim
		\min(2^{\frac{k}{2}},\langle 2^{k}\rangle^{-3})
		{s}^{2\delta_0}
		\epsilon_{1}.
		\end{align*}
		And we estimate
		\begin{align*}
		\| \nabla^{2}\widehat{v_{k}}(s)\|_{L^{2}}
		&\lesssim
		2^{-\frac{k}{2}}
		\|\widehat{v_{k}}\|_{L^{\infty}}
		+ \|\nabla\widetilde\beta_{k}\otimes\nabla\widehat{v}\|_{L^{2}}
		+ \|\widetilde \beta_{k}\nabla^{2}\widehat{v}\|_{L^{2}}  \\
		&\lesssim
		2^{-\frac{k}{2}}\langle 2^{k}\rangle^{-5}\epsilon_{1}
		+2^{-k}\langle 2^{k}\rangle^{-2}\| xv\|_{H^{2}}
		+\langle 2^{k}\rangle^{-2}\|x^{2}v\|_{H^{2}} \\
		&\lesssim
		2^{-k}\langle 2^{k}\rangle^{-2}{s}^{2\delta_0}\epsilon_{1}.
		\end{align*}
	\end{proof}
	
	Here it should be noticed that
	in this range the following additional condition holds:
	\begin{equation}\label{conditiononk3}
	2^{k_{3}}\lesssim 2^{\max(k_{1},k_{2})}.
	\end{equation}
	This leads us to
	\begin{equation}\label{conditiononk}
	2^{k} \sim  2^{\max(k_{1},k_{2})} \lesssim 2^\frac{2m}{N-5}.
	\end{equation}

	In order to make an efficient integration by parts w.r.t. $\eta$ in the expression $(\ref{dyadic2})$ we introduce the identity
	$$
	e^{is\phi(\xi,\eta,\sigma)}
	=\frac{1}{is}
	\mathbf P(\xi,\eta,\sigma)   \cdot  \nabla_\eta[ e^{is\phi(\xi,\eta,\sigma)}],  \quad\mathbf P(\xi,\eta,\sigma) = \frac{\nabla_\eta [\phi(\xi,\eta,\sigma)]}{\vert\nabla_{\eta}[\phi(\xi,\eta,\sigma)]\vert^{2}}.
	$$
	Then we write
	\begin{equation*}
	I_{n_1,n_2}^{k_{1},k_{2},k_{3}}(s,\xi)
	= -\widetilde { I}_{1}(s,\xi) - \widetilde { I}_{2}(s,\xi),
	\end{equation*}
	where
	\begin{align*}
	\widetilde { I}_{1}(s,\xi)
	= \frac{c_1}{s}
	\iint_{\mathbb{R}^{3}\times\mathbb{R}^{3}}
	e^{is\phi(\xi,\eta,\sigma)}
	\mathbf Q(\xi,\eta,\sigma)\cdot
	&\nabla_{\eta}\left[
	\widehat{v_{k_{1}}}(s,\xi+\eta)
	\overline{\widehat{v_{k_{2}}}(}s,\xi+\eta+\sigma)
	\right]\\
	&\qquad\qquad\times \widehat{v_{k_3}(v)}(s,\xi+\sigma)
	d\eta d\sigma,
	\end{align*}
	\begin{align*}
	\widetilde { I}_{2}(s,\xi)
	= \frac{c_1}{s}
	\iint_{\mathbb{R}^{3}\times\mathbb{R}^{3}}
	e^{is\phi(\xi,\eta,\sigma)}
	\nabla_\eta \cdot [ \mathbf Q(\xi,\eta,\sigma)]
	&\widehat{v_{k_{1}}}(s,\xi+\eta)
	\overline{\widehat{v_{k_{2}}}(}s,\xi+\eta+\sigma)\\
	&\;\;\quad\times\widehat{v_{k_3}(v)}(s,\xi+\sigma)
	d\eta d\sigma
	\end{align*}
	and
	\begin{align*}
	\mathbf Q(\xi,\eta,\sigma)
	= \mathbf P(\xi,\eta,\sigma)
	\vert\eta\vert^{-2}
	\beta^{(n_0)}_{n_1}(\eta)\beta_{n_2}(\sigma)
	\beta_{k_{1}}(\xi+ \eta)
	\beta_{k_{2}}(\xi+\eta+\sigma).
	\end{align*}
	
	We first estimate $\widetilde { I}_{1}$.
	\begin{lm}\label{lm:ql}
		Consider
		\begin{equation}\label{condi:kj}
		\vert k_{1}-k_{2}\vert\ge 10
		, \ \max(k_{1},k_{2})\ge n_1
		\end{equation}
		and let
		$$
		\mathbf Q_{\gamma}^{\ell_1}(\xi,\eta,\sigma)
		= \nabla_{\eta}^{\ell_1};
		\left[
		\mathbf P(\xi,\eta,\sigma)
		\vert\eta\vert^{-\gamma}
		\beta^{n_0}_{n_1}(\eta)\beta_{n_2}(\sigma)
		\beta_{k_{1}}(\eta+\xi)
		\beta_{k_{2}}(\eta+\sigma+\xi)
		\right]
		$$
		for nonnegative integer $\gamma$ and $\ell_1$. Then $ \mathbf Q_{\gamma}^{\ell_1} $ satisfies the assumption of
		Lemma \ref{LemmaB} with
		\begin{equation}\label{AP}
		\| \mathcal{F}^{-1}
		(\mathbf Q_{\gamma}^{\ell_1})
		\|_{L^{1}(\mathbb{R}^{6})}
		\lesssim
		2^{-\max(k_{1},k_{2})(\alpha-1)}2^{-\gamma n_1}
		2^{-\ell_1\min(k_{1},n_1)}.
		\end{equation}
	\end{lm}
	\begin{proof}[Proof of Lemma\ref{lm:ql}]
		Performing integration by parts and then change of variables, we have	
		\begin{align*}
		&\left\| \iint_{\mathbb{R}^{3}\times\mathbb{R}^{3}}
		\mathbf Q_{\gamma}^{\ell_1}(\xi,\eta,\sigma)
		e^{ix\eta}e^{iy\sigma}
		d\eta d\sigma \right\|_{L_{x,y}^{1}} \\
		&\qquad\qquad\qquad\sim\left\|
		x^{\ell_1}
		\iint_{\mathbb{R}^{3}\times\mathbb{R}^{3}}
		\mathbf Q_{\gamma}^{0}(\xi,\eta,\sigma)
		e^{ix\eta}e^{iy\sigma}
		d\sigma d\eta
		\right\|_{L_{x,y}^{1}} \\
		&\qquad\qquad\qquad\sim\left\|
		x^{\ell_1}
		\iint_{\mathbb{R}^{3}\times\mathbb{R}^{3}}
		\widetilde{\mathbf Q_{\gamma}^{0}}(\xi,\eta,\sigma)
		e^{ix\eta}e^{iy\sigma}
		d\eta d\sigma\right\|_{L_{x,y}^{1}},
		\end{align*}
		where
		$$
		\widetilde{\mathbf Q_{\gamma}^{0}}(\xi,\eta,\sigma)
		=\frac{\vert\sigma\vert^{\alpha-2}\sigma -\vert\eta\vert^{\alpha-2}\eta}
		{\big\vert \vert\sigma\vert^{\alpha-2}\sigma -\vert\eta\vert^{\alpha-2}\eta \big\vert^{2}}
		\vert\eta-\xi\vert^{-\gamma}
		\beta_{n_1}(\eta-\xi)\beta_{n_2}(\sigma-\eta)
		\beta_{k_{1}}(\eta)\beta_{k_{2}}(\sigma).
		$$	
		In the range of \eqref{condi:kj}, we see that $\widetilde{\mathbf Q_{\gamma}^{0}}$ verifies the following inequality: for any positive integers $\widetilde \ell_1,\widetilde \ell_2$
		\begin{align*}
		\vert \nabla_{\eta}^{\widetilde\ell_1}\nabla_{\sigma}^{\widetilde\ell_2}
		\widetilde{\mathbf Q_{\gamma}^{0}}(\xi,\eta,\sigma) \vert
		&\lesssim
		2^{-\max(k_{1},k_{2})(\alpha-1)}
		2^{-\gamma n_1}
		2^{-\min(k_{1},n_1)\widetilde\ell_1}
		2^{-k_{2}\widetilde\ell_2} \\
		&\qquad\qquad \times \widetilde \beta_{n_1}(\eta-\xi)\widetilde \beta_{n_2}(\sigma-\eta)
		\widetilde \beta_{k_{1}}(\eta)
		\widetilde \beta_{k_{2}}(\sigma).
		\end{align*}	
		This gives us the desired result.
	\end{proof}
	
	Now applying Lemma \ref{LemmaB} to $\widetilde{ I}_{1}$ with \eqref{AP} and Lemma \ref{derivative}, we estimate
	\begin{align}
	\begin{aligned}\label{estimate:Worst case}
	\vert \widetilde { I}_{1}(s,\xi)\vert
	&\lesssim
	s^{-1}
	2^{-\max(k_{1},k_{2})(\alpha-1)}
	2^{-2n_1}\\
	&\qquad \times
	\big(
	\|\nabla \widehat{v_{k_{1}}}\|_{L^{2}}
	\|\widehat{v_{k_{2}}}\|_{L^{2}}
	+\|\widehat{v_{k_{1}}}\|_{L^{2}}
	\|\nabla \widehat{v_{k_{2}}}\|_{L^{2}}
	\big)
	\| u_{k_{3}}(s)\|_{L^{\infty}}\\
	&\lesssim
	2^{-m}
	2^{-\max(k_{1},k_{2})(\alpha-1)}
	2^{-2n_1}
	\big(
	\langle 2^{k_{1}}\rangle^{-3}2^{\frac{k_{1}}{2}}2^{2m\delta_0}
	2^{\frac{3k_{2}}{2}}\langle 2^{k_{2}}\rangle^{-5} \\
	&\qquad+
	2^{\frac{3k_{1}}{2}}\langle 2^{k_{1}}\rangle^{-5}
	\langle 2^{k_{2}}\rangle^{-3}
	2^{\frac{k_{2}}{2}}2^{2m\delta_0}
	\big)
	2^{-\frac{3m}{2}}\epsilon_{1}^{3} \\
	&\lesssim
	2^{(-\frac{5}{2}+2\delta_0)m}2^{-2n_0}
	\langle 2^{k}\rangle^{-3}\epsilon_{1}^{3}\quad (n_1 > n_0)\\
	&\lesssim 2^{(-\frac{5}{2}+2\delta_0)m}2^{(\frac{4+4\delta_0}{\al+1} + \frac{20}{3}\theta)m}2^\frac{4m}{N-5}
	\langle 2^{k}\rangle^{-5}\epsilon_{1}^{3}\\
	&\lesssim
	2^{(-1-3\delta_0)m}
	\langle 2^{k}\rangle^{-5}
	\epsilon_{1}^{3},
	\end{aligned}
	\end{align}
	since $\alpha>5/3$ and $0 < \delta_0 \le \delta_3 := \frac1{5\al+9}(\al-\frac53 - \frac{4(\al+1)}{N-5})$. We have used \eqref{conditiononk} for the third and last inequality. From now on, we will use this technique repeatedly to derive $\langle 2^{k}\rangle^{-5}$ term without mentioning it.
	
	Let us move onto $\widetilde { I}_{2}$. Due to the additional term $2^{-\min(k_{1},n_1)}$ in \eqref{AP}, we cannot apply the Lemma \ref{LemmaB} to $\widetilde {I}_{2}$ as before.
	So, we perform an additional integration by parts in $\widetilde {I}_{2}$
	\begin{align}
	&\widetilde {I}_{2}(s,\xi)
	=\widetilde {J_{1}}(s,\xi) + \widetilde {J_{2}}(s,\xi), \label{J1J2}\\
	&\widetilde { J_{1}}(s,\xi)
	=
	i\frac{c_1}{s^{2}}
	\iint_{\mathbb{R}^{3}\times\mathbb{R}^{3}}
	e^{is\phi(\xi,\eta,\sigma)}
	[\nabla_{\eta}\cdot\mathbf Q
	(\xi,\eta,\sigma)] \nonumber  \\
	&\qquad \times
	\mathbf P(\xi,\eta,\sigma) \cdot \nabla_{\eta}\left(
	\widehat{v_{k_{1}}}(s,\xi+\eta)
	\overline{\widehat{v_{k_{2}}}(}s,\xi+\eta+\sigma)
	\right)
	\widehat{v_{k_{3}}}(s,\sigma+\xi)
	d\sigma d\eta, \nonumber\\
	&\widetilde {J_{2}}(s,\xi)
	=
	i\frac{c_1}{s^{2}}
	\iint_{\mathbb{R}^{3}\times\mathbb{R}^{3}}
	e^{is\phi(\xi,\eta,\sigma)}
	\nabla_{\eta}\cdot [((\nabla_{\eta} \cdot \mathbf Q) \mathbf P)
	(\xi,\eta,\sigma)]\nonumber \\
	&\qquad\qquad\qquad\qquad\times
	\widehat{v_{k_{1}}}(s,\xi+\eta)
	\overline{\widehat{v_{k_{2}}}(}s,\xi+\eta+\sigma)
	\widehat{v_{k_{3}}}(s,\sigma+\xi)
	d\sigma d\eta.\nonumber
	\end{align}
	By Lemma~\ref{lm:ql} and Lemma~\ref{LemmaB2} we have
	\begin{align*}
	\| \mathcal{F}^{-1}
	((\nabla_{\eta}\cdot \mathbf Q)\mathbf P)
	\|_{L^{1}(\mathbb{R}^{6})}
	&\lesssim
	\| \mathcal{F}^{-1}
	(    \mathbf Q_{2}^{1} \otimes
	\mathbf Q_{0}^{0})
	\|_{L^{1}(\mathbb{R}^{6})} \\
	&\lesssim
	2^{-2\max(k_{1},k_{2})(\alpha-1)}
	2^{-2n_1}
	2^{-\min(k_{1},n_1)}.
	\end{align*}
	Applying Lemma \ref{LemmaB} to $\widetilde {J_{1}}$ with above bound and using Lemma \ref{derivative}, we obtain
	\begin{align*}
	\vert \widetilde { J_{1}}(s,\xi)\vert
	&\lesssim
	s^{-2}2^{-2\max(k_{1},k_{2})(\alpha-1)}
	2^{-2n_1}
	2^{-\min(k_{1},n_1)} \\
	&\qquad\qquad\times
	\Big[
	\|\nabla\widehat{v_{k_{1}}}\|_{L^{2}}
	\|\widehat{v_{k_{2}}}\|_{L^{2}}
	+\|\widehat{v_{k_{1}}}\|_{L^{2}}
	\|\nabla\widehat{v_{k_{2}}}\|_{L^{2}}
	\Big]
	\| u_{k_{3}}\|_{L^{\infty}} \\
	&\lesssim
	2^{-\frac{7}{2}m+2\delta_{0}m}
	\langle 2^{k}\rangle^{-3}
	2^{-2n_1}
	2^{-\min(k_{1},n_1)}
	\epsilon_{1}^{3} \\
	&\lesssim
	s^{-\frac{7}{2}m+2\delta_{0}m}
	\langle 2^{k}\rangle^{-5}
	2^{m\big(\frac{4}{N-5}
		+\frac{4+4\delta_0}{\alpha+1}+\frac{20\theta }{3}
		+\frac{2}{3}
		\big(\frac{\alpha+2}{\alpha+1}(1+2\delta_0)
		+\frac{5\theta}{3}\big)\big)}
	\epsilon_{1}^{3}.
	\end{align*}
	If $\al > \frac53$, then $$\frac{2}{3}
	\big(\frac{\alpha+2}{\alpha+1}(1+2\delta_0)
	+\frac{5\theta}{3}\big) < 1.$$
	And hence we get as in \eqref{estimate:Worst case} that
	$$\vert \widetilde { J_{1}}(s,\xi)\vert \lesssim
	2^{(-1-3\delta_0)m}\langle 2^{k}\rangle^{-5}\epsilon_{1}^{3}.$$
	To estimate $\widetilde {J_{2}}$ in \eqref{J1J2},
	we only use the pointwise bound
	$$
	\vert
	\nabla_{\eta}\cdot((\nabla_{\eta}\cdot\mathbf Q)\mathbf P)
	(\xi,\eta,\sigma) \vert
	\lesssim 2^{-2\max(k_{1},k_{2})(\alpha-1)}
	2^{-2n_1}2^{-2\min(k_{1},k_{2},n_1)},
	$$
	and then we see that
	\begin{align*}
	\vert \widetilde { J_{2}}(s,\xi) \vert
	&\lesssim
	\frac{1}{s^{2}}
	\| \nabla_{\eta}\cdot((
	\nabla_{\eta}\cdot \mathbf Q) \mathbf P)
	\|_{L_{\eta,\sigma}^{\infty}}
	\| \beta_{n_1}\widehat{v_{k_{1}}}(\xi+\cdot)\|_{L^{1}}
	\| \widehat{v_{k_{2}}} \|_{L^{2}}
	\| \widehat{v_{k_{3}}} \|_{L^{2}} \\
	&\lesssim
	s^{-2}
	2^{-2\max(k_{1},k_{2})(\alpha-1)}
	2^{-2n_1}
	2^{-2\min(k_{1},k_{2},n_1)}
	\min(2^{3n_1}, 2^{3k_{1}}) \\
	&\qquad\qquad\qquad\qquad\times
	\langle 2^{k_{1}}\rangle^{-5}
	2^{\frac{3k_{2}}{2}}\langle2^{k_{2}}\rangle^{-5}
	2^{\frac{3k_{3}}{2}}\langle2^{k_{3}}\rangle^{-5}\epsilon_1^3.
	\end{align*}
	Using \eqref{conditiononk3}, we see that
	\begin{displaymath}
	\vert \widetilde {J_{2}}(s,\xi) \vert
	\lesssim
	\left\{
	\begin{array}{ll}
	\epsilon_{1}^{3}
	2^{-2m}\langle 2^{k} \rangle^{-5}
	2^{-\frac{k_{2}}{2}},
	&\ \textrm{if} \
	\min(k_{1},k_{2},n_1)=k_{2}, \\
	\epsilon_{1}^{3}2^{-2m}
	\langle 2^{k} \rangle^{-5}
	2^{-n_1},
	& \ \textrm{otherwise}.
	\end{array} \right.
	\end{displaymath}
	And both cases give the desired bound of $\widetilde {J_{2}}$
	because of $(\ref{restriction})$.

	\subsubsection{ Case: $\max(k_{1},k_{2})\ge n_1,
		\vert k_{1}-k_{2}\vert\le 10 $}
	$\newline$
	First, let us observe that
	\[
	2^{k_{3}}\lesssim 2^{k} \sim 2^{k_{1}}\sim 2^{k_{2}}
	\lesssim 2^\frac{2m}{N-5}.
	\]
	Since $\vert \xi+\eta\vert \sim 2^{k_{1}} $,
	$ \vert \xi+\eta+\sigma \vert \sim 2^{k_{2}} $
	with $k_{1}\sim k_{2}$,
	$\vert \eta \vert \sim 2^{n_1}$
	and
	$\vert \sigma \vert \sim 2^{n_2}$,
	one has
	\begin{equation}\label{ineq:PnQn}
	\begin{split}
	&\vert\nabla_{\eta}^{\ell_1}\nabla_{\sigma}^{\ell_2}
	\mathbf P(\xi,\eta,\sigma)\vert
	\lesssim
	2^{k_{1}(2-\alpha)}2^{-n_2}
	2^{-n_1\ell_1}2^{-n_2\ell_2}, \\
	&\vert\nabla_{\eta}^{\ell_1}\nabla_{\sigma}^{\ell_2}
	\mathbf Q(\xi,\eta,\sigma)\vert
	\lesssim
	2^{k_{1}(2-\alpha)}2^{-n_2}
	2^{-2n_1}
	2^{-n_1\ell_1}2^{-n_2\ell_2}
	\end{split}
	\end{equation}
	for any positive integers $\ell_1, \ell_2$.
	These can be induced from mean-value theorem,
	Lemma \ref{lowbound} and the observation that
	$n_2\le \max(k_{1},k_{2})+10$.
	As a consequence,
	\begin{equation}\label{AQn}
	\|
	\mathcal{F}^{-1}(\mathbf P \otimes \mathbf Q)
	\|_{L^{1}(\mathbb{R}^{6})}
	\lesssim
	2^{k_{1}(2-\alpha)}2^{-2n_1}2^{-2n_2},
	\end{equation}
	
	\begin{equation}\label{APPQ}
	\| \mathcal{F}^{-1}( (\nabla_\eta \cdot \mathbf Q)\mathbf P)\|_{L^{1}(\mathbb{R}^{6})}
	\lesssim
	2^{2k_{1}(2-\alpha)}2^{-3n_1}2^{-2n_2}.
	\end{equation}
	Now if we try to estimate $\widetilde { I}_{1}$ applying
	the Lemma \ref{LemmaB} with \eqref{AQn} as before, then due to $2^{-n_2}$ term we cannot obtain the desired result in this case.
	So, we perform an integration by parts twice
	in the expression for $I_{n_1,n_2}^{k_{1},k_{2},k_{3}}$
	in \eqref{dyadic2}. With the previous calculation in \eqref{J1J2}, we write
	$$
	\vert I_{n_1,n_2}^{k_{1},k_{2},k_{3}}(s,\xi) \vert
	\lesssim
	\vert \widetilde { J_{1}}(s,\xi)\vert
	+\vert \widetilde { J_{2}}(s,\xi)\vert
	+\vert \widetilde { J_{3}}(s,\xi)\vert,
	$$
	\begin{align*}
	\widetilde { J_{3}}(s,\xi)
	=
	i\frac{c_1}{s^{2}}
	&\iint_{\mathbb{R}^{3}\times\mathbb{R}^{3}}
	e^{is\phi(\xi,\eta,\sigma)}
	(\mathbf P \otimes \mathbf Q)(\xi,\eta,\sigma) ;\\
	&\quad\quad \nabla_{\eta}^2
	\big(\widehat{v_{k_{1}}}(s,\xi+\eta)
	\overline{\widehat{v_{k_{2}}}(}s,\xi+\eta+\sigma)\big)
	\widehat{v_{k_{3}}}(s,\sigma+\xi)
	d\sigma d\eta.
	\end{align*}
	First applying Lemma \ref{LemmaB} to $\widetilde { J_{3}}$ with  \eqref{AQn}, and then Lemma \ref{derivative}, we have as in \eqref{estimate:Worst case} that
	\begin{equation}
	\begin{split}
	\vert \widetilde {J_{3}}(s,\xi)\vert
	&\lesssim
	s^{-2}2^{2k_{1}(2-\alpha)}2^{-2n_1}2^{-2n_2}
	\| u_{3}\|_{L^{\infty}}\nonumber \\
	&\ \  \times \Big(\| \nabla^2\widehat{v_{k_{1}}}\|_{L^{2}}
	\| \widehat{v_{k_{2}}}\|_{L^{2}}
	+\| \nabla\widehat{v_{k_{1}}}\|_{L^{2}}\| \nabla\widehat{v_{k_{2}}}\|_{L^{2}}
	+\| \widehat{v_{k_{1}}}\|_{L^{2}}
	\|\nabla^2\widehat{v_{k_{2}}}\|_{L^{2}}
	\Big) \\
	&\lesssim
	2^{(-\frac{7}{2}+4\delta_0)m}
	2^{-2n_1}2^{-2n_2}\langle 2^{k}\rangle^{-6}\epsilon_{1}^{3}\nonumber \\
	&\lesssim 2^{(-\frac{7}{2}+4\delta_0)m} 2^{-\frac23(n_1 + 3n_2)}2^{-\frac{4n_1}{3}}\langle 2^{k}\rangle^{-5}\epsilon_{1}^{3}\nonumber \\
	&\lesssim 2^{(-\frac{7}{2}+4\delta_0)m}2^{\frac23(1+2\delta_0)m}2^{-\frac{4n_1}{3}} \langle 2^{k}\rangle^{-5}
	\epsilon_{1}^{3}\nonumber\\
	&\lesssim
	2^{(-1-3\delta_0)m}
	\langle 2^{k}\rangle^{-5}
	\epsilon_{1}^{3},
	\end{split}
	\end{equation}
	provided $\delta_0 \le \frac1{100}$. In the third inequality, we have used $n_1 > n_0$ and
	$n_1 + 3 n_2\ge -(1+2\delta_0)m$.

	To estimate $\widetilde { J_{1}}$, we also apply Lemma \ref{LemmaB}
	with \eqref{APPQ}:
	\begin{align*}
	&   | \widetilde { J_{1}}(s,\xi) |\\
	&\lesssim
	s^{-2}2^{2k_{1}(2-\alpha)}2^{-3n_1}2^{-2n_2}
	\big(\| \nabla\widehat{v_{k_{1}}}\|_{L^{2}}
	\| \widehat{v_{k_{2}}}\|_{L^{2}}
	+ \| \widehat{v_{k_{1}}}\|_{L^{2}}
	\| \nabla\widehat{v_{k_{2}}}\|_{L^{2}}
	\big)
	\| u_{3}\|_{L^{\infty}} \\
	&\lesssim
	2^{-\frac{7m}{2}}2^{2\delta_0m}2^{-\frac23(n_1+3n_2)}
	2^{-\frac{7n_1}3}\langle 2^{k}\rangle^{-5}\epsilon_{1}^{3} \\
	&\lesssim
	2^{(-\frac{7}{2} +2\delta_0)m}2^{(\frac23 + \frac{4\delta_0}3)m}2^{(\frac{14+14\delta_0}{3(\al+1)} + \frac{21\al-35}{36(\al+1)})m}
	\langle 2^{k}\rangle^{-5}\epsilon_{1}^{3} \\
	&\lesssim
	2^{(-1-3\delta_0)m}
	\langle 2^{k}\rangle^{-5}
	\epsilon_{1}^{3}
	\end{align*}
	for $\delta_0 \le \frac1{100}$. $\widetilde { J_{2}}$ can be treated with the pointwise bound for
	$\nabla_{\eta}((\nabla_{\eta}\cdot \mathbf Q)\mathbf P)$.
	From \eqref{ineq:PnQn}, it follows that
	$$
	\vert
	\nabla_{\eta}((\nabla_{\eta}\cdot \mathbf Q)\mathbf P)
	(\xi,\eta,\sigma) \vert
	\lesssim
	2^{-2n_2}2^{(4-2\alpha)k_{1}}2^{-4n_1}.
	$$
	Note that in this range we have $2^{n_2} \lesssim 2^{\max(k_1, k_2)}$.
	Thus if $\delta_0 \le \frac1{100}$, then we see that
	\begin{align*}
	\vert \widetilde {J_{2}}(s,\xi) \vert
	&\lesssim
	\frac{1}{s^{2}}
	\| \nabla_{\eta}((\nabla_{\eta}\cdot \mathbf Q)\mathbf P)
	\|_{L_{\eta,\sigma}^{\infty}}
	\| \beta_{n_1}^{(n_0)}\|_{L^{1}}
	\| \beta_{n_2}\|_{L^{1}}
	\prod_{i=1}^3\| \widehat{v_{k_{i}}}\|_{L^{\infty}} \\
	&\lesssim
	s^{-2}
	2^{(4-2\alpha)k_{1}}2^{n_2}2^{-n_1}
	\langle 2^{k_{1}}\rangle^{-5}\langle 2^{k_{2}}\rangle^{-5}
	\langle2^{k_{3}}\rangle^{-5}\epsilon_{1}^{3}\\
	&\lesssim
	s^{-2}
	2^{(4-2\alpha)k_{1}}2^{\max(k_1, k_2)}2^{-n_1}
	\langle 2^{k_{1}}\rangle^{-5}\langle 2^{k_{2}}\rangle^{-5}
	\langle2^{k_{3}}\rangle^{-5}\epsilon_{1}^{3}\\
	&\lesssim
	2^{-2m}2^{-n_1}
	\langle2^{k}\rangle^{-5}\epsilon_{1}^{3} \\
	&\lesssim  2^{-2m}2^{(\frac{2+2\delta_0}{\al+1} + \frac{3\al-5}{12(\al+1)})}\langle2^{k}\rangle^{-5}\epsilon_{1}^{3} \\
	&\lesssim
	2^{(-1-3\delta_0)m}
	{\langle 2^{k}\rangle^{-5}}
	\epsilon_{1}^{3}.
	\end{align*}
	
	\section*{Acknowledgements}
	The authors would like to thank the anonymous referees for their careful reading of our manuscript and the valuable comments.
	
	This work was supported in part by National Research Foundation of Korea (NRF-2015R1D1A1A09057795).
	C. Yang was also supported in part by National Research Foundation of Korea (NRF-2015R1A4A1041675)
	
	\section{appendix}
	
	\begin{lm}\label{LemmaB}
		Assume that
		$\mathbf m\in L^{1}(\mathbb{R}^{d}\times\mathbb{R}^{d}), d \ge 1$
		satisfies
		\begin{equation}\label{ineq:m}
		\Big\|
		\int_{\mathbb{R}^{d}\times\mathbb{R}^{d}}
		\mathbf m(\eta,\sigma)e^{ix\eta}e^{iy\sigma}
		d\eta d\sigma
		\Big\|_{L_{x,y}^{1}}
		\le A(\mathbf m).
		\end{equation} \\
		Then for any $(p,q,r)$ with $1 \le p, q, r \le \infty$ and $\frac{1}{p}+\frac{1}{q}+\frac{1}{r}=1,$ \\
		$$
		\left\vert
		\int_{\mathbb{R}^{d}\times\mathbb{R}^{d}}
		\mathbf  m(\eta,\sigma)
		\widehat{v}(\eta)\widehat{g}(\sigma)\widehat{h}(\pm\eta\pm\sigma)
		d\eta d\sigma
		\right\vert
		\lesssim
		A(m)
		\| v \|_{L^{p}}
		\| g \|_{L^{q}}
		\| h \|_{L^{r}}.
		$$
		Moreover, for all $p,q$ with
		$\frac{1}{p}+\frac{1}{q}=\frac{1}{2},$
		one has
		$$
		\Big\|
		\int_{\mathbb{R}^{d}}
		m(\xi,\eta)
		\widehat{v}(\xi\pm\eta)
		\widehat{g}(\eta)
		d\eta
		\Big\|_{L_{\xi}^{2}}
		\lesssim
		A(m)
		\| v \|_{L^{p}}
		\| g \|_{L^{q}}.
		$$
	\end{lm}
	
	For this see Appendix B.2 in \cite{pu}. As a corollary we have
	
	\begin{lm}\label{LemmaB2}
		Suppose	
		$ \mathbf P, \mathbf Q \in L^{1}(\mathbb{R}^{d}\times\mathbb{R}^{d}) $
		satisfies $(\ref{ineq:m})$ with $A(\mathbf P), A(\mathbf Q).$
		Then,
		$$
		\Big\|
		\iint_{\mathbb{R}^{d}\times\mathbb{R}^{d}}
		\mathbf P(\eta,\sigma)\otimes \mathbf Q(\eta,\sigma)
		e^{ix\eta}
		e^{iy\sigma}d\sigma d\eta
		\Big\|_{L_{x,y}^{1}}
		\lesssim
		A(\mathbf P)A(\mathbf Q).
		$$
	\end{lm}

	\begin{lm}\label{lowbound}
		Let $1<\alpha<2$ and $d \ge 2$. Then for any $\xi,\sigma \in \mathbb{R}^{d}$ we have
		$$
		\big\vert
		\frac{\xi}{\vert\xi\vert^{2-\alpha}}
		-\frac{\sigma}{\vert\sigma\vert^{2-\alpha}}
		\big\vert
		\gtrsim
		\min\Big(\vert\sigma\vert^{\alpha-1},
		\frac{\vert\xi-\sigma\vert}
		{\vert\sigma\vert^{2-\alpha}}\Big).
		$$
	\end{lm}
	\begin{proof}[Proof of Lemma \ref{lowbound}]
		Set
		$$
		\mathbf z := \frac{\xi}{|\xi|^{2-\alpha}} - \frac{\sigma}{|\sigma|^{2-\alpha}}
		$$
		and let $\xi_\tau = \tau \xi + (1-\tau)\sigma$ for $\tau \in [0, 1]$. If $\xi_\tau = 0$ for some $\tau = \tau_0$, then since $\xi \neq \sigma$, $\xi \neq 0$ and $\sigma \neq 0$, $\tau_0 = |\sigma|/|\xi-\sigma| \in (0, 1)$ and $\xi = - [(1-\tau_0)/\tau_0] \sigma$. Substituting this into $z$, we have
		$$
		\mathbf z = -\al\left(1+ \big(\frac{1-\tau_0}{\tau_0}\big)^{\al-1}\right)\frac{\sigma}{|\sigma|^{2-\al}},
		$$
		which means $|z| \gtrsim |\sigma|^{\al-1}$. Note that $1 > \tau_0 = |\sigma|/|\xi-\sigma|$.
		
		From now we assume that $\xi_\tau \neq 0$ for all $\tau \in [0, 1]$. Then
		\begin{align}\label{z}
		\mathbf z = \int_0^1 \partial_\tau \left[\frac{\xi_\tau}{|\xi_\tau|^{2-\al}}\right] \,d\tau = \int_0^1 |\xil|^{-(2-\al)}\mathbf M_{\xi_\tau}(\xi-\sigma)\,d\tau,
		\end{align}
		where $\mathbf M_{\xi_\tau}$ is the $d \times d$ symmetric matrix $$\mathbf M_{\xi_\tau} = \mathbf I - (2-\al)\frac{\xil \otimes \xil}{|\xi_\tau|^2}.$$
		By Cauchy-Schwarz inequality we have that for any $x \in \mathbb R^d$
		\begin{align*}
		x \cdot \mathbf M_{\xil}x &= |x|^2 - (2-\al)\frac{(x\cdot\xil)^2}{|\xil|^2} \ge (\al-1)|x|^2.
		\end{align*}
		Thus the eigenvalues of $\mathbf M_{\xil}$ is at least $(\al-1)$, which implies that
		\begin{align*}
		|\mathbf M_{\xil}x| \ge (\al-1)|x|.
		\end{align*}
		Using this fact, we get that
		$$
		|\mathbf z| \ge \int_0^1|\xil|^{-(2-\al)}\,d\tau |\xi-\sigma| \ge \int_0^1 |1 + \tau\frac{|\xi-\sigma|}{|\sigma|}|^{-(2-\al)}\,d\tau \frac{|\xi-\sigma|}{|\sigma|^{2-\al}}.
		$$
		If $\frac{|\xi-\sigma|}{|\sigma|} \le 1$, then we have
		\begin{align*}
		|\mathbf z| \gtrsim \frac{|\xi-\sigma|}{|\sigma|^{2-\al}}.
		\end{align*}
		If $\frac{|\xi-\sigma|}{|\sigma|} > 1$, then
		\begin{align*}
		|\mathbf z| \gtrsim \int_0^\frac{|\sigma|}{|\xi-\sigma|}\Big(1+ \tau\frac{|\xi-\sigma|}{|\sigma|}\Big)^{-(2-\al)}\,d\tau \frac{|\xi-\sigma|}{|\sigma|^{2-\al}} \gtrsim  |\sigma|^{\al-1}.
		\end{align*}
		Therefore we estimate
		$$
		|\mathbf z| \gtrsim (\al-1)\min(|\sigma|^{\al-1}, |\xi-\sigma|/|\sigma|^{2-\al}).
		$$
	\end{proof}
	
	\begin{lm}\label{diff-phase}
		Let $\phi, \widetilde\phi$ be as in \eqref{phase}, \eqref{tildephase}, respectively. Then we have
		\begin{align*}
		\vert \phi(\xi,\eta,\sigma)-\widetilde\phi(\xi,\eta,\sigma)\vert
		&\lesssim \vert\eta\vert^{\alpha}.
		\end{align*}
	\end{lm}
	\begin{proof}[Proof of Lemma \ref{diff-phase}]
		\begin{align*}
		&|\phi(\xi, \eta, \sigma) - \widetilde \phi(\xi, \eta, \sigma)| \\
		&\lesssim \big||\xi+\eta|^\al - |\xi|^\al - \al\frac{(\xi\cdot\eta)}{|\xi|^{2-\al}}\big| + \big||\xi+\sigma+\eta|^\al - |\xi+\sigma|^\al - \al\frac{((\xi+\sigma)\cdot\eta)}{|\xi+\sigma|^{2-\al}}\big|.
		\end{align*}
		We only consider the first term. If $|\xi| \lesssim |\eta|$, then a direct calculation gives us
		$$
		\big||\xi+\eta|^\al - |\xi|^\al - \al\frac{(\xi\cdot\eta)}{|\xi|^{2-\al}}\big| \lesssim |\eta|^\al.
		$$
		If $|\xi| \gg |\eta|$, then
		MVP gives us
		$$
		|\xi+\eta|^\al - |\xi|^\al = \int_0^1 \al\frac{\xi_\sigma}{|\xi_\sigma|^{2-\al}} \cdot \eta\,d\sigma,
		$$
		where $\xi_\sigma = \xi + \sigma\eta$. By another MVP as in \eqref{z} we get
		$$
		|\xi+\eta|^\al - |\xi|^\al - \al\frac{(\xi\cdot\eta)}{|\xi|^{2-\al}} = \al \iint_{[0,1]^2} |\xi_{\sigma\tau}|^{\al-2}[\mathbf M_{\xi_{\sigma\tau}}(\sigma\tau\eta)]\cdot \eta\,d\sigma
		d\tau,
		$$
		where $\xi_{\sigma\tau} = \xi + \sigma\tau\eta$. Since $|\xi| \gg |\eta|$, $|\xi_{\sigma\tau}|\sim |\xi|$ and hence $|\xi_{\sigma\tau}|^{\al-2}\lesssim |\eta|^{\al-2}$.
		On the other hand, $|[\mathbf M_{\xi_{\sigma\tau}} \eta]\cdot \eta | \lesssim |\eta|^2$. Therefore
		$$
		||\xi+\eta|^\al - |\xi|^\al| \lesssim |\eta|^\al.
		$$
	\end{proof}

	\bibliographystyle{plain}

\end{document}